\documentclass[leqno,11pt]{amsart}

\usepackage[letterpaper,margin=1in]{geometry}

\usepackage[all]{xy} % diagrams
\usepackage{amsmath, amssymb, amsfonts, latexsym, mdwlist, amsthm, amscd, braket}
\usepackage{subfig}
\usepackage{graphicx}
\usepackage{wrapfig}

\usepackage[bookmarks, colorlinks, breaklinks, pdftitle={},
pdfauthor={}]{hyperref}
\hypersetup{linkcolor=blue,citecolor=blue,filecolor=black,urlcolor=blue}

\usepackage{tikz}
\usetikzlibrary{calc,trees,positioning,arrows,chains,shapes.geometric,%
    decorations.pathreplacing,decorations.pathmorphing,shapes,%
    matrix,shapes.symbols}

\tikzset{
>=stealth',
  punktchain/.style={
    rectangle,
    rounded corners,
    draw=black, thick,
    minimum height=3em,
    text centered,
    on chain},
  line/.style={draw, thick, <-},
  element/.style={
    tape,
    top color=white,
    bottom color=blue!50!black!60!,
    minimum width=8em,
    draw=blue!40!black!90, very thick,
    text width=10em,
    minimum height=3.5em,
    text centered,
    on chain},
  every join/.style={->, thick,shorten >=1pt},
  decoration={brace},
  tuborg/.style={decorate},
  tubnode/.style={midway, right=2pt},
}

\usepackage{paralist}
\usepackage{enumitem} 
\setdefaultenum{(1)}{(a)}{(i)}{}
\setlist[enumerate,1]{label={\upshape(\arabic*)}}
\setlist[enumerate,2]{label={\upshape(\alph*)},ref=\theenumi\upshape(\alph*)}
\setlist[enumerate,3]{label={\upshape(\roman*)},ref=\theenumi\theenumii\upshape(\roman*)}
\usepackage[capitalize]{cleveref}
\crefname{Prop}{Proposition}{Propositions}
\crefname{Thm}{Theorem}{Theorems}
\crefname{Lem}{Lemma}{Lemmas}
\crefname{enumi}{Case}{Cases}

\def\C{\ensuremath{\mathbb{C}}}

\def\N{\ensuremath{\mathbb{N}}}
\def\P{\ensuremath{\mathbb{P}}}
\def\Q{\ensuremath{\mathbb{Q}}}
\def\R{\ensuremath{\mathbb{R}}}

\def\Z{\ensuremath{\mathbb{Z}}}

\def\wGL2{\ensuremath{\widetilde{\mathrm{GL}}_2^+(\R)}}
\def\alg{\mathrm{alg}}
\def\Amp{\mathrm{Amp}}
\def\Aut{\mathop{\mathrm{Aut}}\nolimits}

\def\ch{\mathop{\mathrm{ch}}\nolimits}

\def\Coh{\mathop{\mathrm{Coh}}\nolimits}
\def\codim{\mathop{\mathrm{codim}}\nolimits}

\def\dim{\mathop{\mathrm{dim}}\nolimits}

\def\Ext{\mathop{\mathrm{Ext}}\nolimits}
\def\ext{\mathop{\mathrm{ext}}\nolimits}
\def\Fix{\mathop{\mathrm{Fix}}}

\def\FixG{\mathop{\mathrm{Fix}(G')}}
\def\For{\mathop{\mathrm{Forg}_{G'}}}

\def\gr{\mathop{\mathrm{gr}}\nolimits}
\def\GL{\mathop{\mathrm{GL}}}
\def\Hal{H^*_{\alg}}
\def\Hilb{\mathop{\mathrm{Hilb}}\nolimits}
\def\HN{\mathop{\mathrm{HN}}\nolimits}
\def\Hom{\mathop{\mathrm{Hom}}\nolimits}

\def\Id{\mathop{\mathrm{Id}}\nolimits}
\def\im{\mathop{\mathrm{im}}\nolimits}

\def\Inf{\mathop{\mathrm{Inf}_{G'}}}

\def\JH{\mathop{\mathrm{JH}}\nolimits}
\def\Ker{\mathop{\mathrm{Ker}}\nolimits}

\def\lcm{\mathop{\mathrm{lcm}}\nolimits}

\def\min{\mathop{\mathrm{min}}\nolimits}

\def\num{\mathop{\mathrm{num}}\nolimits}
\def\Num{\mathop{\mathrm{Num}}\nolimits}
\def\NS{\mathop{\mathrm{NS}}\nolimits}
\def\ord{\mathop{\mathrm{ord}}\nolimits}

\def\Pic{\mathop{\mathrm{Pic}}\nolimits}

\def\Quot{\mathop{\mathrm{Quot}}\nolimits}
\def\rk{\mathop{\mathrm{rk}}}

\def\Sing{\mathop{\mathrm{Sing}}}

\def\SL{\mathop{\mathrm{SL}}\nolimits}

\def\stab{\mathop{\mathrm{stab}}}

\def\Sym{\mathop{\mathrm{Sym}}\nolimits}

\def\Tr{\mathop{\mathrm{Tr}}\nolimits}
\def\td{\mathop{\mathrm{td}}\nolimits}

\def\v{\mathbf{v}}
\def\u{\mathbf{u}}
\def\w{\mathbf{w}}

\def\a{\mathbf{a}}
\def\b{\mathbf{b}}

\def\o{\mathop{\ord(K_S)}\nolimits}

\def\MG13{\ensuremath{{\mathcal M}_{\Gamma_1(3)}}}
\def\tildeMG13{\ensuremath{\widetilde{\mathcal M}_{\Gamma_1(3)}}}
\def\Stab{\mathop{\mathrm{Stab}}}
\def\Stabd{\mathop{\Stab^{\dagger}}}
\def\into{\ensuremath{\hookrightarrow}}
\def\onto{\ensuremath{\twoheadrightarrow}}

\def\blank{\underline{\hphantom{A}}}

\def\Db{\mathrm{D}^{\mathrm{b}}}

\def\pt{[\mathrm{pt}]}

\makeatletter
\newtheorem*{rep@theorem}{\rep@title}
\newcommand{\newreptheorem}[2]{%
\newenvironment{rep#1}[1]{%
 \def\rep@title{#2 \ref{##1}}%
 \begin{rep@theorem}}%
 {\end{rep@theorem}}}
\makeatother

\newtheorem{Thm}{Theorem}[section]
\newreptheorem{Thm}{Theorem}
\newtheorem{Prop}[Thm]{Proposition}
\newtheorem{PropDef}[Thm]{Proposition and Definition}
\newtheorem{Lem}[Thm]{Lemma}

\newreptheorem{Cor}{Corollary}

\newreptheorem{Con}{Conjecture}

\newtheorem{thm-int}{Theorem}

\theoremstyle{definition}
\newtheorem{Def-s}[Thm]{Definition}
\newtheorem{Def}[Thm]{Definition}
\newtheorem{Rem}[Thm]{Remark}

\def\C{\ensuremath{\mathbb{C}}}

\def\N{\ensuremath{\mathbb{N}}}
\def\P{\ensuremath{\mathbb{P}}}
\def\Q{\ensuremath{\mathbb{Q}}}
\def\R{\ensuremath{\mathbb{R}}}
\def\Z{\ensuremath{\mathbb{Z}}}

\def\AA{\ensuremath{\mathcal A}}
\def\BB{\ensuremath{\mathcal B}}
\def\CC{\ensuremath{\mathcal C}}

\def\EE{\ensuremath{\mathcal E}}
\def\FF{\ensuremath{\mathcal F}}
\def\GG{\ensuremath{\mathcal G}}
\def\HH{\ensuremath{\mathcal H}}

\def\JJ{\ensuremath{\mathcal J}}

\def\MM{\ensuremath{\mathcal M}}
\def\NN{\ensuremath{\mathcal N}}
\def\OO{\ensuremath{\mathcal O}}
\def\PP{\ensuremath{\mathcal P}}

\def\QQ{\ensuremath{\mathcal Q}}
\def\TT{\ensuremath{\mathcal T}}

\def\ZZ{\ensuremath{\mathcal Z}}

\newcommand{\mor}[1][]{\xrightarrow{#1}}
\newcommand{\isomor}{\mor[\sim]}

\def\MMM{\mathfrak M}

\def\mX{\ensuremath{M_{\sigma',X}(\pi^*\v)}}
\def\mXs{\ensuremath{M^s_{\sigma',X}(\pi^*\v)}}

\def\mS{\ensuremath{M_{\sigma,S}(\v)}}
\def\mSs{\ensuremath{M^s_{\sigma,S}(\v)}}

\newcommand{\todo}[1]{\vspace{5 mm}\par \noindent
\marginpar{\textsc{ToDo}}
\framebox{\begin{minipage}[c]{0.95 \textwidth}
\tt #1 \end{minipage}}\vspace{5 mm}\par}

\newcommand{\ignore}[1]{}

\begin{document}
\author{Howard Nuer}
\title[Stable sheaves on bielliptic surfaces]{Stable sheaves on bielliptic surfaces: from the classical to the modern}
\begin{abstract}
    Bielliptic surfaces are the last family of Kodaira dimension zero algebraic surfaces without a classification result for the Chern characters of stable sheaves.  
    We rectify this and prove such a classification using a combination of classical techniques, on the one hand, and derived category and Bridgeland stability techniques, on the other.  
    Along the way, we prove the existence of projective coarse moduli spaces of objects in the derived category of a bielliptic surface that Bridgeland semistable with respect to a generic stability condition.
    By systematically studying the connection between Bridgeland wall-crossing and birational geometry, we show that for any two generic stability conditions $\tau,\sigma$, the two moduli spaces $M_\tau(\v)$ and $M_\sigma(\v)$ of objects of Chern character $\v$ that are semistable with respect to $\tau$ (resp. $\sigma$) are birational.  
    As a consequence, we show that for primitive $\v$, the moduli space of stable sheaves of class $\v$ is birational to a moduli space of stable sheaves whose Chern character has one of finitely many easily understood ``shapes".
\end{abstract}
\maketitle
\setcounter{tocdepth}{1}
\tableofcontents
\section{Introduction}

Moduli spaces of stable sheaves have attracted great interest from mathematicians and physicists alike for almost forty years now. 
They have been studied using vastly different mathematical disciplines. 
More recently, the development of Bridgeland stability has provided a unifying framework to study moduli spaces of sheaves using tools such as Fourier-Mukai transforms, enumerative and motivic invariants, and the minimal model program.  
The set of geometric Bridgeland stability conditions, $\Stabd(X)$, is a complex manifold admitting a wall-and-chamber for any given Chern character, and there is a strong connection between wall-crossing in $\Stabd(X)$ and both birational transformations on the moduli spaces and wall-crossing formulae for enumerative invariants.

One family of surfaces whose stable sheaves have received little attention are the bielliptic surfaces.  
Indeed, the only previous works dealing specifically with stable sheaves on bielliptic surfaces are \cite{Tak73,Ume75}, which only deal with some of the disriminant zero cases.  
One can also apply \cite[Theorem 1.1]{Bri02} to study certain positive discriminant cases on a few of the families of bielliptic curves.  
In this paper, we apply both classical methods and derived category techniques to classify the Mukai vectors (i.e. Chern characters) of slope-stable sheaves (see \cref{MainThm1:Non-emptinessOfGieseker} for a precise statement).

On the way to achieving this classification, we prove fundamental results about Bridgeland stable objects on bielliptic surfaces and the moduli spaces in which they deform.  
In particular, we prove the existence of projective coarse moduli spaces of Bridgeland semistable objects with respect to a generic Bridgeland stability condition.  
Moreover, we study the singularities and Kodaira dimension of the stable locus of these moduli spaces.  

We continue our investigation of these moduli spaces by studying their birational geometry via wall-crossing.  
More specifically, given a Mukai vector $\v$, a wall $W\subset\Stabd(X)$ for $\v$, and two stability conditions $\sigma_\pm$ in the opposite and adjacent chambers separated by $W$, we seek a precise answer to the question: how are $M_{\sigma_+}(\v)$ and $M_{\sigma_-}(\v)$ related?  
The basic answer, given by \cref{MainThm2:BirationalBridgeland}, is that $M_{\sigma_+}(\v)$ and $M_{\sigma_-}(\v)$ are birational.  

This question, and others associated with it, are motivated by a larger trend in moduli theory.
This program, known as the Hassett-Keel program, attempts to view minimal models of a given moduli space as moduli spaces in and of themselves, just of slightly different objects.
The relation between wall-crossing and the minimal model program studied here has been explored for many other surfaces: K3 surfaces \cite{BM14a,BM14b,MYY14b,MYY14}; Enriques surfaces \cite{Nue14b,NY19}; $\P^2$ \cite{ABCH13,BMW14,CH15,CHW14,LZ13,LZ19}; Hirzebruch and del Pezzo surfaces \cite{BC13,AM17}; and abelian surfaces \cite{MM13,YY14,Yos12}.
Our investigations here are a big step toward completing such a program in the case of bielliptic surfaces.

It has become even more apparent with these and similar results for other surfaces that Bridgeland stability conditions are a crucial tool not only for studying the birational geometry of certain moduli spaces, but also more intrinsic and classical questions about the objects they classify.  

\subsection*{Summary of Results and Techniques}
Let us turn now to stating our main results more precisely along with the tools we use to prove them.  
While we briefly introduce notation, the reader is invited to see \cref{sec:BiellipticSurfaces} for more details.

A smooth projective surface $S$ over $\C$ is said to be bielliptic if its Kodaira dimension $\kappa(S)=0$, $q(S)=H^1(S,\OO_S)=\C$ and $p_g(S)=H^0(S,\OO(K_S))=0$.  
It is known that bielliptic surfaces come in seven families, $\o$ is one of four (finite) values, and the canonical cover $\pi\colon X\to S$ is an abelian surface.  
The name bielliptic comes from the existence of two elliptic fibrations on $S$, one with base an elliptic curve and the other with $\P^1$ as a base.  
It is known that $\Num(S)=\Z[A_0]\oplus\Z[B_0]$, where $A_0^2=B_0^2=0$ and $A_0.B_0=1$.  
Both $A_0$ and $B_0$ are (reductions of) a (multiple-)fiber of one of the two elliptic fibrations on $S$.

The topological invariants of a coherent sheaf (or an object $E\in\Db(S)$ in the bounded derived category of coherent sheaves) are encoded in its Mukai vector $$\v(E):=\ch(E)\sqrt{\td(S)}=\ch(E)\in H^*(S,\Q).$$
We consider the Mukai lattice $$\Hal(S,\Z):=\v(K(S)),$$ where $K(S)$ is the Grothendieck group, along with the induced pairing $\langle\v(E),\v(F)\rangle=-\chi(E,F)$.
For any primitive $\v\in\Hal(S,\Z)$, if we write $\pi^*\v=l(\v)\w\in\Hal(X,\Z)$ with $\w$ primitive, then $l(\v)\mid\o$ (see \cref{primitive}).   
The Bogomolov inequality says that $\v^2\geq 0$ is a necessary condition for the existence of a $\mu$-semistable sheaf. 
Our first main theorem not only proves the converse to this statement, but also gives a precise classification of the Mukai vectors of $\mu$-stable (locally free) sheaves.
\begin{Thm}[\cref{Thm:Non-emptinessOfGieseker}]\label{MainThm1:Non-emptinessOfGieseker}
Let $S$ be a bielliptic surface and $\omega$ a generic polarization with respect to a primitive Mukai vector $\v$ of positive rank.  Let $n\in\N$.
\begin{enumerate}
    \item If $\v^2=0$, then $M^{\mu ss}_\omega(n\v)\ne\varnothing$ with $M^{\mu s}_\omega(n\v)\ne\varnothing$ iff $nl(\v)\mid\o$.
    \item If $\v^2>0$, then $M^{\mu ss}_\omega(n\v)\neq\varnothing$ and there is a $\mu$-stable sheaf in each irreducible component of $M^{\mu ss}_\omega(n\v)$.  Moreover, except when $\o=2$ and $\v=(2,0,-1)e^D$ or $S$ is of Type 1 and $\v=(2,B_0,-1)e^D$, this $\mu$-stable sheaf may be taken to be locally free as long as $\rk(n\v)>1$.  Even in these exceptional cases, there are components which contain $\mu$-stable locally free sheaves.
\end{enumerate}
\end{Thm}

In \cref{Part:ClassicalStability}, we apply the classical stability theory toolbox to significantly reduce the scope of the analysis  required for the proof of \cref{MainThm1:Non-emptinessOfGieseker}.
To reduce the problem further, we bring in a more general notion of stability in \cref{Part:BridgelandStability}: Bridgeland stability conditions on the bounded derived category $\Db(S)$ of coherent sheaves on $S$.

The stability conditions that we consider in \cref{Part:BridgelandStability} are all contained in the distinguished connected component $\Stabd(S)$ containing geometric stability conditions, those $\sigma$ such that skyscraper sheaves of points are $\sigma$-stable.  
For a fixed $\v\in\Hal(S,\Z)$, we call $\sigma\in\Stabd(S)$ \emph{generic} if $\sigma$ is in some chamber of the wall-and-chamber decomposition of $\Stabd(S)$ with respect to $\v$.  
By applying an idea going back to \cite{BM14a,MYY14}, we use autoequivalences (Fourier-Mukai functors) to reduce our moduli functors to those of classical stability, giving the first claim in following theorem.
\begin{Thm}[\cref{thm:projective coarse moduli spaces,local singularities}]\label{MainThm3:projective coarse moduli spaces}
Let $S$ be a bielliptic surface and $\v\in\Hal(S,\Z)$ with $\v^2\geq 0$.  Assume that $\sigma\in\Stabd(S)$ is generic with respect to $\v$.  Then there exists a nonempty projective variety $M_{\sigma,S}(\v)$ that is a coarse moduli space parametrizing $S$-equivalence classes of $\sigma$-semistable objects with Mukai vector $\v$.  If $\v$ is primitive and $\v^2\geq 3\o$, then $K_{\mS}$ is a torsion Cartier divisor and $\mS$ is normal and Gorenstein, with at worst terminal l.c.i singularities.
\end{Thm}

Once we have the existence of coarse moduli spaces, we are ready to prove \cref{MainThm1:Non-emptinessOfGieseker} using motivic invariants that are invariant under autoequivalences and wall-crossing.
By applying Fourier-Mukai functors associated to the two elliptic fibrations on $S$, we pick up where classical stability left off and reduce the problem to a finite list of easily checked cases.

Although a detailed study of the singularities of these moduli spaces is unnecessary for the purpose of proving the existence of Gieseker semistable sheaves (except in the isotropic case when it is necessary), we felt it was nevertheless important for the sake of completeness to study these singularities in general. 
To keep the paper moving, however, we relegate to an appendix, our analysis of the outlying moduli spaces of small dimension not covered by the hypothesis in the second statement of \cref{MainThm3:projective coarse moduli spaces} (i.e. when $\v^2<3\o$).

In \cref{Part:WallCrossing}, we refine this wall-crossing analysis by showing that the invariance of these motivic invariants is due to the existence of a birational equivalence between the moduli spaces on opposite sides of any wall.  More generally, we prove the following result.
\begin{Thm}[\cref{Thm:MainTheorem1}]\label{MainThm2:BirationalBridgeland}
Let $\v\in\Hal(S,\Z)$ satisfy $\v^2\geq 0$, and let $\sigma,\tau\in\Stabd(S)$ be generic stability conditions with respect to $\v$ (that is, they are not contained in any wall for $\v$).
\begin{enumerate}
\item  The two moduli spaces $M_{\sigma,S}(\v)$ and $M_{\tau,S}(\v)$ of Bridgeland semistable objects are birational to each other.
\item More precisely, there is a birational map induced by a derived (anti-)autoequivalence $\Phi$ of $\Db(S)$ in the following sense: there exists a common open subset $U\subset M_{\sigma,S}(\v),U\subset M_{\tau,S}(\v)$ such that for any $u\in U$, the corresponding objects $E_u\in M_{\sigma,S}(\v)$ and $F_u\in M_{\tau,S}(\v)$ satisfy $F_u=\Phi(E_u)$.  
\end{enumerate}

\end{Thm}
The proof of the analogous statement to \cref{MainThm2:BirationalBridgeland} for K3 surfaces \cite[Theorem 1.1]{BM14b} relies heavily on the fact that $M_\sigma(\v)$ is a projective hyperk\"{a}hler manifold in the K3 case.  
Our proof uses dimension estimates of stacks of Harder-Narasimhan filtrations that we developed with Yoshioka in \cite{NY19}. 
These estimates allow us to determine what objects in $M_{\sigma_+}(\v)$ are destabilized when crossing the wall $W$, and most importantly the dimension of this locus, purely in terms of a rank two hyperpolic lattice $\HH_W$ associated with $W$.  
This is the content of \cref{classification of walls} which gives a much more refined and detailed classification of the type of birational map in \cref{MainThm2:BirationalBridgeland} in terms of the presence of special isotropic classes in $\HH_W$.

The next result is the prototypical application of a result like \cref{MainThm2:BirationalBridgeland}.  
To study a given moduli space $M_{\sigma,S}(\v)$, we apply a Fourier-Mukai transform $\Phi:\Db(S)\to\Db(S)$ inducing an isomorphism $M_{\sigma,S}(\v)\isomor M_{\Phi(\sigma),S}(\Phi_*(\v))$, where $\Phi_*(\v)$ has a more familiar form, say that of a well-understand type of coherent sheaf. 
Using \cref{MainThm2:BirationalBridgeland}, we know that there is a birational map $M_{\Phi(\sigma),S}(\Phi_*(\v))\dashrightarrow M_{\tau,S}(\Phi_*(\v))$, where $\tau$ is now in the so-called Gieseker chamber so that $M_{\tau,S}(\Phi_*(\v))$ is simply the moduli space of stable sheaves of Mukai vector $\Phi_*(\v)$. 
This argument gives the next result, which generalizes \cite[Theorem 1.1]{Bri98} to arbitrary Mukai vectors.  Indeed, \cite[Theorem 1.1]{Bri98} corresponds to the case $q=1$ and $\v_0=(q,0,s)$ in \cref{MainThm4:FM reduction}.

\begin{Thm}\label{MainThm4:FM reduction} Suppose that $\omega$ is a generic polarization for a primitive Mukai vector $\v=(r,aA_0+bB_0,t)\in\Hal(S,\Z)$ with $\v^2\geq 0$.  Then there is a birational map $$M_\omega(\v)\dashrightarrow M_\omega(\v_0),$$ where $\v_0$ is one of the reduced forms appearing in  \cref{table:IntroFiniteList} and $1\leq\rk(\v_0)\leq r$.  
\end{Thm}

\begin{table}[ht]
\caption{The reduced forms of primitive Mukai vectors}
\begin{tabular}{|c|  c| c| c| c|}
\hline
Type &  $\o$& $\v_0$\\
\hline
1 &  $2$ & $(q,0,s),(2q,qB_0,s)$\\
\hline
2 &  $2$ & $(q,0,s),(2q,qA_0,s),(2q,qB_0,s),(2q,qA_0+qB_0,s)$\\
\hline
3  & $4$ & $(q,0,s),(4q,qB_0,s),(2q,qB_0,s)$\\
\hline
4  & $4$ & $(q,0,s),(2q,qA_0,s),(4q,qB_0,s),(2q,qB_0,s),$\\
        &        &$(4q,2qA_0+qB_0,s),(2q,qA_0+qB_0,s)$\\
\hline
5 & $3$ & $(q,0,s),(3q,qB_0,s)$\\
\hline
6 & $3$ & $(q,0,s),(3q,qA_0,s),(3q,qB_0,s),$\\
         &        &$(3q,qA_0+qB_0,s)$\\
   \hline
7 & $6$ & $(q,0,s),(6q,qB_0,s),(3q,qB_0,s),(2q,qB_0,s),(q,bB_0,s), \frac{1}{3}<\frac{b}{q}<\frac{1}{2}$\\
\hline
\end{tabular}
\label{table:IntroFiniteList}
\end{table}

\subsection*{Open Questions and Further Directions}
In this paper, we have thoroughly explored moduli spaces of stable sheaves on a bielliptic surface from an algebro-geometric point of view (e.g. singularities, Kodaira dimension, birational geometry).  
Nevertheless, a main theme left unexplored in this manuscript is the topology of these moduli spaces.  
Certain basic questions remain unresolved.
For example, how many irreducible components does $\mS$ have?  
What is $\pi_1(\mS)$?  
What is the order of $K_{\mS}$?  
What are $b_i(\mS)$?
We have used Fourier-Mukai transforms and Bridgeland stability to answer questions such as these for Enriques surfaces \cite{NY19}, and we believe similar tools can be fruitfully applied for bielliptic surfaces as well. 
It would also be interesting to see if the results variations of Hodge structure techniques of \cite{Sac} could be generalized to answer some of these questions for bielliptic surfaces.

In addition, we have not explored the Albanese map for $\mS$.  
As $\Pic^0(S)$ is far from trivial, neither is the Albanese map for $\mS$, and on the abelian surface cover, the fiber of the Albanese map has proved to be an important variety to study.  
Indeed, in the abelian surface case, this fiber is a hyperk\"{a}hler manifold.  
It shown in \cite{Yos12} that all minimal models of this fiber are obtained by Bridgeland wall crossing, as conjectured by the Hassett-Keel program.  
It is natural to wonder if the corresponding Albanese map fiber might be equally amenable to such a result in the bielliptic case.

\subsection*{Acknowledgements} 
The author would like to thank Arend Bayer, Izzet Coskun, Antony Maciocia, Emanuele Macr\`{i}, Benjamin Schmidt, K\={o}ta Yoshioka, and Xiaolei Zhao for very helpful discussions related to this article.  

The seed for this project was planted by discussions with Antony Maciocia during the workshop on ``Geometry from Stability Conditions" held at the University of Warwick in February of 2015.  The author would like to thank the university for its warm hospitality and the organizers for arranging such a stimulating environment.  The author was partially supported by the NSF postdoctoral fellowship DMS-1606283, by the NSF RTG grant DMS-1246844, and by the NSF FRG grant DMS-1664215.

\subsection*{Notation}
For a complex number $z\in\C$ we denote its real and imaginary parts by $\Re z$ and $\Im z$, respectively.

We will denote by $\Db(X)$ the bounded derived category of coherent sheaves on a smooth projective variety $X$.  
We will use non-script letters ($E,F,G,\dots$) for objects on a fixed scheme and reserve curly letters ($\EE,\FF,\GG,\dots$) for families of such objects.

For a vector $\v$ in a lattice $\HH$ with pairing $\langle\blank,\blank\rangle$, we abuse notation and write $$\v^2:=\langle\v,\v\rangle.$$
The intersection pairing on a smooth surface $X$ will be denoted by $\blank.\blank$ and the self-intersection of a divisor $D$ by $D^2$.  

\section{Bielliptic surfaces}\label{sec:BiellipticSurfaces}
\subsection{The seven families}
A smooth projective surface $S$ over $\C$ is said to be bielliptic if its Kodaira dimension $\kappa(S)=0$, $q(S)=1$ and $p_g(S)=0$.  It is known that any bielliptic surface is the \'{e}tale quotient $\bar{\pi}:A\times B\to S$ of the product of two elliptic curves $A\times B$ by the action of a finite abelian group $G$ acting on A by translations (so that $A/G$ is still elliptic) and on $B$ in a such a way that $B/G\cong\P^1$.  It follows that the canonical bundle must be torsion with $\o=2,3,4,6$ and that $S$ falls into one of seven families \cite[V.5]{BHPV}.  To simplify notation we let $B\cong\C/(\Z\oplus\Z\omega)$, where $\omega\in\C$ is uniquely determined if it is chosen in the fundamental domain $-\frac{1}{2}\leq\Re\omega<\frac{1}{2},\Im\omega>0,|\omega|\geq 1$ if $\Re\omega\leq 0$ and $|\omega|>1$ if $\Re\omega>0$.  The seven families are described in Table \ref{table:families}.

\begin{table}[!htbp]
\caption{The seven families of bielliptic surfaces}
\begin{tabular}{c c c c c c}
\hline
Type & $\omega$ & $G$ & Action of $G$ on $B$\\
\hline
1 & any & $\Z/2\Z$ & $x\mapsto -x$\\
2 & any & $\Z/2\Z\times\Z/2\Z$ & $x\mapsto -x,x\mapsto x+\epsilon$ with $2\epsilon=0$\\
3 & $i$ & $\Z/4\Z$ & $x\mapsto ix$\\
4 & $i$ & $\Z/4\Z\times\Z/2\Z$ & $x\mapsto ix,x\mapsto x+\frac{1+i}{2}$\\
5 & $e^{\frac{2\pi i}{3}}$ & $\Z/3\Z$ & $x\mapsto e^{\frac{2\pi i}{3}} x$\\
6 & $e^{\frac{2\pi i}{3}}$ & $\Z/3\Z\times\Z/3\Z$ & $x\mapsto e^{\frac{2\pi i}{3}}x,x\mapsto x+\frac{1-e^{\frac{2\pi i}{3}}}{3}$\\
7 & $e^{\frac{2\pi i}{3}}$ & $\Z/6\Z$ & $x\mapsto -e^{\frac{2\pi i}{3}}x$\\

\hline
\end{tabular}
\label{table:families}
\end{table}

The two projections $A\times B\to A$ and $A\times B\to B$ are $G$-equivariant and descend to two fibrations $p_A:S\mor A/G,p_B:S\mor B/G$.  Since $A\mor A/G$ is \'{e}tale, all fibres of $p_A$ are smooth, and a fibre of $p_B$ over a point $[x]\in B/G$ is a multiple of a smooth elliptic curve whose multiplicity is that of the point $[x]$ for the finite map $B\to B/G$.  All smooth fibres of $p_A$ (resp. $p_B$) are isomorphic to $B$ (resp. $A$), so we denote the classes of these fibres by $B$ and $A$, respectively, in $\Num(S)$, $H^2(S,\Z)$ and $H^2(S,\Q)$.  It is well known \cite{Ser90} that $A$ and $B$ span $H^2(S,\Q)$ and satisfy $A^2=B^2=0,AB=|G|$.  We recall some other facts that will be helpful for us:

\begin{Lem}\label{divisor facts}\cite[Lemma 1.3]{Ser90} Let $D\in\Num(S)$ have class $\alpha A+\beta B$ with $\alpha,\beta\in\Q$.  Then:

\begin{enumerate}
\item $\chi(\OO_S(D))=\frac{1}{2}D^2=\alpha\beta|G|$.
\item $D$ is ample if and only if $\alpha>0,\beta>0$.
\item If $D$ is ample then $h^0(\OO_S(D))=\chi(\OO_S(D))$.
\item If $D\geq 0$, then $\alpha\geq 0,\beta\geq 0$.
\end{enumerate}
\end{Lem}

\begin{table}[ht]
\caption{Topological invariants of the seven families}
\begin{tabular}{c c c c c c}
\hline
Type & $(m_1,\dots,m_s)$ & $H^2(S,\Z)$ & $\o$\\
\hline
1 & $(2,2,2,2)$ & $\Z[\frac{1}{2}A]\oplus \Z[B]\oplus \Z/2\Z\oplus\Z/2\Z$ & 2\\
2 & $(2,2,2,2)$ & $\Z[\frac{1}{2}A]\oplus\Z[\frac{1}{2}B]\oplus\Z/2\Z$ & 2\\
3 & $(2,4,4)$ & $\Z[\frac{1}{4}A]\oplus\Z[B]\oplus\Z/2\Z$ & 4\\
4 & $(2,4,4)$ & $\Z[\frac{1}{4}A]\oplus\Z[\frac{1}{2}B]$ & 4\\
5 & $(3,3,3)$ & $\Z[\frac{1}{3}A]\oplus\Z[B]\oplus\Z/3\Z$ & 3\\
6 & $(3,3,3)$ & $\Z[\frac{1}{3}A]\oplus\Z[\frac{1}{3}B]$ & 3\\
7 & $(2,3,6)$ & $\Z[\frac{1}{6}A]\oplus\Z[B]$ & 6\\
\hline
\end{tabular}
\label{table:cohomology}
\end{table}

\begin{Thm}\label{Pic}\cite[Theorem 1.4, Remark 1.6, Proposition 1.7]{Ser90} Table \ref{table:cohomology} gives the multiplicities $(m_1,\dots,m_s)$ of the singular fibers of $p_B:S\mor \P^1$, the decomposition of $H^2(S,\Z)$, and the order of the canonical bundle $\o$ for each of the seven families of bielliptic surfaces.  At the level of numerical equivalence, we have simply $$\Num(S)=\Z[\frac{1}{\o}A]\oplus\Z[\frac{\o}{|G|}B].$$
\end{Thm}

As they are useful invariants, we define $\lambda_S:=\frac{|G|}{\o},A_0:=\frac{1}{\o}A$, and $B_0:=\frac{1}{\lambda_S}B$.  
In this notation, the above theorem says that 
$$\Num(S)=\Z[A_0]\oplus\Z[B_0].$$
Note that $\lambda_S$ is also the smallest possible degree of a multisection of the elliptic fibration $p_A:S\to A/G$.
The multisection is provided by $A_0$.

\subsection{The algebraic Mukai lattice}
Let $Y$ be a smooth projective surface.  
We define the Mukai vector of an object $E\in\Db(Y)$ by $\v(E):=\ch(E)\sqrt{\td(Y)}\in H^*(Y,\Q)$, and we denote the algebraic Mukai lattice by $\Hal(Y,\Z)$, where $$\Hal(Y,\Z):=\v(K(Y)).$$  
In the cases of interest to us in this paper, $Y$ is either an abelian surface or a bielliptic surface.
In either case $\td(Y)=1$ so that $\v(E)=\ch(E)$.  
It follows that 
\begin{equation}\label{eq:AlgebraicMukaiLattice}
\Hal(Y,\Z) = H^0(Y,\Z) \oplus \Num(Y) \oplus H^4(Y,\Z),
\end{equation} since $\ch_2$ is always an integer because $c_1^2$ is even.  Writing the Mukai vector according to the decomposition \eqref{eq:AlgebraicMukaiLattice}, we get \[\v(E)=(r(E),c_1(E),\ch_2(E)).\]  We define the Mukai pairing $\langle\blank, \blank\rangle$ on $\Hal(Y,\Z)\times \Hal(Y,\Z) \to \Z$ by $\langle \v(E), \v(F)\rangle := - \chi(E, F)$.  
According to the Hierzebruch-Riemann-Roch theorem and the decomposition \eqref{eq:AlgebraicMukaiLattice}, we have
\[
\left\langle (r,c,s),(r',c',s')\right\rangle = c.c' - rs' - r's,
\]
for $(r,c,s),(r',c',s')\in \Hal(Y,\Z)$.

Given a Mukai vector $\v\in \Hal(Y,\Z)$, we denote
its orthogonal complement by
\[
\v^\perp:=\left\{\w\in \Hal(Y,\Z)\colon \langle \v,\w\rangle=0 \right\}.
\]
We call a Mukai vector $\v$ \emph{primitive} if it is not divisible in $\Hal(Y,\Z)$.  

For a bielliptic surface $S=(A\times B)/G$, it is most convenient to consider the canonical cover $X$ of $S$, which is the uniquely defined  \'{e}tale cover $\pi:X\to S$ of order $\o$ such that $\omega_X\cong\OO_X$ and for which $\pi$ is the quotient map for a naturally defined free action of the group $\Z/\o\Z$.
In this case, $X$ is an abelian surface that is an intermediate quotient of $A\times B$ by a subgroup $\Z/\lambda_S\Z\cong H<G$. 
As such, $X$ admits smooth elliptic fibrations $X\to A/H$ and $X\to B/H$ with fibres isomorphic to $B$ and $A$, respectively, and we  use $A_X,B_X$ for the corresponding classes in $\Num(X)$, etc., which satisfy $A_X. B_X=\lambda_S$. 
The action of $\Z/\o\Z$ is identified with the action $G':=G/H$.
We further observe that the quotient map $\pi$ induces an embedding $$\pi^*:\Hal(S,\Z)\into\Hal(X,\Z)$$ such that $\langle\pi^*\v,\pi^*\w\rangle=\o\langle \v,\w\rangle$.
Moreover, $\pi^*$ identifies $\Hal(S,\Z)$ with an index $\frac{\o^2}{\lambda_S}$ sublattice of the $G'$-invariant component of $\Hal(X,\Z)$ (although the cokernel of $\pi^*$ has exponent $\o$). 
Indeed, $\pi^*\v(\pt)=\pi^*(0,0,1)=(0,0,\o)$, and the embedding $\pi^*:\Num(S)\into\Num(X)$ satisfies $$\pi^*(A_0)=A_X,\pi^*(B_0)=\frac{\o}{\lambda_S}B_X.$$  It follows that for a primitive Mukai vector $\v\in\Hal(S,\Z)$, $\pi^*\v$ is divisible by at most $\o$.  We elaborate in the following result:

\begin{Lem}\label{primitive} A Mukai vector $\v=(r,c_1,s)\in\Hal(S,\Z)$ is primitive if and only if $$\gcd(r,c_1,s)=1.$$ For primitive $\v$ with $c_1=aA_0+bB_0$,  we set $$l(\v):=\gcd(r,\pi^*c_1,\o s)=\gcd(r,a,\frac{\o}{\lambda_S}b,\o s).$$  Then $l(\v)\mid \o$ and $\frac{\pi^*\v}{l(\v)}$ is primitive.
\end{Lem}
\begin{proof} The statement that $\v$ is primitive if and only if $\gcd(r,c_1,s)=1$ is clear since $s$ is an integer.

Suppose now that $\v$ is primitive and that some prime $p|l(\v)$.  If $p\nmid\o$, then $p$ divides $\gcd(r,a,b,s)=1$, a contradiction.  Thus $l(\v)$ is a power of 2 in Cases 1-4, a power of 3 in Cases 5-6, and of the form $2^i 3^j$ in Case 7.  Now we can write 
$$r=l(\v)r',a=l(\v)a',\frac{\o}{\lambda_S} b=l(\v)b',\mbox{ and }\o s=l(\v) s'$$ for some $r',a',b',c'\in\Z$.  Writing $l(\v)=2^i 3^j$ (with the understanding that $i$ or $j$ vanishes if $2$ or $3$ does not divide $\o$, respectively), suppose that $i>\ord_2(\o)$.  Then it follows that $2|r,2|a,2|b$, and $2|s$, a contradiction to $\gcd(r,a,b,s)=1$, so $i\leq\ord_2(\o)$, and, by symmetry, $j\leq\ord_3(\o)$ as well.  Thus $l(\v)\mid\o$.  

To see the final claim, observe that $$\pi^*\v=(r,aA_X+b\frac{\o}{\lambda_S}B_X,\o s)=l(\v)(r',a'A_X+b'B_X,s'),$$ where $s'=a'b'-c_2$ for an appropriate choice of $c_2$ and $(r',a'A_X+b'B_X,s')$ is primitive by construction.
\end{proof}

In addition to the above lemma, we make the follow observations which will be useful. 
\begin{Lem}\label{lem:intermediate bielliptic order composite}
If $\o$ is composite, with proper divisor $d$, then there is a bielliptic surface $\tilde{S}$ sitting as an intermediate \'{e}tale cover between $S$ and $X$,  $$\pi:X\mor[\tilde{\pi}]\tilde{S}\mor[\pi'] S,$$ such that $\ord(\omega_{\tilde{S}})=\frac{\o}{d}$ and \begin{equation}\label{bielliptic cover 1} \pi'^*A_0=\tilde{A}_0,\pi'^*B_0=\frac{\o}{\ord(\omega_{\tilde{S}})}\tilde{B}_0=d\tilde{B}_0,
\end{equation} where $\tilde{A}_0,\tilde{B}_0$ are the natural generators of $\Num(\tilde{S})$.  
\end{Lem}
\begin{proof}
Indeed, if we denote by $g$ a generator of $G'$ then $\tilde{S}=X/\langle g^d\rangle$ is the required bielliptic surface.
\end{proof}

There is a similar observation we will need:
\begin{Lem}\label{lem:intermediate bielliptic lambda bigger than 1}
If $\lambda_S>1$, then there is a bielliptic surface $\tilde{S}$ sitting as an intermediate \'{e}tale cover between $S$ and $A\times B$,  $$\bar{\pi}:A\times B\mor[\tilde{\pi}]\tilde{S}\mor[\pi'] S,$$ such that $\lambda_{\tilde{S}}=1$, $\ord(\omega_{\tilde{S}})=\o$, and \begin{equation}\label{bielliptic cover 2}\pi'^*A_0=\lambda_S\tilde{A}_0,\pi'^*B_0=\tilde{B}_0,
\end{equation}
where $\tilde{A}_0,\tilde{B}_0$ are the natural generators of $\Num(\tilde{S})$.  
\end{Lem}
\begin{proof}
The assumption that $\lambda_S>1$ implies that $S$ is of Type 2,4, or 6.  Then we may take $\tilde{G}$ to be the subgroup of $G$ generated by $-1$,$i$, or $e^{\frac{2\pi i}{3}}$, respectively.  Here we view $G$ via its action on $B$ as in Table \ref{table:families}.  Then by \cite[p. 589]{GH94}, $\tilde{S}=A\times B/\tilde{G}$ is a bielliptic surface of Type 1,3, or 5, respectively, and $A\times B\to S$ factors as $A\times B\mor[\tilde{\pi}]\tilde{S}\mor[\pi'] S$ with the required properties.
\end{proof}

\part{Classical Methods in stability}\label{Part:ClassicalStability}
In this part of the paper, we reduce the proof of \cref{MainThm1:Non-emptinessOfGieseker} by pushing as far as possible some classical methods in the theory of Gieseker and slope stability.  We begin with the results we need from this theory.
\section{Review: Gieseker and slope stability}
We review here the classical notions of stability for sheaves on surfaces.  Let $Y$ be a smooth projective surface over $\C$ and $\omega,\beta\in\NS(Y)_\Q$ with $\omega$ ample.  

\subsection*{Slope stability}
We define the slope function
$\mu_{\omega, \beta}$ on $\Coh Y$ by
\begin{equation} \label{eq:muomegabeta}
\mu_{\omega, \beta}(E) = 
\begin{cases}
\frac{\omega.(c_1(E) - \rk(E)\beta)}{\rk(E)} & \text{if $\rk(E) > 0$,} \\
+\infty & \text{if $\rk(E) = 0$.}
\end{cases}
\end{equation}
This gives a notion of slope stability for sheaves, for which Harder-Narasimhan filtrations exist (see \cite[Section 1.6]{HL10}).  
We refer to this as $\beta$-twisted $\omega$-slope stability and use the notation $\mu_{\omega,\beta}$-stability (or just $\mu$-stability).

An important notion that makes our analysis simpler is that of a generic polarization.  
For any Mukai vector $\v$, there is a locally finite set of walls in the ample cone $\Amp(S)$ corresponding to polarizations $\omega$ for which a $\mu_{\omega,\beta}$-semistable sheaf $E$ with $\v(E)=\v$ contains a subsheaf $F$ of strictly smaller positive rank with $\mu_{\omega,\beta}(F)=\mu_{\omega,\beta}(E)$.  
We say $\omega$ is \emph{generic} (with respect to $\v$) if it lies on no walls for $\v$.
A consequence of this definition (see \cite[Theorem 4.C.3]{HL10}) is that for a $\mu_{\omega,\beta}$-semistable sheaf $E$ and a generic polarization $\omega$, any subsheaf $F\subset E$ of the same slope must satisfy 
$$\frac{c_1(E)}{\rk(E)}=\frac{c_1(F)}{\rk(F)}.$$

\subsection*{Gieseker stability}
For $E\in\Coh(Y)$, we define the \emph{twisted Hilbert polynomial} by
\[
P(E,\beta,m) := \int_Y e^{-\beta+m\omega}\ch(E)\td(Y)=a_2(E,\beta)\frac{m^2}{2}+a_1(E,\beta)m+a_0(E,\beta),
\]
and for a $d$-dimensional sheaf $E$ we define the \emph{twisted reduced Hilbert polynomial} by  $$p(E,\beta,m):=\frac{P(E,\beta,m)}{a_d(E,\beta)}.$$  Note that $a_i(E,\beta)=0$ for $i>d$.
This gives rise to the notion of $\beta$-twisted $\omega$-Gieseker stability for pure sheaves, introduced first in \cite{MW97}.
When $\beta=0$, this is nothing but Gieseker stability.  In the sequel, we will say a sheaf $E$ is \emph{semistable} (resp. \emph{stable}) if it is Gieseker semistable (resp. Gieseker stable).
We refer to \cite[Section 1]{HL10} for basic properties of Gieseker stability.  
In particular, recall that for every $\beta$-twisted $\omega$-semistable sheaf $E$ there exists Jordan-H\"{o}lder filtrations with $\beta$-twisted $\omega$-stable factors of the same twisted reduced Hilbert polynomial as $E$, and each such filtration for $E$ has the same associated graded object $gr^{\JH}(E)$. 
Moreover, two semistable sheaves $E_1$ and $E_2$ are said to be $S$-\emph{equivalent} (i.e. $E_1\sim_S E_2$) if $gr^{\JH}(E_1)\cong gr^{\JH}(E_2)$.  

We also recall that $\mu$-stability and Gieseker stability are related as follows: $$E\mbox{ is }\mu\mbox{-stable}\Rightarrow E\mbox{ is stable}\Rightarrow E\mbox{ is semistable}\Rightarrow E\mbox{ is }\mu\mbox{-semistable}.$$

\subsection*{Moduli spaces of semistable sheaves}
For a fixed Mukai vector $\v\in\Hal(Y,\Z)$, we denote by $\MM_{\omega}^{\beta}(\v)$ the moduli stack of flat families of $\beta$-twisted $\omega$-Gieseker semistable sheaves with Mukai vector $\v$.
By \cite[Section 4]{HL10} and \cite{MW97}, there exists a projective variety $M_{\omega}^{\beta}(\v)$ which is a coarse moduli space parameterizing $S$-equivalence classes of semistable sheaves.  We will have cause to consider the moduli stack $\MM_{\omega}^{\beta,\mu ss}(\v)$ of flat families of $\mu_{\omega,\beta}$-semistable sheaves in which $\MM_{\omega}^{\beta}(\v)$ is an open substack.  Moreover, the open substacks $\MM_{\omega}^{\beta,\mu s}(\v)\subseteq\MM_{\omega}^{\beta,s}(\v)\subseteq \MM_{\omega}^{\beta}(\v)$ parameterizing $\mu$-stable sheaves and stable sheaves, respectively, are $\mathbb{G}_m$-gerbes over the open subsets $M_{\omega}^{\beta,\mu s}(\v)\subseteq M_{\omega}^{\beta, s}(\v)\subseteq M_{\omega}^{\beta}(\v)$.
When $\beta=0$, we will denote the corresponding moduli spaces by $\MM_\omega(\v)$, etc.

\section{Comparing classical stabilities}\label{subsec:SlopeSemistable}
%We begin with a helpful but elementary observation (a version of \cite[Lemma 3.5]{Yos14} modified for bielliptic surfaces and proven in the same way):
%\begin{Lem}\label{det} For a locally free sheaf $F$ of rank $r$ on $X$, if $2\mid\o$, then $$\det\pi_*(F)\cong\det(\pi_*(\det F))\otimes\OO\left((r-1)\left(\frac{\o}{2}\right)K_S\right),$$ while if $\o=3$ then $$\det\pi_*(F)\cong\det(\pi_*(\det F)).$$
%\end{Lem}
Moduli spaces behave differently depending on whether $\v^2$ vanishes or not, so we separate our investigation into two.

\subsection*{The isotropic case} Let us begin by studying sheaves with vanishing discriminant, that is, sheaves with Mukai vectors $\v$ such that $\v^2=0$, completing the results of \cite{Tak73} and \cite{Ume75}.

\begin{Prop}\label{isotropic} If $\omega$ is a generic polarization with respect to $\v$ such that $\v^2=0$, then 
\begin{enumerate}
\item $M^{\mu ss}_\omega(\v)=M_\omega(\v)$ and any $E\in M^{\mu ss}_\omega(\v)$ is $S$-equivalent to $\bigoplus E_i$ where each $E_i$ is a $\mu$-stable locally free sheaf with $\v(E_i)\in\Q \v$.  
\item Let $\v_0$ be a primitive isotropic Mukai vector.  If $M^s_\omega(n\v_0)\neq\varnothing$, then $nl(\v_0)\mid\o$.  Moreover, when nonempty, $M^s_\omega(n\v_0)$ is smooth and has dimension 1 unless $n=\frac{\o}{l(v_0)}$, when it has dimension 2. 
\item Finally, $\dim\MM_\omega(n\v_0)\leq\left\lfloor\frac{nl(\v_0)}{\o}\right\rfloor$.
\end{enumerate}
\end{Prop}
\begin{proof}  
(1)  Let $\v=(lr,lD,s)$ with $\gcd(r,D)=1$.  Choose a Jordan-H\"{o}lder filtration of $E\in M^{\mu ss}_\omega(\v)$ for $\mu$-stability with $\mu$-stable factors $E_i$.  Since $\omega$ is generic, we must have $\v_i=\v(E_i)=(l_i r,l_i D,s_i)$ with $\v_i^2\geq 0$.  Since $$0=\frac{\v^2}{l}=\sum_i\frac{\v_i^2}{l_i},$$ we must have $\v_i^2=0$ for all $i$.  Thus $\frac{D^2}{2r}=\frac{s_i}{l_i}=\frac{\chi(E_i)}{r_i}=\frac{\chi(E)}{r}$ for all $i$, so $E$ is Gieseker semistable as well.  Finally, note that since $E_i^{\vee\vee}$ is a $\mu$-stable locally free sheaf, $\v(E_i^{\vee\vee})^2\geq 0$, so  $$0=\v_i^2=\v(E_i^{\vee\vee})^2+2l_i r \chi(E_i^{\vee\vee}/E_i).$$  Thus $\chi(E_i^{\vee\vee}/E_i)=0$ and $E_i$ is locally free.  It follows that $E$ is locally free as well.

(2) Write $\pi^*\v_0=l(\v_0)\w$ for primitive $\w\in\Hal(X,\Z)$, and suppose that $E\in M^s_\omega(n\v_0)$.  If $\pi^*E\in M^s_{\pi^*\omega,X}(n\pi^*\v_0)$, then we must have $nl(\v_0)=1$ so $n=l(\v_0)=1$.  Otherwise, $\pi^*E$ is strictly semistable, so by \cite{Tak73} (or by the more general result we prove in \cref{lem:WhenIsPullbackStable} below) we have 
\begin{equation}\label{eqn:semistable pull-back}
\pi^*E\cong\bigoplus_{k=0}^{d-1}(g^k)^*F
\end{equation}
for some $1<d\leq\o$ such that $d\mid \o$ and some $\pi^*\omega$-stable $F$ such that the $(g^k)^*F$ are distinct for $k=0,\dots,d-1$.  Recall that $g$ denotes a generator of $G'\cong \Z/\o\Z$.

Now write $\w=(r,c_1,s)$ so that $$n\pi^*\v_0=nl(\v_0)\w=nl(\v_0)(r,c_1,s).$$  From \eqref{eqn:semistable pull-back}, we get  $$n\pi^*\v_0=\sum_{k=0}^{d-1}(g^k)^*(\rk(F),c_1(F),\ch_2(F))=(d\rk(F),\sum_{k=0}^{d-1}(g^k)^*c_1(F),d\ch_2(F)),$$ so $nrl(\v_0)=d\rk(F)$, $\sum_{k=0}^{d-1}(g^k)^*c_1(F)=nl(\v_0)c_1$, and $d\ch_2(F)=nsl(\v_0)$.  From the Hodge index theorem we see that for any $k$, $$c_1(F)^2-c_1(F).(g^k)^*c_1(F)\leq 0,$$ with strict inequality unless $c_1(F)=(g^k)^*c_1(F)$, so we get that $$n^2 l(\v_0)^2c_1^2=d(c_1(F)^2+\sum_{k=1}^{d-1}c_1(F).(g^k)^*c_1(F))\geq d^2c_1(F)^2.$$  Moreover, from the Bogomolov inequality, $$0\leq\Delta(F)=c_1(F)^2-2\rk(F)\ch_2(F)\leq\frac{n^2 l(\v_0)^2 c_1^2}{d^2}-2\left(\frac{nrl(\v_0)}{d}\right)\left(\frac{nsl(\v_0)}{d}\right)=0,$$ as $c_1^2-2rs=\w^2=0$.  It follows that $c_1(F)=(g^k)^*c_1(F)$ for all $k$.  

Thus $F$ is a stable locally free sheaf with respect to $\pi^*\omega$ such that $\v(F)\in\Q_{>0} \w$ and thus $\v(F)=\w$.  Indeed, even if $\pi^*\omega$ lies on a wall with respect to $\w$, $F$ will remain stable for small deformations of $\pi^*\omega$, so $\v(F)$ must be primitive since $\v(F)^2=0$.   It follows that $nl(\v_0)=d$ which divides $\o$, as claimed.

Now let us address the question of smoothness.  From \cite{Tak73}, we know that $M^s_\omega(\v)$ is smooth of dimension $\v^2+1=1$ at $E$ unless $E\cong E\otimes\o$.  But then we would have $E\cong\pi_*(F)$ for some stable $F$.  As $M^s_{\pi^*\omega,X}(\v(F))$ is always at least 2 dimensionial and the map $\pi_*:M^s_{\pi^*\omega,X}(\v(F))^{\circ}\to M^s_\omega(\v)$ is \'{e}tale, such an $E$ cannot be a singular point and instead lies on a smooth 2 dimensional component.  Moreover, in this case $\pi^*E\cong\oplus_{i=0}^{\o}(g^i)^*F$ with $F\ncong(g^i)^*F$ for $0<i<\o$ from which it follows that $nl(\v_0)=\o$.  Thus $M_\omega(n\v_0)$ is indeed always smooth, and it has dimension one unless $n=\frac{\o}{l(\v_0)}$ when it has dimension two.

(3) We have a decomposition as in \cite[Lemma 9.3]{Nue14b},\cite[Proposition 1.6(3)]{Yos16a}  $$M_\omega(n\v_0)=\coprod_{(n_1r_1,\dots,n_t r_t)\in S_n}\prod_i \Sym^{n_i}(M^s_\omega(r_i\v_0)),$$ where $$S_n:=\{(n_1r_1,\dots,n_tr_t)|r_1<r_2<\cdots<r_t,n_1,\dots,n_t\in\Z_{>0},\sum_i n_ir_i=n\}.$$  We consider the moduli stack $\MM_\omega(n\v_0)$ of semistable sheaves of Mukai vector $n\v_0$ and the corresponding morphism $\phi:\MM_\omega(n\v_0)\to M_\omega(n\v_0)$ which identifies all $S$-equivalent sheaves with the unique polystable sheaf in their $S$-equivalence class.  For a point $x\in M_\omega(n\v_0)$  there are $E_j\in M^s_\omega(k_j \v_0)$, with $E_j\ncong E_{j'}(pK_S)$ for any $p\in\Z$ and $j'\neq j$, such that $x=\oplus_{j=1}^t\oplus_p E_j(pK_S)^{\oplus n_{j,p}}$.  It follows that for any $j\neq j'$, $\Hom(E_j,E_{j'}(pK_S))=0$ and $\Ext^2(E_j,E_{j'}(pK_S))=\Hom(E_{j'}(pK_S),E_j(K_S))=0$.  As $\chi(E_j,E_{j'}(pK_S))=0$, we see that $\Ext^1(E_j,E_{j'}(pK_S))=0$ as well.  So each $E\in\phi^{-1}(x)$ decomposes as $E=\oplus_j F_j$ where $F_j$ is $S$-equivalent to $\oplus_p E_j(pK_S)^{\oplus n_{j,p}}$.  One can show as in \cite[(3.8)]{KY08} that for each $j$ the stack of such $F_j$ has dimension at most $-1$.  Thus $\dim\phi^{-1}(x)\leq -t$.  So for $1\leq d<\o$, if we let 
\begin{equation}
\begin{split}
t_1=|\{j\mid\dim M^s_\omega(k_j \v_0)=1\}|\\
t_2=|\{j\mid\dim M^s_\omega(k_j\v_0)=2\}|,
\end{split}
\end{equation}
then $$\dim \MM_\omega(n\v_0)\leq\max_{x\in M_\omega(\v)} \{-t+t_1+2t_2\}=\max_{x\in M_\omega(\v)} t_2\leq n,$$ as $t=t_1+t_2$.  As $\dim M^s_\omega(k\v_0)=2$ implies $k=\frac{\o}{l(\v_0)}$, we thus get $t_2\leq \left\lfloor\frac{nl(\v_0)}{\o}\right\rfloor$, as claimed.
\end{proof}
\begin{Rem}
Once we've introduced Bridgeland stability conditions, we'll be able to prove that $nl(\v_0)\mid\o$ is also sufficient to guarantee that $M_\omega^s(n\v_0)\ne\varnothing$.
\end{Rem}
\subsection*{The case \texorpdfstring{$\v^2>0$}{Lg}}
Now we consider $\v=(nr,nD,s)$ with $\gcd(r,D)=1$ and $\v^2>0$.  Let $\u=(n_0r,n_0 D,s_0)$ be the unique primitive isotropic vector of this form.  Indeed, it is clear from $\u^2=0$ that $s_0=\frac{n_0 D^2}{2r}$.  From  \cref{isotropic}, we know that when non-empty, $M^{\mu ss}_\omega(\u)=M_\omega(\u)$ and consists solely of $\mu$-stable vector bundles.  We record the following observation:

\begin{Lem}\label{divisible} For $\v'=(n'r,n'D,s')$, $r\mid \langle \u,\v'\rangle$.
\end{Lem}
\begin{proof} Indeed, $\langle \u,\v'\rangle=n'n_0D^2-s'n_0r-n'rs_0$.  Using the fact that $n_0D^2=2rs_0$, we get \begin{equation}\label{eqn: divisible}
\langle \u,\v'\rangle=n'rs_0-s'n_0r=r(n's_0-s'n_0),\end{equation} as claimed.
\end{proof}

Our main result in the positive square case is the following:

\begin{Prop}\label{slope stability} 
Suppose that $\v=(nr,nD,s)\in\Hal(S,\Z)$ is a Mukai vector such that $\v^2>0$ and $\gcd(r,D)=1$.  For a polarization $\omega$ that is generic with respect to $\v$, we have:
\begin{enumerate}
\item $\codim(M_\omega(\v)\backslash M^s_\omega(\v))\geq 2$.
\item $\codim(M^{\mu ss}_\omega(\v)\backslash M_\omega(\v))\geq 1$.
\item $\codim(M^s_\omega(\v)\backslash M^{\mu s}_\omega(\v))\geq 1$.
\item If $nr>1$, then $\codim(M_\omega^{\mu s}(\v)\backslash M^{\mu s,lf}_\omega(\v))\geq 1$ unless $\o=2$ and $\v=(2,0,-1)e^D$ or $S$ is of Type 1 and $\v=(2,B_0,-1)e^D$.
\end{enumerate}
Thus if $M^{\mu ss}_\omega(\v)\neq\varnothing$, then $\dim M^{\mu ss}_\omega(\v)=\v^2+1$ and there is a $\mu$-stable sheaf in each irreducible component.  Moreover, except when $\o=2$ and $\v=(2,0,-1)e^D$ or $S$ is of Type 1 and $\v=(2,B_0,-1)e^D$, this $\mu$-stable sheaf may be taken to be locally free as long as $nr>1$.
\end{Prop}
\begin{proof}  We will prove the corresponding statements for the appropriate moduli stacks.  It will follow from the final statement that the statements hold for the coarse moduli schemes as well.  For a good introduction to this perspective, see \cite{Yos16a}.  In particular, we note from the quotient stack description provided there that for any irreducible component $\MM$ of $\MM^{\mu ss}_\omega(\v)$, $\dim \MM\geq \v^2$.  

(1) Writing $\v=l\v_0$ with $\v_0$ primitive, we proceed by induction on $l$.  When $l=1$ so that $\v$ is primitive, then $\MM_\omega(\v)=\MM_\omega^s(\v)$ so the claim is trivial.  Otherwise, $l>1$ and any $E\in \MM_\omega(\v)\backslash \MM^s_\omega(\v)$ admits a saturated stable subsheaf $E_1\in \MM^s_\omega(l_1\v_0)$ such that $E$ fits into a short exact sequence $$0\to E_1\to E\to E_2\to 0$$ with $E_2\in\MM_\omega(l_2 \v_0)$ and $l_2=l-l_1$.  Denote by $J(l_1,l_2)$ the stack of such extensions and by $J(l_1,l_2)^n$ the substack of those extensions with $\hom(E_1,E_2(K_S))=n$.  Then by \cite[Lemma 5.2]{KY08} $$\dim J(l_1,l_2)^n= \dim\FF(l_1,l_2)^n+\langle l_1 \v_0,l_2 \v_0\rangle+n,$$ where $$\FF(l_1,l_2)^n:=\Set{(E_1,E_2)\in\MM^s_\omega(l_1\v_0)\times\MM_\omega(l_2\v_0)\ |\ \hom(E_1,E_2(K_S))=n}.$$  Of course, $\hom(E_1,E_2(K_S))\leq\frac{\rk E_2(K_S)}{\rk E_1}=\frac{l_2}{l_1}$ by \cite[Lemma 3.1]{KY08} and by induction on $l$, $\MM^s_\omega(l_i \v_0)$ is non-empty of dimension $l_i^2\v_0^2$, so generically $\hom(E_1,E_2(K_S))=0$ and thus \begin{align}
\begin{split}
\dim J(l_1,l_2)&\leq l_1^2\v_0^2+l_2^2\v_0^2-1+\langle l_1\v_0,l_2 \v_0\rangle+\left\lfloor\frac{l_2}{l_1}\right\rfloor\\&\leq \v^2-(l_1 l_2 \v_0^2-\frac{l_2}{l_1}+1). 
\end{split}
\end{align}  As we are assuming that $\v^2>0$, we must have $\v_0^2>0$ and since it must be even, $\v_0^2\geq 2$.  It follows that $$l_1 l_2 \v_0^2-\frac{l_2}{l_1}+1\geq 2l_1 l_2-\frac{l_2}{l_1}+1\geq l_2+1\geq 2,$$ so the statement about the codimension follows from $\dim \MM\geq \v^2$ for any irreducible component $\MM$ of $\MM_\omega(\v)$.

(2) Consider the Harder-Narasimhan filtration of a sheaf $E\in \MM^{\mu ss}_\omega(\v)\backslash \MM_\omega(\v)$, $$0=\HN^0(E)\subset \HN^1(E)\subset \cdots \HN^k(E)=E,$$ where the factors $E^i=\HN^i(E)/\HN^{i-1}(E)$ are semistable of strictly decreasing reduced Hilbert polynomial.  As $\omega$ is generic, $\v_i=\v(E^i)=(n_ir,n_iD,s_i)$ with $\v_i^2\geq 0$.  As $\v_i^2\geq 0=\u^2$, it follows that $\frac{s_i}{n_i}\leq\frac{s_0}{n_0}$, so the condition on the reduced Hilbert polynomials becomes $$\frac{s_0}{n_0}\geq \frac{s_1}{n_1}>\cdots>\frac{s_k}{n_k}.$$  

We consider the substack $\FF^{\HN}(\v_1,\dots,\v_k)\subset \MM^{\mu ss}_\omega(\v)$ of those $E$ having a Harder-Narasimhan filtration as above.  As $\Hom(E^i,E^j(K_S))=0$ for $i<j$, from the strict inequality of reduced Hilbert polynomials, we get by \cite[Lemma 5.3]{KY08} $$\dim \FF^{\HN}(\v_1,\dots,\v_k)=\sum_{i=1}^k \dim \MM_\omega(\v_i)+\sum_{i<j}\langle \v_i,\v_j\rangle.$$  Of course, for $i<j$ we have $$\langle \v_i,\v_j\rangle=n_j\frac{\v_i^2}{n_i}+(s_in_j-s_j n_i)r\geq (s_in_j-s_j n_i)r\geq r\geq 1.$$  If $\v_i^2>0$ for all $i$, then by induction on $n$ we get $$\dim\FF^{\HN}(\v_1,\dots,\v_k)=\v^2-\sum_{i<j}\langle \v_i,\v_j\rangle\leq \v^2-1.$$  If some $\v_i^2=0$, then necessarily $i=1$ and $\v_1=n_1' \u$.  Now $r\mid\langle \u,\v_j\rangle$ by \cref{divisible}, so it follows that $\langle \v_1,\v_j\rangle-n_1'=n_1'(\langle \u,\v_j\rangle-1)\geq n_1'(r-1)\geq0$.  As $\dim\MM_\omega(\v_1)\leq n_1'$,  we get
$$\dim\FF^{\HN}(\v_1,\dots,\v_k)\leq \v^2+n_1'-\sum_{i<j}\langle \v_i,\v_j\rangle\leq \v^2-\sum_{2\leq i<j}\langle \v_i,\v_j\rangle\leq \v^2-1,$$ unless perhaps $k=2$, $\dim\MM_\omega(\v_1)=n_1'$, and $\langle \v_1,\v_2\rangle=n_1'$.  But then $l(\u)=\o$ and $\langle \u,\v_2\rangle=1$, the latter of which forces $r=1$.  The conditions $l(\u)=\o$ and $r=1$ force $\o\mid n_0$ and $n_0\mid s_0$, so from \eqref{eqn: divisible} we see that $\o\mid \langle \u,\v_2\rangle$, a contradiction.  Thus we get the claim about the codimension in either case.

(3) For $E\in\MM^s_\omega(\v)\backslash\MM^{\mu s}_\omega(\v)$ we consider a saturated $\mu$-stable subsheaf $E_1\subset E$ of the same slope.  Then $E_2=E/E_1$ is $\mu$-semistable, $\v_i=(n_ir,n_i D,s_i)$, and $\frac{s_1}{n_1}<\frac{s}{n}<\frac{s_2}{n_2}$, so in particular $\v_1^2>0$.  By induction on $n$, we thus have $\dim\MM^{\mu ss}_\omega(\v_1)=\v_1^2$.  

Let $\JJ(\v_1,\v_2)$ be the substack of $\MM^s_\omega(\v)$ parametrizing those $E$ admitting a saturated $\mu$-stable subsheaf $E_1\subset E$ of Mukai vector $\v_1$.  By \cite[Lemma 5.2]{KY08}, if we define $$\JJ_k^m(\v_1,\v_2):=\Set{E\in \JJ(\v_1,\v_2)\ |\ \hom(E_1,E_2(K_S))=m,\chi(E_1^{\vee\vee}/E_1)=k},$$ then \begin{equation}\label{filtration}\dim\JJ_k^m(\v_1,\v_2)=\dim \NN_k^m(\v_1,\v_2)+\langle \v_1,\v_2\rangle+m,\end{equation} where $$\NN_k^m(\v_1,\v_2):=\Set{(E_1,E_2)\in\MM^{\mu s}_\omega(\v_1)\times\MM^{\mu ss}_\omega(\v_2)\ |\ \hom(E_1,E_2(K_S))=m,\chi(E_1^{\vee\vee}/E_1)=k}.$$  We also let $\NN_k(\v_1,\v_2)=\bigcup_{m>0}\NN_k^m(\v_1,\v_2)$.  

For fixed $E_2\in\MM^{\mu ss}_\omega(\v_2)$, suppose that $\hom(E_1,E_2(K_S))\neq 0$ for some $E_1\in\MM^{\mu s}_\omega(\v_1)$.  Then $E_1^{\vee\vee}$ is a summand in the double-dual of the graded object associated to a Jordan-H\"{o}lder filtration of $E_2(K_S)$.  As such, given a fixed $E_2$, there are only finitely many $E_1^{\vee\vee}$ for which $\hom(E_1,E_2(K_S))\neq 0$.  So the fiber over $E_2$ of the second projection $\pi_2:\NN_k(\v_1,\v_2)\to\MM^{\mu ss}_\omega(\v_2)$, which parametrizes those $E_1$ with $\hom(E_1,E_2(K_S))>0$, is the union over the finite set of such $E_1^{\vee\vee}$ of open subschemes of $\Quot_{E_1^{\vee\vee}}^k$ .  As $\dim\Quot_{E_1^{\vee\vee}}^k=(n_1r+1)k$ \cite[Theorem 6.A.1]{HL10}, $$\dim\pi_2^{-1}(E_2)\leq \dim\Quot_{E_1^{\vee\vee}}^k-\dim \Aut(E_1)=(n_1r+1)k-1.$$

It follows that \begin{align}\label{estimate}
\begin{split}\dim\NN_k(\v_1,\v_2)&\leq\dim\MM^{\mu ss}_\omega(\v_2)+(n_1r+1)k-1\\
&\leq\dim\MM^{\mu ss}_\omega(\v_2)+\dim\MM^{\mu s}_\omega(\v_1)-((n_1r-1)k+1).
\end{split}
\end{align}
Indeed, the last inequality is true because $E_1^{\vee\vee}$ being $\mu$-stable forces $0\leq \v(E_1^{\vee\vee})^2=\v(E_1)^2-2n_1rk$, so $\dim\MM^{\mu s}(\v_1)\geq 2n_1 rk$.  Thus $\NN_k(\v_1,\v_2)$ has codimension at least 1 in $\MM^{\mu s}_\omega(\v_1)\times\MM^{\mu ss}_\omega(\v_2)$.

Now suppose that $\v_2^2>0$.  Then by induction on $n$, $\dim\MM^{\mu ss}_\omega(\v_2)=\v_2^2$ and $\hom(E_1,E_2(K_S))\leq\frac{n_2}{n_1}$, so by \eqref{filtration}, 
\begin{align}
\begin{split}\dim\JJ_k^m(\v_1,\v_2)&=\dim\NN_k^m(\v_1,\v_2)+\langle \v_1,\v_2\rangle+m\\
&\leq\dim\MM^{\mu ss}_\omega(\v_2)+\dim\MM^{\mu s}_\omega(\v_1)-1+\langle \v_1,\v_2\rangle+\left\lfloor\frac{n_2}{n_1}\right\rfloor\\
&\leq\dim\MM^s_\omega(\v)-\left(\langle \v_1,\v_2\rangle+1-\frac{n_2}{n_1}\right)\\
&=\dim\MM^s_\omega(\v)-\left(\frac{n_1}{2n_2}\v_2^2+\frac{n_2}{2n_1}\v_1^2+1-\frac{n_2}{n_1}\right).
\end{split}
\end{align}

As $\v_i^2>0$ and are both even, $\v_i^2\geq 2$, so $$\frac{n_1}{2n_2}\v_2^2+\frac{n_2}{2n_1}\v_1^2+1-\frac{n_2}{n_1}\geq \frac{n_1}{n_2}+\frac{n_2}{n_1}+1-\frac{n_2}{n_1}\geq 1,$$and it follows that every irreducible component of $\MM^s_\omega(\v)$ contains a $\mu$-stable sheaf if $\v_2^2>0$.

If $\v_2^2=0$, then $\v_2=n_2'\u$ and $n_2=n_2' n_0$.  Moreover, if in addition $\hom(E_1,E_2(K_S))>0$, then by Lemma \ref{isotropic} we must have $\v(E_1^{\vee\vee})\in\Q \v_2$, so $\v(E_1^{\vee\vee})=n_1'\u$ and $n_1=n_1' n_0$.  As $E_1^{\vee\vee}$ is $\mu$-stable, $n_1'l(\u)\mid\o$ and $\dim\MM^{\mu ss}_\omega(\v_2)\leq\lfloor\frac{n_2'l(\u)}{\o}\rfloor$ by Proposition \ref{isotropic}, so using \eqref{estimate} we get 
\begin{align}
\begin{split}
\dim\JJ_k^m(\v_1,\v_2)&\leq\frac{n_2'l(\u)}{\o   }+2n_1rk-((n_1r-1)k+1)+\langle \v_1,\v_2\rangle+\frac{n_2}{n_1}\\
&=\dim\MM^s_\omega(\v)-\langle \v_1,\v_2\rangle+\frac{n_2'l(\u)}{\o}-((n_1r-1)k+1)+\frac{n_2}{n_1}\\
&=\dim\MM^s_\omega(\v)-\left(n_2'\left(\langle \v_1,\u\rangle-\left(\frac{l(\u)}{\o}+\frac{1}{n_1'}\right)\right)+k(n_1' n_0 r-1)+1\right)\\
&\leq\dim\MM^s_\omega(\v)-\left(n_2'\left(r-\left(\frac{l(\u)}{\o}+\frac{1}{n_1'}\right)\right)+k(n_1' n_0 r-1)+1\right)\\
&\leq\dim\MM^s_\omega(\v)-1,
\end{split}
\end{align}
where the last inequality holds if $r\geq 2$, since $n_1'l(\u)\mid\o$ implies that $\frac{l(\u)}{\o}+\frac{1}{n_1'}\leq 2$, so the term in parentheses is at least 1.  If $r=1$, then the conditions on $\u$ guarantee that $n_0=1$ and $l(\u)=1$.  If $n_1'>1$, then as $\o\geq 2$, the final inequality still holds.  Otherwise, $n_1'=1$, and using the explicit calculation that $\langle \v_1,\u\rangle=\frac{\v_1^2}{2n_1'}$, we find that the final inequality continues to hold unless $\v_1^2=2$.  

So let us deal separately with the case $r=n_1=n_0=l(\u)=1$ and $\v_1^2=2$.  Then twisting by $D$ we may assume that $\v_1=(1,0,-1)$ and $\v_2=n_2\u=n_2(1,0,0)$.  Thus $\MM^{\mu s}_\omega(\v_1)=\MM^s_\omega(\v_1)$ parametrizes sheaves of the form $I_x(L)$ for a point $x\in S$ and $L\in\Pic^0(S)$, and $\MM^{\mu ss}_\omega(\v_2)$ parametrizes line bundles in $\Pic^0(S)$, stable bundles $E_2$ such that $\pi^*E_2$ is the direct sum of line bundles in $\Pic^0(X)$ \cite[Lemma 2.6]{Ume75}, and iterated extensions of these.  If $m=\hom(I_x(L),E_2(K_S))>0$, then by \cite[Lemma 3.1]{KY08} the evaluation map is injective, $I_x(L)^{\oplus m}\into E_2(K_S)$, and by taking double-duals we get $\OO_S(L)^{\oplus m}$ is subsheaf of $E_2(K_S)$.  Moreover, from $\mu_\omega(\OO_S(L))=\mu_\omega(E_2(K_S))$ and the $\mu$-semistability of $E_2(K_S)$, $\OO_S(L)^{\oplus m}$ must be saturated in $E_2(K_S)$.  Thus $\OO_S(L)$ has to appear at least $m$ times in the polystable sheaf $\phi(E_2(K_S))$, where $\phi:\MM_\omega(\v_2)\to M_\omega(\v_2)$ is the map to the coarse moduli space.  Considering the image of the second projection $\pi_2:\NN_1^m(\v_1,\v_2)\to \MM_\omega^{\mu ss}(\v_2)$, it follows as in part (c) of Lemma \ref{isotropic} that $$\dim\pi_2(\NN^m_1(\v_1,\v_2))\leq\max_{z\in\pi_2(\NN^m_1(\v_1,\v_2))}\{-t+t_1+2t_2\}=\max_{z\in\pi_2(\NN^m_1(\v_1,\v_2))}t_2\leq\left\lfloor\frac{n_2-m}{\o}\right\rfloor,$$ where the last inequality follows because $l(\u)=1$ and $L$ appears $m$ times so that $\o t_2+m\leq n_2$.  Thus $\dim\NN_1^m(\v_1,\v_2)\leq \left\lfloor\frac{n_2-m}{\o}\right\rfloor+2$, so $$\dim\JJ_1^m(\v_1,\v_2)\leq\frac{n_2-m}{\o}+n_2+m+2=2+\frac{\o+1}{\o}n_2+\frac{\o-1}{\o}m< 2n_2+2=\v^2,$$ if $m<n_2$.  Thus again we get $\dim\JJ_1^m(\v_1,\v_2)<\dim\MM^s_\omega(\v)$, as claimed.  If $m=n_2$, then our estimates give $\dim\JJ_1^{n_2}(\v_1,\v_2)=\v^2$, and we necessarily have that $E_2(K_S)=\OO_S(L)^{\oplus n_2}$.  But for an extension $E$ fitting into the short exact sequence \begin{equation}\label{eqn: exceptional case}
0\to I_x(L)\to E\to \OO_S(L-K_S)^{\oplus n_2}\to 0
\end{equation} to be stable we must have $n_2\leq 2$.  Indeed, if $E$ were stable, then $\Hom(\OO_S(L-K_S),E)=0$ as $$\frac{\chi(E)}{n_2+1}=\frac{-1}{n_2+1}<0=\frac{\chi(\OO_S(L-K_S))}{1},$$ so applying $\Hom(\OO_S(L-K_S),-)$ to \eqref{eqn: exceptional case} we get that $$\C^{n_2}=\Hom(\OO_S(L-K_S),\OO_S(L-K_S))^{\oplus n_2}\into\Ext^1(\OO_S(L-K_S),I_x(L))=H^1(S,I_x(K_S))=\C^2,$$ where the final equality is easily seen.  Thus $n_2\leq 2$.  

To finish our analysis of the cases $n_2=1,2$, it is easiest to consider the coarse moduli spaces directly.  If $n_2=1$, then for each $I_x(L)\in M_\omega(\v_1)$, there is a $\P^1$ worth of distinct stable extensions as in \cref{eqn: exceptional case}.  Thus the dimension of the sublocus of such extensions in $M_\omega^s(\v)$ is $$\dim M_{\omega}(\v_1)+1=3+1+1=4<5=\v^2+1\leq\dim M^s_{\omega}(\v).$$  Similarly, if $n_2=2$, then there is a unique such stable extension for each $I_x(L)$, so this sublocus is three dimensional, while $\dim M_{\omega}(\v)\geq \v^2+1=7$.   

Finally, we must consider the case when $\hom(E_1,E_2(K_S))=0$.  But then 
\begin{align}
\begin{split}
\dim\JJ_k^0(\v_1,\v_2)&\leq\dim\MM^{\mu ss}_\omega(\v_2)+\dim\MM^{\mu s}_\omega(\v_1)+\langle \v_1,\v_2\rangle\\
&\leq \frac{n_2'l(\u)}{\o}+\v_1^2+n_2'\langle \v_1,\u\rangle\\
&=\dim\MM^s_\omega(\v)-n_2'\langle \v_1,\u\rangle+\frac{n_2'l(\u)}{\o}\\
&\leq\dim\MM^s_\omega(\v)-n_2'\left(\langle \v_1,\u\rangle-\frac{l(\u)}{\o}\right)<\dim\MM^s_\omega(\v).\\
\end{split}
\end{align}
The last inequality follows since $\frac{l(\u)}{\o}\leq 1\leq r\leq \langle \v_1,\u\rangle$ and equality cannot hold for each inequality.  Indeed, if $r=1$, then we saw above that $n_0=1$, so $l(\u)=1$ as well.

(4) Consider the closed sublocus $\MM_\omega^{\mu s}(\v)\backslash \MM^{\mu s,lf}_\omega(\v)$ of non-locally free sheaves.  From the double-dual short exact sequence for $E\in\MM^{\mu s}_\omega(\v)$, $$0\to E\to E^{\vee\vee}\to E^{\vee\vee}/E\to 0,$$ we see that for a fixed $E^{\vee\vee}$, all such $E$ are parametrized by $\Quot^k_{E^{\vee\vee}}$ of dimension $(nr+1)k$ where $k=\chi(E^{\vee\vee}/E)$.  As $\v(E^{\vee\vee})=\v+(0,0,k)=(nr,nD,s+k)$, we get 
\begin{equation}\label{non-locally free}
\dim\MM^{nlf}_\omega(\v)\leq\max_{k>0}\{\dim\MM^{\mu s}_\omega(nr,nD,s+k)+(nr+1)k\}.
\end{equation}
As $E^{\vee\vee}$ is $\mu$-stable, it follows that  $\v(E^{\vee\vee})^2\geq 0$, so $k\leq\frac{\v^2}{2nr}$ with equality only if $\v(E^{\vee\vee})=n'\u$, $n=n'n_0$, and $n'l(\u)\mid\o$.  As long as $k<\frac{\v^2}{2nr}$, $\v(E^{\vee\vee})^2>0$, so $\dim\MM^{\mu s}_\omega(nr,nD,s+k)=\v(E^{\vee\vee})^2=\v^2-2nrk$ and thus $$\dim\MM^{\mu s}_\omega(nr,nD,s+k)+(nr+1)k=\v^2-2nrk+(nr+1)k=\v^2-(nr-1)k<\dim\MM^{\mu s}_\omega(\v),$$ since we assumed that $nr>1$ and $k>0$.  If $k=\frac{\v^2}{2nr}$, then \begin{equation}\label{worst case 1}\dim\MM^{\mu s}_\omega(nr,nD,s+k)\leq \frac{n'l(\u)}{\o}\leq 1.\end{equation}   Thus \begin{align}\label{worst case 2}
\begin{split}
    \dim\MM^{\mu s}_\omega(\v(E^{\vee\vee}))+(rn+1)k&\leq 1+2rnk+k-rnk=\v^2-((nr-1)k-1)\\
    &=\dim\MM^{\mu s}_\omega(\v)-((nr-1)k-1).
\end{split}\end{align}  
Now $$(nr-1)k-1>0,$$ unless $nr=2$ and $k=1$, and $nr=2$ implies $(n,r)=(2,1)$ or $(1,2)$.  
But if $r=1$, then $n_0=1=l(\u)$ and we only have equality achieved in \eqref{worst case 1} if $\o=2$, which is the case excluded.  
If $\o\neq 2$ then we have strict inequality in \eqref{worst case 1} and thus in \eqref{worst case 2}.  
If instead $n=1$, so that $n'=n_0=1$ as well, and $r=2$, then we see that $l(\u)\mid 2$, so we have strict inequality in \eqref{worst case 1} and \eqref{worst case 2} if $\o>2$.  
If $\o=2$, then we only get equality if $l(\u)=2$ as well.  As $\u=(2,aA_0+bB_0,\frac{ab}{2})$ is primitive with $l(\u)=2$, one and only one of $a,b$ is even, so up to twisting by an appropriate line bundle we see that $\u=(2,B_0,0)$ and $S$ is of Type 1, which was also excluded.  
It follows that in each case considered, every component of $\MM_\omega^{\mu s}(\v)$ contains a locally free sheaf, as required.
\end{proof}

\begin{Rem}\label{half-half} In the two cases excluded in Proposition \ref{slope stability}, the proof above gives an irreducible component consisting of non-locally free sheaves.  Indeed, if $\o=2$, then there is a two-dimensional component of $M^s_\omega(\u)$, $\u=(2,0,0)$ or $(2,B_0,0)$ (for Type 1), coming from the pushforward of non-invariant line bundles in $\Pic^0(X)$.  For each $F\in M^s_\omega(\u)$, we take  all kernels $E$ fitting into $$0\to E\to F\to \C_x\to 0,$$ which are parametrized by $\Quot^1_F$ of dimension 3.  Varying $F$ in this two-dimensional component then gives a five dimensional component of $M^s_\omega(\u-(0,0,1))$ consisting of non-locally free sheaves, as required.  

In each case, there are also components of the moduli space which do contain locally free sheaves.  Indeed, let $\w=(1,0,0)$ or $(1,B_0,0)$, i.e. the Mukai vector of $\OO(L)$ or $\OO(B_0+L)$ for $L\in\Pic^0(S)$.  Note that for a point $x\in S$, $\langle (1,0,0),\v(I_x(M))\rangle=1$ for $M\in M^s_\omega(\w)$ in either case.  Also $\Hom(I_x(L),\OO)=\Ext^2(I_x(L),\OO)=0$ for $L\neq 0$, and $\Hom(I_x(B_0+L),\OO)=0$ for all $x\in S,L\in\Pic^0(S)$ while $\Ext^2(I_x(B_0+L),\OO)=\Hom(\OO,I_x(B_0+L+K_S))=0$ except for when $B_0+L+K_S$ is effective, which only happens for at most one $L\in\Pic^0(S)$, and $x\in B_0+L+K_S$.  Thus for generic $M\in M^s_\omega(\w)$ and generic $x\in S$, $\Ext^1(I_x(M),\OO)=\C$.  Taking the unique non-trivial extension $$0\to\OO\to E\to I_x(M)\to 0,$$ we must have $E$ locally free, as $(x,\OO(M+K_S))$ satisfies the Cayley-Bacharach property since $H^0(M+K_S)=0$ for generic $M$ (see \cite[Theorem 5.1.1]{HL10}).  Thus there are indeed $\mu$-stable locally free sheaves by openness of both of these properties.
\end{Rem}
\section{Making the problem finite}
\cref{slope stability} is a powerful tool for constructing $\mu$-stable sheaves.  Indeed, if we can show that $M^{\mu ss}_\omega(\v)\neq\varnothing$, then every component of $M^{\mu ss}_\omega(\v)$ contains a $\mu$-stable sheaf (which can be taken to be locally free in rank at least two, except for certain components in the two exceptional cases), so $M^{\mu s}_\omega(\v)\ne\varnothing$.  Writing $\v=(nr,nD,s)$ with $\gcd(r,D)=1$, it follows from $\v^2\geq 0$ that to prove $M^{\mu ss}_\omega(\v)\ne\varnothing$ it suffices to prove $M^{\mu ss}_\omega(\v')\ne\varnothing$, where $\v'=(nr,nD,s')$ and $s'=\left\lfloor\frac{nD^2}{2r}\right\rfloor$.  Indeed, for $F\in M^{\mu ss}_\omega(\v')$ stable, any $E$ fitting into the short exact sequence $$0\to E\to F\to Q\to 0$$ is in $M^{\mu ss}_\omega(\v)$, where $Q\in\Quot_F^{s'-s}$ is a 0-dimensional torsion sheaf of length $s'-s$.  To summarize, for fixed $(r,D)$ with $\gcd(r,D)=1$, it suffices to prove $M_\omega^{\mu ss}\left(nr,nD,\left\lfloor\frac{nD^2}{2r}\right\rfloor\right)\ne\varnothing$ for each $n\geq 1$.

Using isotropic vectors we can reduce our work to a finite list of $n$.  Recall that $\u=(n_0r,n_0D,s_0)$ was the unique primitive isotropic Mukai vector in the series.  Let us write $n=n_0n'+q$ for $0\leq q<n_0$.  Then 
$$\v'=\left(nr,nD,\left\lfloor\frac{nD^2}{2r}\right\rfloor\right)=n'\u+\left(qr,qD,\left\lfloor\frac{qD^2}{2r}\right\rfloor\right).$$
If we can show that $M_\omega^{\mu ss}\left(qr,qD,\left\lfloor\frac{qD^2}{2r}\right\rfloor\right)\ne\varnothing$ and $M^{\mu ss}_\omega(\u)\ne\varnothing$, then taking 
$$F\in M^{\mu ss}_\omega\left(qr,qD,\left\lfloor\frac{qD^2}{2r}\right\rfloor\right) \mbox{ and }G\in M^{\mu ss}_\omega(\u),$$
we get $G^{n'}\oplus F\in M_\omega^{\mu ss}(\v')$ as required.
For fixed $(r,D)$, this reduces the construction problem to showing that $M_\omega^{\mu ss}\left(qr,qD,\left\lfloor\frac{qD^2}{2r}\right\rfloor\right)\ne\varnothing$ for $0<q\leq n_0$.

\part{Enter Bridgeland stability and Fourier-Mukai transforms}\label{Part:BridgelandStability}
To go deeper, we will employ the full power of Fourier-Mukai transforms made possible by using Bridgeland stability conditions.    But first we provide a quick introduction.
\section{Bridgeland stability conditions: a brief review}

We begin by giving a brief review of Bridgeland's definition of a stability condition on the derived category $\Db(Y)$ of a smooth projective variety $Y$.

Recall that a Bridgeland stability condition $\sigma$ on $\Db(Y)$ consists of a pair $(Z,\AA)$, where $Z:K_{\num}(Y)\to\C$ is a group homomorphism called the central charge and $\AA\subset\Db(Y)$ is the heart of a bounded $t$-structure on $\Db(Y)$, satisfying the following properties:
\begin{enumerate}
    \item{(\emph{positivity})} For any $0\neq E\in\AA$, $\Im Z(E)\geq 0$ and $\Im Z(E)=0$ implies $\Re Z(E)\leq 0$.  This condition allows us to define a notion of slope-stability on $\AA$ via the slope $\mu(E)=\frac{-\Re Z(E)}{\Im Z(E)}$, where we set $\mu(E)=\infty$ if $\Im Z(E)=0$.  An object $0\neq E\in\AA$ is called $\sigma$-(semi)stable if $\mu(F)(\leq)<\mu(E)$ for all $F\subset E$ in the abelian category $\AA$.
    \item{(\emph{HN-filtrations})} Given this definition of slope-stability, every object $E\in\AA$ admits a Harder-Narasimhan filtration $0=E_0\subset E_1\subset\dots\subset E_n=E$ such that $E_i/E_{i-1}$ is $\sigma$-semistable and $\mu(E_i/E_{i-1})>\mu(E_{i+1}/E_i)$ for all $i$.
    \item{(\emph{support})} There is a constant $C>0$ such that for all $\sigma$-semistable $E\in\AA$, we have $$\lVert E\rVert\leq C|Z(E)|,$$ where $\lVert{*}\rVert$ is a fixed norm on $K_{\num}(Y)_\R$. 
\end{enumerate}

It is not difficult to show that the above Harder-Narasimhan filtrations are in fact unique.  Moreover, the above properties guarantee that every $\sigma$-semistable object $E\in\AA$ of admits a Jordan-H\"{o}lder filtration $0=E_0\subset E_1\subset\dots\subset E_n=E$ where every $E_i/E_{i-1}$ is  $\sigma$-stable of the same slope. Unlike Harder-Narasimhan filtrations, however, Jordan-H\"{o}lder filtrations are not unique.  Nevertheless, the associated graded object $$\gr(E)=\bigoplus_{i=1}^n E_i/E_{i-1}$$ is well-defined and independent of the specific filtration.  We say that two $\sigma$-semistable objects $E$ and $E'$ are $S$-equivalent if $\gr(E)\cong\gr(E')$.  

%Given $(Z,\AA)$ as above, we extend the notion of stability to $\Db(Y)$ as follows: for $\phi\in(0,1]$, we let $\PP(\phi)\subset\AA$ be the full subcategory of $\sigma$-semistable objects $E$ with $Z(E)\in\R_{>0}e^{i\phi\pi}$; for general $\phi$, $\PP(\phi)$ is defined by $\PP(\phi+n)=\PP(\phi)[n]$.  Each subcategory $\PP(\phi)$ is extension-closed and abelian of finite length.  Its nonzero objects are called $\sigma$-semistable of phase $\phi$, and its simple objects are called $\sigma$-stable.  Similarly to the case of objects in $\AA$, each object $E\in\Db(Y)$ admits a Harder-Narasimhan filtration where the inclusions $E_{i-1}\subset E_i$ are replaced by exact triangles $E_{i-1}\to E_i\tp A_i$, where the $A_i$'s are $\sigma$-semistable of decreasing phases $\phi_i$.  

Bridgeland's main theorem in \cite{Bri98} says that the set of all Bridgeland stability conditions on $\Db(Y)$, denoted by $\Stab(Y)$, admits the structure of a complex manifold.  More precisely:
\begin{Thm}[Bridgeland]\label{thm:BridgelandMain}
The map $$\ZZ:\Stab(Y)\to\Hom(K_{\num}(Y),\C), \qquad (Z,\AA)\mapsto Z,$$
is a local homeomorphism.  In particular, $\Stab(Y)$ is a complex manifold of dimension equal to the rank of $K_{\num}(Y)$.
\end{Thm}
The topology on $\Stab(Y)$ in the above theorem comes from a natural metric, and the theorem essentially says that $(Z,\AA)$ can be deformed uniquely given a small deformation of its central charge $Z$.  We will also have need for a dual forgetful map $$\eta:\Stab(Y)\to K_{\num}(Y)_\C, \qquad (Z,\AA)\mapsto \Omega,$$ where $\Omega$ is the unique class in $K_{\num}(Y)_\C$ such that $Z(E)=\langle \Omega,E\rangle$ for all $E\in\AA$.  

It is natural to ask how the set of $\sigma$-semistable objects of a given class $\v\in K_{\num}(Y)$ changes as $\sigma$ varies in $\Stab(Y)$.  The following result was proved in \cite[Section 9]{Bri08},\cite[Proposition 3.3]{BM11}, and \cite[Proposition 2.8]{Tod08}:
\begin{Prop}\label{prop:wall-and-chamber}
There exists a locally finite set of walls (real codimension one submanifolds with boundary) in $\Stab(Y)$, depending only on $\v$, with the following properties:
\begin{enumerate}
    \item When $\sigma$ varies within a chamber (a connected component of the complement of the walls), the sets of $\sigma$-semistable and $\sigma$-stable objects of class $\v$ do not change.
    \item When $\sigma$ lies on a single wall $W\subset\Stab(Y)$, then there is a $\sigma$-semistable object of class $\v$ that is unstable in one of the adjacent chambers, and semistable in the other adjacent chamber.
\end{enumerate}
\end{Prop}
From the construction of these walls, it is clear that if $\v$ is primitive, then $\sigma$ lies on a wall if and only if there exists a strictly $\sigma$-semistable object of class $\v$.  Moreover, the Jordan-H\"{o}lder filtration of $\sigma$-semistable objects does not change when $\sigma$ remains inside a chamber.

Finally, we note that there are two groups actions on $\Stab(Y)$ \cite[Lemma 8.2]{Bri07}: on the left, the group of autoequivalences $\Aut(\Db(Y))$ acts via $$F(Z,\AA)=(Z\circ F_*^{-1},F(\AA)),$$ for $F\in\Aut(\Db(Y))$ and $F_*$ is the induced automorphism on $K_{\num}(Y)$.  On the right, the universal cover $\wGL2$ of the group $\GL_2^+(\R)$ of matrices with positive determinant acts on the right as a lift of the action of $\GL_2^+(\R)$ on $\Hom(K_{\num}(Y),\C)\cong\Hom(K_{\num}(Y),\R^2)$.

\subsection{Review: Stability conditions via tilting slope-stability}
Now we recall the construction from \cite{AB13,Bri08} of Bridgeland stability conditions on a smooth porjective surface $Y$ over $\C$ .  As in the previous subsection, let $\omega,\beta\in\NS(Y)_{\Q}$ with $\omega$ ample.  Let $\TT(\omega,\beta) \subset \Coh(Y)$ be the subcategory of torsion sheaves and torsion-free sheaves whose HN-filtration factors (with respect to $\mu_{\omega,\beta}$-stability)
have $\mu_{\omega, \beta} > 0$, and $\FF(\omega,\beta)$ the subcategory of torsion-free sheaves with
HN-filtration factors satisfying $\mu_{\omega, \beta} \le 0$.
Next, consider the abelian category
\begin{equation*} %\label{eq:AK3}
\AA(\omega,\beta):=\left\{E\in\Db(Y):\begin{array}{l}
\bullet\;\;\HH^p(E)=0\mbox{ for }p\not\in\{-1,0\},\\\bullet\;\;
\HH^{-1}(E)\in\FF(\omega,\beta),\\\bullet\;\;\HH^0(E)\in\TT(\omega,\beta)\end{array}\right\}
\end{equation*}
and the $\C$-linear map
\begin{equation} \label{eq:ZK3}
Z_{\omega,\beta} \colon K_{\num}(Y)\to\C,\qquad E\mapsto\langle\exp{(\beta+\sqrt{-1}\omega)},v(E)\rangle.
\end{equation}

For $Y$ an abelian surface or a bielliptic surface, it follows from \cite[Lemma 6.2, Prop.\ 7.1]{Bri08} or \cite[Corollary 2.1]{AB13},
that the pair $\sigma_{\omega,\beta}=(Z_{\omega,\beta},\AA(\omega,\beta))$ defines a stability condition.
For objects $E \in \AA(\omega, \beta)$, we will denote their phase with respect to $\sigma_{\omega,
\beta}$ by $\phi_{\omega, \beta}(E) = \phi(Z_{\omega,\beta}(E)) \in (0, 1]$.
By using the support property, as proved in \cite[Proposition 10.3]{Bri08}, we can extend the above definition to stability conditions $\sigma_{\omega,\beta}$, for $\omega,\beta\in\NS(Y)_\R$.  We denote by $\Stabd(Y)$ the connected component of $\Stab(Y)$ containing the $\sigma_{\omega,\beta}$.  In fact, up to the action of $\wGL2$, every stability condition in $\Stabd(Y)$ is equal to $\sigma_{\omega,\beta}$ for some $\omega,\beta\in\NS(Y)_\Q$.  The component $\Stabd(Y)$ can also be characterized more intrinsically by two properties.  The first is that for any $\sigma\in\Stabd(Y)$, skyscraper sheaves of points on $Y$ are all $\sigma$-stable of the same slope.  The second is that for all $\sigma\in\Stabd(Y)$, the real and imaginary parts of $\eta(\sigma)$ span a positive definite 2-plane.
\begin{Rem}\label{rem:ObjectsOfInfiniteSlope}
It will be important later to describe the set $$\left\{E\in\AA(\omega,\beta)\colon \Im Z_{\omega,\beta}(E)=0\right\}.$$ Considering the short exact sequence in $\AA(\omega,\beta)$, $$0\to\HH^{-1}(E)[1]\to E\to\HH^0(E)\to 0,$$ and using the definitions of $\TT(\omega,\beta)$ and $\FF(\omega,\beta)$, it is not difficult to show that $\Im Z_{\omega,\beta}(E)=0$ implies that $\HH^{-1}(E)$ is a $\mu_{\omega,\beta}$-semistable sheaf with $\mu_{\omega,\beta}=0$ and $\HH^0(E)$ is a $0$-dimensional torsion sheaf.  
\end{Rem}
Already from the definition of $\AA(\omega,\beta)$, it is clear that there is \emph{some} relation between the its objects and slope stability.  By using the central charge $Z_{\omega,\beta}$ which detects the \textbf{entire} Chern character, including $\ch_2$, we get a very strong connection back to classical stability with the following result:
\begin{Thm}\label{Thm:GiesekerChamber}\cite[Proposition 14.2]{Bri08}
Let $\omega,\beta\in\NS(Y)_\Q$ with $\omega$ ample.  Let $\v\in K_{\num}(Y)$ be the class of a $\beta$-twisted $\omega$-semistable sheaf.  Then there is a chamber $\GG\subset\Stabd(Y)$ for $\v$ such that for any $\sigma\in \GG$, $E\in\Db(Y)$ with $\v(E)=\v$ is $\sigma$-(semi)stable iff $E=F[k]$ for a $\beta$-twisted $\omega$-(semi)stable sheaf $F$ of class $\v$.
\end{Thm}

\subsection{Review: Inducing stability conditions}
Although stability conditions can be constructed on the bielliptic surface $S$ intrinsically as in the previous subsection, it will also be useful to consider them as being \emph{induced} from the canonical cover $X$ by the group action $G'$.  Recall that $\pi:X\to S$ denotes covering map of a bielliptic surface $S$ by its abelian canonical cover $X$.  Via the fixed-point free action of $G'$, $\Coh(S)$ is naturally isomorphic to the category of coherent $G'$-sheaves on $X$, $\Coh_{G'}(X)$, where $G'=G/H$, thus giving a natural equivalence of $\Db(S)$ with $\Db_{G'}(X)$.  We make this identification implicitly below.  

In \cite{MMS09} the authors construct two faithful adjoint functors $$\For:\Db_{G'}(X)\rightarrow \Db(X),$$ which forgets the $G'$-sheaf structure, and $$\Inf:\Db(X)\rightarrow \Db_{G'}(X),\qquad\Inf(E):=\oplus_{g\in G'} g^*E.$$  Under the above identifications we have $\For=\pi^*$ and $\Inf=\pi_*$.  Since $G'$ acts on $\Stab(X)$ via the natural action of $\Aut(\Db(X))$ on $\Stab(X)$, we can define $$\Gamma_{X}:=\{\sigma\in \Stab(X):g^*\sigma=\sigma,\text{ for all }g\in G'\}.$$

These functors induce two continuous maps.  The first $(\pi^*)^{-1}:\Gamma_{X}\rightarrow \Stab(\Db(S))$ is given by $Z_{(\pi^*)^{-1}(\sigma)}=Z_{\sigma}\circ \pi^*$ and $\AA_{(\pi^*)^{-1}(\sigma)}=\{E\in\Db(S)\colon \pi^*E\in\AA_{\sigma}\}$, where we use $\pi^*$ also for the morphism between $K$-groups.  The second $(\pi_*)^{-1}:(\pi^*)^{-1}(\Gamma_{X})\rightarrow \Stab(X)$ is defined similarly with $\pi^*$ replaced by $\pi_*$.

We consider the connected component $\Stab^{\dagger}(X)\subset \Stab(X)$ described in the section above.  The following result is relevant to us:

\begin{Thm}\label{induced}\cite[Proposition 3.1]{MMS09} The non-empty subset $\Sigma(S):=(\pi^*)^{-1}(\Gamma_{X}\cap \Stab^{\dagger}(X))$ is open and closed in $\Stab(S)$, and it is embedded into $\Stab^{\dagger}(X)$ as a closed submanifold via the functor $(\pi_*)^{-1}$.  Moreover, the diagram 

\[\begin{CD}
\Gamma_{X}\cap \Stab^{\dagger}(X)@ >(\pi^*)^{-1} >> \Sigma(S)@ >(\pi_*)^{-1} >>\Gamma_{X}\cap\Stab^{\dagger}(X)\\
@VV V  @VV\mathcal Z V  @VV V\\
({K_{\num}(X)}^{G'}_{\C})^{\vee} @>(\pi^*)^{\vee} >>{K_{\num}(S)}_{\C}^{\vee}@>\pi_*^{\vee} >>({K_{\num}(X)}^{G'}_{\C})^{\vee}
\end{CD}\] commutes.

\end{Thm}

In particular, if we choose $\omega,\beta\in\NS(X)^{G'}_{\Q}$, then $\sigma_{\omega,\beta}\in\Gamma_X\cap\Stabd(X)$, and $(\pi^*)^{-1}(\sigma_{\omega,\beta})=\sigma_{\omega',\beta'}$, where $\omega',\beta'\in\NS(S)_{\Q}$ are such that $\pi^*\omega'=\omega$ and $\pi^*\beta'=\beta$.  From Theorem \ref{induced} it follows that $\Sigma(S)=\Stabd(S)$, so we may always relate stability on $S$ to stability on $X$.  We explore this in the next section.

\section{Basic properties of Bridgeland semistable objects on $S$}\label{sec:StableObjects}
We begin by recalling a convenient criterion for determining when an object in $\Db(X)$ descends to $S$ and when an object in $\Db(S)$ is the pushforward of an object from $X$.
\begin{Lem}\cite[Lemma 2.4, Proposition 2.5]{BM01} \label{lem:DescentAndPushforward}
\begin{enumerate}
    \item Let $E\in\Db(S)$.  Then $E\cong\pi_*(F)$ for some $F\in\Db(X)$ if and only if $E\otimes\o\cong E$.
    \item Let $g\in G'$ be a generator.  Then $F\in\Db(X)$ satisfies $F\cong\pi^*E$ for some $E\in\Db(S)$ if and only if $g^*F\cong F$.
\end{enumerate}
\end{Lem}

Now we turn more specifically to the relationship between stable objects on $S$ and those on $X$.  To that end, let $\sigma\in\Stabd(S)$ satisfy $\sigma=(\pi^*)^{-1}(\sigma')$ for some $\sigma'\in\Gamma_X\cap\Stabd(X)$. We first study when two $\sigma$-stable objects pull back to the same object on $X$:
\begin{Lem}\label{lem:PullBackEqual}
If $\sigma$-stable $E_1$ and $E_2$ in $\Db(S)$ satisfy $\pi^*E_1\cong\pi^*E_2$, then $E_1\cong E_2\otimes \o^i$ for some $i\in\Z$.
\end{Lem}
\begin{proof}
Pushing forward the isomorphism $\pi^*E_1\cong\pi^*E_2$, we get from the construction of the canonical cover that  \begin{align*}
    \bigoplus_{i=0}^{\o-1}E_1\otimes\o^i&\cong E_1\otimes\left(\bigoplus_{i=0}^{\o-1}\o^i\right)\cong E_1\otimes\pi_*\OO_X\cong\pi_*\pi^*E_1\\&\cong\pi_*\pi^*E_2\cong E_2\otimes\pi_*\OO_X\cong\bigoplus_{i=0}^{\o-1}E_2\otimes\o^i.
\end{align*}
Applying $\Hom(E_1,\blank)$, we see that for some $i$, we must have $\Hom(E_1,E_2\otimes\o^i)\neq 0$.  But as $E_1$ and $E_2\otimes\o^i$ are both $\sigma$-stable of the same slope, it follows that $E_1\cong E_2\otimes\o^i$, as required.
\end{proof}

Of course, it is important to understand when a $\sigma$-stable object pulls back to a $\sigma'$-stable object, the subject of our next lemma:

\begin{Lem}\label{lem:WhenIsPullbackStable} Denote by $g$ a generator of $G'$.  If $E\in\Db(S)$ is $\sigma$-semistable, then $\pi^*E$ is $\sigma'$-semistable.  Moreover, for a $\sigma$-stable object $E$, $\pi^*E$ is $\sigma'$-stable of the same slope unless $E\cong E\otimes\o^{\o/d}$ for some $1<d\leq\o$, in which case $\pi^*E\cong \bigoplus^{d-1}_{k=0} (g^k)^*F$, where $F$ is $\sigma'$-stable and the $(g^k)^*F$ are distinct for $k=0,...,d-1$.  
\end{Lem}

\begin{proof} We show first that $\pi^*E$ is $\sigma'$-semistable if $E$ is $\sigma$-semistable.  If not, we can take the maximal destabilizing subobject $F\subset\pi^*E$ with $\mu_{\sigma'}(F)>\mu_{\sigma'}(\pi^*E)$.  But by its uniqueness, $F$ must be fixed by the action of $G'$, so $F\cong\pi^*E'$ for some $E'\subset E$ by Lemma \ref{lem:DescentAndPushforward}.  But then $\mu_{\sigma}(E')>\mu_{\sigma}(E)$, contradicting the $\sigma$-semistability of $E$.

Now suppose that $E$ is $\sigma$-stable and $\pi^*E$ is strictly $\sigma'$-semistable.  Let $F\subset \pi^*E$ be a proper nontrivial $\sigma'$-stable subobject of the same slope.  Consider the stabilizer of $F$, $\stab(F)\subset G'$.  Since $G'$ is cyclic, $\stab(F)=\langle g^d\rangle$ for some $d\mid\o$.  If $d=1$, then by Lemma \ref{lem:DescentAndPushforward} there is a nontrivial $\sigma$-stable object $E'\subset E$ with $\pi^*E'=F$, so $\mu_\sigma(E')=\mu_\sigma(E)$, contradicting the stability of $E$.

Otherwise, $d>1$, and we note that $(g^k)^*F\subset\pi^*E$ for all $k$.  So consider $\bigoplus_{k=0}^{d-1} (g^k)^*F$.  Then by induction and the $\sigma'$-stability of $(g^k)^*F$, it follows that $\tilde{F}:=\bigoplus_{k=0}^{d-1} (g^k)^*F\subset \pi^*E$.  Indeed, the base case is clear, so suppose that we have already shown that $\bigoplus_{k=0}^{l} (g^k)^*F\subset \pi^*E$.  Then we have the short exact sequence $$0\to \bigoplus_{k=0}^{l} (g^k)^*F\cap (g^{l+1})^*F\to\bigoplus_{k=0}^{l+1} (g^k)^*F\to\bigoplus_{k=0}^{l} (g^k)^*F+(g^{l+1})^*F\to 0,$$ from which it follows that $$\mu_{\sigma'}\left(\bigoplus_{k=0}^{l} (g^k)^*F\cap (g^{l+1})^*F\right)=\mu_{\sigma'}\left(\bigoplus_{k=0}^{l+1} (g^k)^*F\right)=\mu_{\sigma'}\left(\bigoplus_{k=0}^{l} (g^k)^*F+(g^{l+1})^*F\right)=\phi,$$ by the see-saw principle, $\sigma'$-semistability of $\bigoplus_{k=0}^{l+1} (g^k)^*F$ and $\pi^*E$, and the inclusion $\bigoplus_{k=0}^{l} (g^k)^*F+(g^{l+1})^*F\subset\pi^*E$.  But then $\bigoplus_{k=0}^{l} (g^k)^*F\cap (g^{l+1})^*F\subset (g^{l+1})^*F$ and has the same slope, so it must either be 0 or all of $(g^{l+1})^*F$.  The latter is impossible, so we get that the sum $\bigoplus_{k=0}^{l} (g^k)^*F+(g^{l+1})^*F$ is direct as claimed.

Now clearly $\tilde{F}$ satisfies $g^*\tilde{F}\cong\tilde{F}$, so $\tilde{F}\cong\pi^*F'$ for some $F'\subset E$.  If $\tilde{F}\neq\pi^*E$, then $F'$ is a proper subobject of $E$ of the same slope, contradicting the $\sigma$-stability of $E$.  Thus $\bigoplus_{k=0}^{d-1} (g^k)^*F\cong\pi^*E$.  Pushing forward gives $$\bigoplus_{i=0}^{\o-1} E\otimes\o^i\cong\pi_*\pi^*E\cong \bigoplus_{k=0}^{d-1} \pi_*((g^k)^*F)\cong\pi_*(F)^{\oplus d},$$ since $\pi_*(g^*F)\cong\pi_*(F)$.  Now we take $\Hom(E,\blank)$ to obtain 
\begin{align*}\bigoplus_{i=0}^{\o-1} \Hom(E,E\otimes\o^i)&=\Hom(E,\pi_*(F))^{\oplus d}\\
&=\Hom(\pi^*E,F)^{\oplus d}\\
&=\bigoplus_{k=0}^{d-1} \Hom((g^k)^*F,F)^{\oplus d},\end{align*} 
where we have used the fact that $\pi^*$ and $\pi_*$ are adjoint functors.  Of course, $\Hom(E,E)=\Hom(F,F)=\C$ and $\Hom(F,(g^k)^*F)=0$ for $0<k<d$ from stability and the fact that $(g^k)^*F$ are non-isomorphic $\sigma'$-stable objects of the same slope.  It follows that for $d$ values of $i$ (including $i=0$) we have $\Hom(E,E\otimes\o^i)=\C$.  Thus $E\otimes\o^{\o/d}\cong E$.

Conversely, suppose that $E\cong E\otimes\o^{\o/d}$.  Then 
\begin{align*}\Hom(\pi^*E,\pi^*E)
&=\Hom(E,\pi_*\pi^*E)=\Hom(E,\bigoplus_{i=0}^{\o-1} E\otimes\o^i)\\
&=\bigoplus_{i=0}^{\o-1}\Hom(E,E\otimes\o^i)=\bigoplus_{k=0}^{d-1}\Hom(E,E\otimes\o^{k\frac{\o}{d}})\\
&=\C^d,
\end{align*}
so that $\pi^*E$ cannot be stable.
\end{proof}

It is also important to know when $\sigma'$-stable $F\in\Db(X)$ pushes forward to a $\sigma$-stable object.  The proof of the following lemma is precisely dual to that of Lemma \ref{lem:WhenIsPullbackStable}, with $\pi^*$ replaced by $\pi_*$, so we omit the details.
\begin{Lem}\label{lem:WhenIsPushForwardStable}
Let $g\in G'$ be a generator.  If $F\in\Db(X)$ is $\sigma'$-semistable, then $\pi_*(F)$ is $\sigma$-semistable.  Moreover, for a $\sigma'$-stable object $F$, $\pi_*(F)$ is $\sigma$-stable unless $(g^{\o/d})^*F\cong F$ for some $1<d\leq\o$, in which case $\pi_*(F)\cong\bigoplus_{k=0}^{d-1}E\otimes\o^k$, where $E$ is $\sigma$-stable of the same slope and the $E\otimes\o^k$ are distinct for $k=0,\dots,d-1$.
\end{Lem}

We end this section with a quick observation that narrows the scope of those Mukai vectors that correspond to $\sigma$-stable objects.
\begin{Lem}\label{lem:MukaiVectorsOfStableObjects}
Let $\sigma\in\Stabd(S)$ and $E\in\Db(S)$ a $\sigma$-stable object.  Then $\v(E)^2\geq 0$.
\end{Lem}
\begin{proof}
By stability and Serre duality, $\hom(E,E)=1$ and $\ext^2(E,E)=\hom(E,E\otimes\o)\leq 1$.  By Riemann-Roch it follows that $$\v(E)^2=-\chi(E,E)=\ext^1(E,E)-\hom(E,E)-\ext^2(E,E)\geq -2,$$ with equality only if $E$ is a spherical object.  But by \cite[Proposition 5.5]{Sos13}, $S$ does not admit any spherical or exceptional objects.  As $\Hal(S,\Z)$ is an even lattice, $\v(E)^2=-1$ is an impossibility as well, so the  lemma follows.
\end{proof}
\section{Projective coarse moduli spaces for Bridgeland stability on $S$}\label{sec: coarse moduli}
We turn to our first main result in this part.  Namely, we show that there exist projective coarse moduli spaces parametrizing $S$-equivalence classes of $\sigma$-semistable objects of Mukai vector $\v\in\Hal(S,\Z)$ for generic $\sigma\in\Stabd(S)$.  We begin by defining the moduli stack of $\sigma$-semistable objects and then move on to studying the special case of isotropic vectors $\v$ with $l(\v)=\o$.  We conclude this section by proving our first main theorem that these moduli stacks admit projective coarse moduli spaces for $\sigma$ generic with respect to $\v$.
\subsection{Moduli stacks of Bridgeland semistable objects on $S$}
Fix a smooth projective variety $Y$, $\sigma=(Z,\mathcal A)\in \Stabd(Y)$ and $\v\in H_{\alg}^*(Y,\Z)$.  

\begin{Def} 
Define $M_{\sigma,Y}(\v)$ to be the set of $\sigma$-semistable objects $E\in\AA$ of Mukai vector $\v$, and $\MM_{\sigma,Y}(\v)$ to be the 2-functor $$\MM_{\sigma,Y}(\v)\colon(\text{Sch}/\C)\rightarrow (\text{groupoids}),$$ 
which sends a $\C$-scheme $T$ to the groupoid $\MM_{\sigma,Y}(\v)(T)$ whose objects consist of $\mathcal E\in \mathrm{D}_{T\text{-}\mathrm{perf}}(T\times Y)$ satisfying $$\mathcal E_t\in M_{\sigma,Y}(\v)\text{ for all }t\in T.$$
\end{Def}
For smooth projective surfaces, we have the following result about the moduli stack:
\begin{Thm}\cite{Tod08,AlperHLH:ExistenceofModuliSpaces}\label{thm:ModuliStacksAreArtin} Let $Y$ be a smooth projective surface.  For any $\v\in\Hal(Y,\Z)$ and $\sigma\in\Stabd(Y)$, $\MM_{\sigma,Y}(\v)$ is an Artin stack of finite type over $\C$.   Moreover, $\MM_{\sigma,Y}(\v)$ admits a proper good moduli space over $\C$.
\end{Thm}

Although the first statement of Theorem \ref{thm:ModuliStacksAreArtin} was originally proved for K3 surfaces in \cite{Tod08}, its proof goes through for any smooth projective surface.  The second statement follows from the more recent groundbreaking results in \cite{AlperHLH:ExistenceofModuliSpaces}.  In particular, there are proper algebraic spaces parametrizing $S$-equivalence classes of $\sigma$-semistable objects of Mukai vector $\v$.  In specific cases, these proper algebraic spaces are known to be projective varieties.  Indeed, this is known for K3 surfaces \cite{BM14a}, abelian surfaces \cite{MYY14}, and Enriques surfaces \cite{Nue14b}.  We turn now to showing this for bielliptic surfaces as well.

\subsection{Projective coarse moduli spaces isomorphic to $S$}
In this subsection, we show that in certain cases, the coarse moduli space for these moduli stacks is in fact isomorphic to $S$ itself, and thus is indeed projective.

\begin{Lem}\label{lem:moduli spaces of fully induced isotropic vectors}
Let $S$ be a bielliptic surface and $\sigma\in\Stabd(S)$.  If $\v\in\Hal(S,\Z)$ is a primitive isotropic vector with $l(\v)=\o$, then the Fourier-Mukai functor $\Phi_{\EE}$ with kernel given by the universal family $\EE$ induces an isomorphism from the coarse moduli space $M_{\sigma,S}(\v)\isomor S$ such that $\Phi_\EE((0,0,1))=\v$.  
\end{Lem}
\begin{proof}
Let $\sigma'\in\Stabd(X)$ be such that $\sigma=(\pi^*)^{-1}(\sigma')$ and write $\pi^*\v=\o \w$, which is possible by Lemma \ref{primitive}.  Notice that the push-pull formula gives $$\o \v=\pi_*\pi^*\v=\o\pi_*(\w),$$ so $\pi_*(\w)=\v$ in this case.  As there are no walls for $\w$ by \cite[Section 15]{Bri08}, there exists a projective coarse moduli space $M_{\sigma',X}(\w)$ for $\MM_{\sigma',X}(\w)$, which is irreducible and isomorphic to an abelian surface by \cite[Proposition 5.16]{MYY14}.  Let $g$ be a generator for $G'$.  Any object $F\in M_{\sigma',X}(\w)$ that is fixed by $(g^k)^*$ for some $0<k<\o$ would give a $G'$-invariant object $\bigoplus_{i=0}^{k-1}(g^*)^i F\in M_{\sigma',X}(k\w)$ which must be $\pi^*E$ for some $E\in M_{\sigma,S}(\v')$.  By the push-pull formula, $$\o \v'=\pi_*\pi^*\v'=\pi_*(k\w)=k\pi_*(\w)=k\v,$$ a contradiction as $\v$ is primitive and $k<\o$.  It follows that $\pi_*(F)$ is $\sigma$-stable for all $F\in M_{\sigma',X}(\w)$ by Lemma \ref{lem:WhenIsPushForwardStable}.  This gives a complete irreducible component of $M_{\sigma,S}(\v)$ of dimension 2, which must be all of $M_{\sigma,S}(\v)$ by Mukai's trick (see \cite[Theorem 6.1.8]{HL10}).  By \cite[Corollary 2.8]{BM01}, if we let $\EE$ be the universal family for $M_{\sigma,S}(\v)$, then the integral functor $$\Phi_\EE:\Db(M_{\sigma,S}(\v))\to\Db(S)$$ is an equivalence, i.e. a Fourier-Mukai transform, and sends $(0,0,1)$, the Mukai vector of the skyscraper sheaf of a point, to $\v$.  By \cite[Proposition 6.2]{BM01}, we get that $M_{\sigma,S}(\v)\cong S$, as claimed.
\end{proof}
Lemma \ref{lem:moduli spaces of fully induced isotropic vectors} turns out to not only be a special case but also a key tool in proving the existence of projective coarse moduli spaces for arbitrary $\v\in\Hal(S,\Z)$ as we see in the next subsection.
\subsection{Projective coarse moduli spaces for arbitrary $\v\in\Hal(S,\Z)$}
Fix a $\v\in\Hal(S,\Z)$, which we may assume satisfies $\v^2\geq 0$ by Lemma \ref{lem:MukaiVectorsOfStableObjects}.  Recall \cref{prop:wall-and-chamber} which guaranteed a wall-and-chamber decomposition of $\Stabd(S)$ such that, in particular, the $\sigma$-stability and semistability of objects of Mukai vector $\v$ remain constant as $\sigma$ varies within a chamber.  We begin this subsection with a lemma that guarantees that we can choose a nearby stability condition in the same chamber which will allow us to use Lemma \ref{lem:moduli spaces of fully induced isotropic vectors}.
\begin{Lem}\label{lem:nearby isotropics 1}
Let $\v$ be a Mukai vector satisfying either $\v^2>0$ or $\v^2=0$ and $l(\v)<\o$, and let $\sigma$ be a stability condition lying inside a chamber $\CC$ of the wall-and-chamber decomposition with respect to $\v$.  Then $\CC$ contains a dense subset of stability conditions $\tau$ for which there exists a primitive isotropic Mukai vector $\w$ such that:
\begin{enumerate}
\item\label{enum:isotropic same ray} $Z_\tau(\w)$ and $Z_\sigma(\v)$ lie on the same ray in the complex plane.
\item\label{enum:isotropic bielliptic} All $\tau$-semistable objects with Mukai vector $\w$ are stable, and $M_{S,\tau}(\w)\cong S$.
\end{enumerate}
\end{Lem}
\begin{proof}
We first show that $\CC$ contains a dense subset of stability conditions $\tau$ satisfying only condition \ref{enum:isotropic same ray}, and then we demonstrate that we can restrict to a subset of these stability conditions that is still dense and also satisfies \ref{enum:isotropic bielliptic}.  By using the action of $\wGL2$, we may assume without loss of generality that $Z_\sigma(\v)=-1$ and restrict our attention to stability conditions $\tau$ with $Z_\tau(\v)=-1$.  

Recall that we can view a central charge $Z_\sigma$  dually as the vector $\eta(\sigma)=a_\sigma+b_\sigma i\in\Hal(S,\C)$ such that $Z_\sigma(\blank)=\langle \blank,\eta(\sigma)\rangle$, in which case the real two-plane spanned by $a_\sigma$ and $b_\sigma$ is positive definite.  Consider the quadric $Q\subset\Hal(S,\R)$ defined by $\w^2=0$.  As the Mukai pairing on $\Hal(S,\R)$ has signature  $(2,\rho(S))=(2,2)$ and $b_\sigma^2>0$, the codimension 1 subspace $b_\sigma^\perp=\Im Z_\sigma=0$ has signature $(1,2)$, so there is a real solution $\w_r$ to the pair of equations $\Im Z_\sigma(\w)=0$ and $\w^2=0$.  Since $Q$ certainly has rational points, rational points must be dense in $Q$, so there exists $\w_q\in\Hal(S,\Q)$ arbitrarily close to $\w_r$ with $\w_q^2=0$.  

Suppose that $\v^2>0$.  Then as $\w_q^2=0$, $\v$ and $\w_q$  must be linearly independent, so $\v^\perp\cap \w_q^\perp$ (as well as $\v^\perp\cap \w_r^\perp$) must have codimension 2.    Letting $\w_q$ be sufficiently close to $\w_r$, we may choose $b_\tau\in \v^\perp\cap \w_q^\perp$ sufficiently close to $b_\sigma\in \v^\perp\cap \w_r^\perp$.  Then there exists $\tau=(Z_\tau,\AA_\tau)$ with $\eta(\tau)=a_\sigma+b_\tau i$ close to $\sigma$ such that $\Im Z_\tau(\v)=\Im Z_\tau(\w_q)=0$ and $\Re Z_\tau=\Re Z_\sigma$.  Replacing $\w_q$ by the unique primitive integral class $\w\in\Q\cdot \w_q$ with $\Re Z_\tau(\w)<0$ finishes the proof of condition \ref{enum:isotropic same ray}.  By Lemma \ref{lem:nearby isotropics 2} below, we may assume that $\w$ in fact satisfies $l(\w)=\o$.  

If instead $\v^2=0$ and $l(\v)<\o$, then $\w_q\in\Hal(S,\Z)$ with $l(\w_q)=\o$ is automatically linearly independent from $\v$.  So the same argument goes through to give $\tau$ nearby $\sigma$ such that $\Im Z_\tau(\v)=\Im Z_\tau(\w_q)=0$ and $\Re Z_\tau=\Re Z_\sigma$.

Condition \ref{enum:isotropic bielliptic} now follows from Lemma \ref{lem:moduli spaces of fully induced isotropic vectors} and its proof.
\end{proof}
\begin{Lem}\label{lem:nearby isotropics 2}
For any primitive isotropic Mukai vector $\w=(r,D,\frac{D^2}{2r})$, there is an arbitrarily close isotropic Mukai vector $\v_0$ with $l(\v_0)=\o$.  Here, distance is measured in the hyperbolic metric on the space of rays in $\Hal(S,\Q)$.
\end{Lem}
\begin{proof}
First we observe that $\w=re^{D/r}$.  Writing $d_0=\gcd(r,\o)$, we note that for any $n\in\Z$ $$\gcd\left(\frac{\o}{d_0}nr+1,\o\right)=1.$$  Thus, if we write $\frac{n\frac{\o}{d_0}D}{\frac{\o}{d_0}nr+1}$ in lowest terms, the denominator is still relatively prime to $\o$, and $\frac{n\frac{\o}{d_0}D}{\frac{\o}{d_0}nr+1}\to D/r$ in $\NS(S)_\Q$ as $n\to \infty$.  Indeed, $$\frac{D}{r}-\frac{n\frac{\o}{d_0}D}{\frac{\o}{d_0}nr+1}=\frac{D}{r(\frac{\o}{d_0}nr+1)}\to 0$$ as $n\to\infty$.  Thus we may assume from the outset that $\gcd(r,\o)=1$ and $\frac{D}{r}$ is written in lowest terms.

Now write $\o=2^x 3^y$ (according to the prime decomposition of $\o$ as in Table \ref{table:cohomology}), and write $D=2^i3^jm_0 D_0$, where $\gcd(6,m_0)=1$ and $D_0=a_0A_0+b_0B_0$ is primitive.  Furthermore, we write $a_0b_0=2^k3^l h$ with $\gcd(6,h)=1$.  Pick a positive integer $\tilde{r}$ such that $r\tilde{r}\equiv 1\pmod {2^{x+i+k}3^{y+j+l}}$.  Defining $\tilde{n}=2^{x+i+k}3^{y+j+l}n-\tilde{r}$, we have $\gcd(\tilde{n},6)=1$ and $\tilde{n}r+1\equiv 0\pmod{2^{x+i+k}3^{y+j+l}}$, so we may write $\tilde{n}r+1=2^{x+i+k}3^{y+j+l}g$ for some $g\in\Z$.

We observe that, as above, $\frac{\tilde{n}D}{\tilde{n}r+1}\to\frac{D}{r}$ as $\tilde{n}\to\infty$, so we just need to show that the primitive integral generator $\v_0$ of the ray $\Q_+e^{\frac{\tilde{n}D}{\tilde{n}r+1}}$ satisfies $l(\v_0)=\o$.  Writing out $e^{\frac{\tilde{n}D}{\tilde{n}r+1}}$, we get \begin{align*}e^{\frac{\tilde{n}D}{\tilde{n}r+1}}&=\left(1,\frac{\tilde{n}2^i3^jm_0D_0}{2^{x+i+k}3^{y+j+l}g},\frac{\tilde{n}^2 2^{2i}3^{2j}m_0^2a_0b_0}{2^{2x+2i+2k}3^{2y+2j+2l}g^2}\right)\\
&=\left(1,\frac{\tilde{n}m_0D_0}{2^{x+k}3^{y+l}g},\frac{\tilde{n}^2m_0^22^k3^lh}{2^{2x+2k}3^{2y+2l}g^2}\right)\\
&=\left(1,\frac{\tilde{n}m_0D_0}{2^{x+k}3^{y+l}g},\frac{\tilde{n}^2m_0^2h}{2^{2x+k}3^{2y+l}g^2}\right).
\end{align*}
Without loss of generality, we may assume that $\gcd(g,\tilde{n}m_0)=1$, and we let $d=\gcd(g,h)$.  Then $$\v_0=\left(2^{2x+k}3^{2y+l}\frac{g^2}{d},2^x 3^y\frac{g}{d}\tilde{n}m_0D_0,\tilde{n}^2m_0^2\frac{h}{d}\right)$$ is certainly integral, and we claim that $\v_0$ is primitive and $l(\v_0)=\o$.  Indeed, $$\gcd(2^x3^y\frac{g}{d}\tilde{n}m_0,\tilde{n}^2m_0^2\frac{h}{d})=\tilde{n}m_0\mbox{ and }\gcd(2^{2x+k}3^{2y+l}\frac{g^2}{d},\tilde{n}m_0)=1,$$ by construction, so primitivity follows.  As $\o=2^x 3^y$ divides both $$2^{2x+k}3^{2y+l}\frac{g^2}{d}\mbox{ and }2^x 3^y\frac{g}{d}\tilde{n}m_0D_0,$$ but $\gcd(\o,\tilde{n}^2m_0^2\frac{h}{d})=1$, we must have $l(\v_0)=\o$, as claimed.
\end{proof}
We now have the precise tools needed to prove our first main theorem:
\begin{Thm}\label{thm:projective coarse moduli spaces}
Let $S$ be a bielliptic surface and $\v\in\Hal(S,\Z)$ with $\v^2\geq 0$.  Assume that $\sigma\in\Stabd(S)$ is generic with respect to $\v$.  Then there exists a projective variety $M_{\sigma,S}(\v)$ that is a coarse moduli space parametrizing $S$-equivalence classes of $\sigma$-semistable objects with Mukai vector $\v$.
\end{Thm}
\begin{proof}
Let $\CC$ be the chamber for $\v$ that contains $\sigma$, and consider the isotropic vector $\w$ and the stability condition $\tau\in\CC$ from Lemma \ref{lem:nearby isotropics 1}.  Denote by $\Phi:\Db(S)\to\Db(S)$ the autoequivalence satisfying $\Phi(\w)=(0,0,1)$ (the inverse of the autoequivalence from Lemma  \ref{lem:moduli spaces of fully induced isotropic vectors} associated to the universal family on $M_{S,\tau}(\w)$).  Letting $\tau'=\Phi(\tau)\in\Stab(S)$, we see that for every $F\in M_{S,\tau}(\w)$, $\Phi(F)=\OO_s$ for some $s\in S$ and is $\tau'$-stable, so in fact $\tau'\in\Stabd(S)$.  Therefore, up to acting by $\wGL2$, we can assume that $\tau'=\sigma_{\omega,\beta}$ for some $\omega,\beta\in\NS(S)_\Q$ with $\omega$ ample.

Since $Z_\tau(\v)$ and $Z_\tau(\w)$ lie on the same ray in the complex plane, we have $Z_{\omega,\beta}(\Phi(\v))\in\R_{<0}$.  Moreover, $\tau'$ is still generic with respect to $\Phi(\v)$ by the construction of $\tau$.  By Remark \ref{rem:ObjectsOfInfiniteSlope}, for any $E\in\MM_{S,\tau}(\v)(\C)$, $\Phi(E)$ sits in a short exact sequence in $\AA(\omega,\beta)$, $$0\to\HH^{-1}(\Phi(E))[1]\to\Phi(E)\to\HH^0(\Phi(E))\to0,$$ where $\HH^{-1}(\Phi(E))$ is a $\mu_{\omega,\beta}$-semistable sheaf and $\HH^0(\Phi(E))$ is a $0$-dimensional torsion sheaf.  Since $\w$ is not on a wall for $\v$, $\OO_s$ is not a stable factor of $\Phi(E)$, and thus $\Phi(E)[-1]$ is a $\mu_{\omega,\beta}$-semistable locally free sheaf.

The result follows by showing that $\Phi\circ[-1]$ induces an isomorphism $$\MM_{S,\tau}(\v)\isomor\MM_\omega(-\Phi(\v))$$ between the stack $\MM_{S,\tau}(\v)$ and the moduli stack of $\omega$-Gieseker semistable sheaves.  To see this, suppose that $E'$ is a $\mu_{\omega,\beta}$-semistable sheaf with Mukai vector $-\Phi(\v)$.  If $F'\subset E'$ is a saturated subsheaf of the same slope, then $F'[1]$ is a subobject of $E'[1]$ in $\AA_{\omega,\beta}$ with $\Im Z_{\omega,\beta}(F'[1])=0$, and from the genericity of $\tau'$ with respect to $\Phi(\v)$, $\v(F')$ must be proportional to $\v(E')$.  But then it follows that the $\beta$-twisted Hilbert polynomial of $F'$ is proportional to that of $E'$ (and both of these remain true independent of $\beta$).  Thus $E'$ is $\omega$-Gieseker semistable, and we get the required isomorphism.  As this isomorphism preserves $S$-equivalence, we get a projective coarse moduli space of $S$-equivalence classes of $\tau$-semistable objects of class $\v$, since the same is true of Gieseker stability.  As $\sigma$ and $\tau$ are in the same chamber, this proves the theorem.
\end{proof}
\section{Singularities and Kodaira dimension}\label{sec:Singularities}
Having shown in the previous section that there exists coarse moduli spaces of Bridgeland semistable objects, we describe in this section some of the global and local structure of their singularities, at least in the open locus of stable objects.  In particular, we show that the stable locus is generically smooth, and in high enough dimension these moduli spaces are normal and Gorenstein with only l.c.i. singularities.  

For $\sigma\in\Stabd(S)$, write $\sigma':=\pi_*(\sigma)\in\Stabd(X)$.  Our first result is the following.
\begin{Prop}\label{global singularities} Let $S$ be a bielliptic surface with canonical cover $X$, and set $G'=G$ for Types 1,3,5,7, $G'=G/(\Z/2\Z)$ for Types 2 and 4, and $G'=G/(\Z/3\Z)$ for Type 6 so that $S=X/G'$.  As each $G'$ is cyclic, we denote by $g$ a generator of $G'$.  Then $\mSs$ is singular at $E$ if and only if $E\cong E\otimes\o$ and $E$ does not lie along a smooth exceptional component of dimension two, which occurs if and only if $\v^2=0$ and $\pi^*\v=\o \v(F)$.  The singular locus is the union of the images under $\pi_*$ of finitely many components of the open loci $$M^s_{\sigma',X}(\w)^{\circ}:=\{F\in M^s_{\sigma',X}(\w)|g^*F\ncong F\}$$ with possibly different Mukai vectors $\w$ on $X$.  Its dimension is less than \begin{equation}\label{singular}\frac{1}{\o}\left(\dim\mSs+2\o-1\right).\end{equation}  So $\mSs$ is generically smooth.  The pull-back morphism $\mSs\mor\mXs$ is $\o$-to-one onto the isotropic subvariety $\FixG$, \'{e}tale away from the fixed loci of $-\otimes\o^k$ for $0<k<\o$.  For Types 1 and 2, $\FixG$ is in fact Lagrangian.
\end{Prop}

\begin{proof} Recall that we have the following dimension inequalities \cite[Corollary 4.5.2]{HL10}: \begin{align*}\v(E)^2+1&=\ext^1(E,E)-\ext^2(E,E)\leq\dim_E \mSs\leq\dim T_E\mSs=\ext^1(E,E)\\&=\v^2+1+\hom(E,E\otimes\o),\end{align*} for $E\in\mSs$, since \begin{align*}\ext^1(E,E)&=\v(E)^2+\hom(E,E)+\ext^2(E,E)=\v(E)^2+\hom(E,E)+\hom(E,E\otimes\o)\\
&=\v(E)^2+1+\hom(E,E\otimes\o),\end{align*} by Hirzebruch-Riemann-Roch, Serre duality and stability.  We note further that $\hom(E,E\otimes\o)=0$ or $1$, as $E$ and $E\otimes\o$ are $\sigma$-stable objects of the same Mukai vector so are either isomorphic or have no $\Hom$'s between them.  

Thus if $E\in \mSs$ is a singular point, then the obstruction space, $\Ext^2(E,E)$, does not vanish, so $\ext^2(E,E)=1$ and  $E\cong E\otimes \o$, and we assume this to be the case for the time being.  Then by \cite[Proposition 2.5]{BM98}, $E\cong \pi_*F$ for some $F$ which must be in $\mXs^{\circ}$ by Lemma \ref{lem:WhenIsPushForwardStable}.  It follows that $\pi^*E\cong\displaystyle\bigoplus_{k=0}^{\o-1} (g^k)^*F$.  If the Mukai vector of $F$ is $\v(F)=(r,c_1(F),\frac{1}{2}c_1(F)^2-c_2(F))$, then the rank of $E$ is $\o r$, \begin{equation}\label{c1}\pi^*c_1(E)=\sum_{k=0}^{\o-1}(g^k)^*c_1(F), \mbox{ and }\end{equation} \begin{equation}\label{c2}\o c_2(E)=c_2(\pi^*E)=\sum_{i\neq j}(g^i)^*c_1(F).(g^j)^*c_1(F)+\o c_2(F).\end{equation}  
We also have 
\begin{align}\label{square}
\begin{split}
    \o c_1(E)^2&=(\pi^*c_1(E))^2=\left(\sum_{k=0}^{\o-1}(g^k)^*c_1(F)\right)^2\\
    &=\o c_1(F)^2+2\sum_{i\neq j}(g^i)^*c_1(F).(g^j)^*c_1(F).
\end{split}\end{align}  Thus we have \begin{align*}\v(E)^2&=c_1(E)^2-2(\o r)\left(\frac{1}{2}c_1(E)^2-c_2(E)\right)=(1-\o r)c_1(E)^2+2\o rc_2(E)\\
&=(1-\o r)\left(c_1(F)^2+\frac{2}{\o }\sum_{i\neq j}(g^i)^*c_1(F).(g^j)^*c_1(F)\right)\\&+2r\left(\sum_{i\neq j}(g^i)^*c_1(F).(g^j)^*c_1(F)+\o c_2(F)\right)\\
&=(1-\o r)c_1(F)^2+\frac{2}{\o}\sum_{i\neq j}(g^i)^*c_1(F).(g^j)^*c_1(F)+2r\o c_2(F)\\
&=\o \v(F)^2+\left(\frac{2}{\o }\sum_{i\neq j}(g^i)^*c_1(F).(g^j)^*c_1(F)-(\o -1)c_1(F)^2\right).
\end{align*}  

Expanding \eqref{square} in another way we obtain 
\begin{align*} \left(\sum_{k=0}^{\o-1}(g^k)^*c_1(F)\right)^2&=\sum_{k=0}^{\o-1}(g^k)^*c_1(F).\left(\sum_{j=0}^{\o-1}(g^j)^*c_1(F)\right)\\&=\sum_{k=0}^{\o-1}\left((g^k)^*c_1(F).\sum_{j=0}^{\o-1}(g^j)^*c_1(F)\right)\\
&=\sum_{k=0}^{\o-1}\left(\sum_{j=0}^{\o-1}c_1(F).(g^{j-k})^*c_1(F)\right)\\
&=\sum_{k=0}^{\o-1}\left(\sum_{j=0}^{\o-1}c_1(F).(g^j)^*c_1(F)\right)\\
&=\o\left(c_1(F)^2+\sum_{k=1}^{\o-1}c_1(F).(g^k)^*c_1(F)\right).
\end{align*}

It follows that we may write \begin{equation}\label{mukai comparison}\v(E)^2=\o \v(F)^2+\sum_{k=1}^{\o-1}(c_1(F).(g^k)^*c_1(F)-c_1(F)^2).\end{equation}

Choosing any ample divisor $H$ on $S$, $\pi^*H$ is a $g^*$-invariant ample divisor on $X$, so $c_1(F).\pi^*H=(g^k)^*c_1(F).\pi^*H$ for any $k$, and by the Hodge Index Theorem \begin{equation}\label{hodge}2(c_1(F)^2-c_1(F).(g^k)^*c_1(F))=(c_1(F)-(g^k)^*c_1(F))^2\leq 0\end{equation} with equality if and only if $(g^k)^*c_1(F)=c_1(F)$.  Thus we see that $$\v(E)^2+1=\o(\v(F)^2+2)+\sum_{k=1}^{\o-1}\left(c_1(F).(g^k)^*c_1(F)-c_1(F)^2\right)-(2\o-1).$$  Since $F$ is stable on the abelian surface $X$, the moduli space $M_{\sigma',X}(\v(F))$ is smooth of dimension $\v(F)^2+2$ at $F$.  It follows from \eqref{hodge} that $$\dim M_{\sigma',X}(\v(F))\leq \left\{\begin{array}{l} \frac{1}{\o}(\dim \mSs+2\o-1)\mbox{ if }\dim_E \mSs=\v^2+1\\ \frac{1}{\o}(\dim \mSs+2\o-2)\mbox{ if }\dim_E \mSs=\v^2+2\end{array}\right\},$$ with equality if and only if $c_1(F)=g^*c_1(F)$.  Note that if $E$ is on a component of dimension $\v^2+2$, then the entire component is smooth and $E'\cong E'\otimes\o$ for every other point $E'$ on this component.  

If $E$ is indeed a singular point, then we must have that $E$ is on a component of the expected dimension $\v^2+1$.  Let $B=-\dim_E \mSs-(2\o-1)$ so that from the above inequality we must have $$\sum_{k=1}^{\o-1}\left(c_1(F)^2-c_1(F).(g^k)^*c_1(F)\right)\geq B,\text{ and}$$ $$2B\leq \sum_{k=1}^{\o-1}\left(c_1(F)-(g^k)^*c_1(F)\right)^2\leq 0,$$ where each term is non-positive.  Thus there can only be finitely many numbers $(c_1(F)-(g^k)^*c_1(F))^2$ and for any fixed value there can be only finitely many choices for $c_1(F)$ since $(\pi^*H)^{\perp}$ is a negative definite lattice.  Since $c_2(F)$ is determined by \ref{c2} above, there are only finitely many possible Mukai vectors for $F$. 

Furthermore, by Lemma \ref{lem:WhenIsPushForwardStable} $\pi_*F$ is $\sigma$-stable if and only if $F\ncong(g^k)^*F$ for $0<k<\o$.  Thus the singular locus of $\mSs$ is the image under $\pi_*$ of finitely many $M^s_{\sigma',X}(\v(F))^{\circ}$'s for the finitely many $\v(F)$'s satisfying the above necessary requirements.

Now we show that the push-forward map $\pi_*:M^s_{\sigma',X}(\v(F))^{\circ}\rightarrow \mSs$ is \'{e}tale of degree $\o$.  Indeed, if $\pi_*(F)\cong\pi_*(G)$, then pulling back gives $\displaystyle\bigoplus_{k=0}^{\o-1}(g^k)^*F\cong \displaystyle\bigoplus_{k=0}^{\o-1}(g^k)^*G$, so $$\C^{\o}=\Hom\left(\bigoplus_{k=0}^{\o-1}(g^k)^*F,\bigoplus_{k=0}^{\o-1}(g^k)^*G\right)=\bigoplus_{k=0}^{\o-1}\Hom(F,(g^k)^*G)^{\oplus \o},$$ so $F\cong (g^k)^*G$ for precisely one $k$.  Thus the singular locus is of even dimension and smooth itself.

Suppose now that $E$ lies on a component of dimension $\v^2+2$, and is thus a smooth point.  Then letting $C=-\dim_E \mSs-(2\o-2)$ we again get a similar bound $$2C\leq \sum_{k=1}^{\o-1}(c_1(F)-(g^k)^*c_1(F))^2\leq 0,$$ with each term of the sum nonpositive, and thus again the component $M\subset \mSs$ containing $E$ is the image under $\pi_*$ of some components of $M^s_{\sigma',X}(\v(F))^{\circ}$ for the finitely many $\v(F)$ satisfying this bound.  The argument above shows that this map $\pi_*$ is finite \'{e}tale so that if $\overline{M}\subset M^s_{\sigma',X}(\v(F))^{\circ}$ is one such component, then $$\dim M=\dim \overline{M}\leq \frac{1}{\o}(\dim M+2\o-2).$$  This forces the dimension of $M$ to be 0 or 2.  If $\dim M=2$, then we get equality in this inequality so that $c_1(F)=g^* c_1(F)$ and thus $\pi^*\v=\o \v(F)$.  If $\dim M=0$, then $M$ has a unique stable object $E$ and and $\overline{M}$ has $\o$ objects $(g^k)^*F$, $k=0,\dots,\o-1$, with $c_1(F)^2=c_1(F).(g^k)^*c_1(F)-2$.  But the object $E$ would then be spherical, which is impossible for a bielliptic surface \cite[Proposition 5.5]{Sos13}.

Since $G'$ is cyclic, any object in $\FixG=\mathrm{Fix}(g^*)$ admits a linearization and thus comes from a $G'$-equivariant object.  In other words, $\FixG$ is indeed the image of $\pi^*$.  To see that $\FixG$ is isotropic, recall that we have a commutative diagram,

$$\renewcommand{\arraystretch}{1.3}
\begin{array}[c]{ccccc}
\Ext^1(\pi^*E,\pi^*E)&\times&\Ext^1(\pi^*E,\pi^*E)&\rightarrow&\Ext^2(\pi^*E,\pi^*E)\\
\uparrow &&\uparrow &&\uparrow\\
\Ext^1(E,E)&\times&\Ext^1(E,E)&\rightarrow&\Ext^2(E,E)
\end{array},$$
where the top right arrow is the holomorphic symplectic form on $\mXs$ defined originally for sheaves in \cite{Muk84} and generalized to complexes in \cite[Theorem 3.3]{Ina11}.  Since $\Ext^2(E,E)=0$ for $E\in\mSs_{sm}$, we get that $\FixG$ is isotropic.  

Now we prove that $\pi^*$ is unramified on the claimed locus.  Lieblich and Inaba (in \cite{Lie} and \cite{Ina02}, respectively) generalized the well-known results about the deformation theory of coherent sheaves to complexes of such.  In particular, for $E\in \mSs$ and $F=\pi^*E$, the tangent spaces are $$T_{E}\mSs\cong \Ext^1(E,E),\text{ and }T_{F}\mX\cong \Ext^1(F,F),$$ and the differential is just the natural map $$d\pi^*:\Ext^1(E,E)\mor[\pi^*] \Ext^1(F,F).$$  

We claim that the differential must be injective for those $E$ such that $E\ncong E\otimes \o^k$ for any $0<k<\o$.  Indeed, suppose $E'\in \text{Ext}^1(E,E)$, i.e. $E'$ is an extension $$0\rightarrow E\rightarrow E'\rightarrow E\rightarrow 0.$$  Applying $\Hom(-,E\otimes\o^k)$ and noting that $E$ and $E\otimes\o^k$ are nonisomorphic for $0<k<\o$ and stable of the same phase so that $\Hom(E,E\otimes\o^k)=0$, we must have $\Hom(E',E\otimes\o^k)=0$.  Suppose that $\pi^*E'=0\in\Ext^1(F,F)$, i.e. the short exact sequence $$0\rightarrow F\rightarrow \pi^*E'\rightarrow F\rightarrow 0$$ splits.  But then so does the short exact sequence $$0\rightarrow \pi_*(F)\rightarrow \pi_*(\pi^*E')\rightarrow \pi_*(F)\rightarrow 0.$$  But this is precisely the sequence $$0\rightarrow \bigoplus_{k=0}^{\o} E\otimes\o^k\rightarrow \bigoplus_{k=0}^{\o} E'\otimes\o^k\rightarrow  \bigoplus_{k=0}^{\o} E\otimes\o^k\rightarrow 0.$$  Since $\Hom(E',E\otimes\o^k)=\Hom(E'\otimes\o^k,E)=0$ for $0<k<\o$, it follows that any morphism $$\bigoplus_{k=0}^{\o} E'\otimes\o^k\rightarrow \bigoplus_{k=0}^{\o} E\otimes\o^k$$ must be component wise, and thus any splitting of this short exact sequence induces a splitting of the original exact sequence $$0\rightarrow E\rightarrow E'\rightarrow E\rightarrow 0,$$ proving injectivity and giving the final claim.
\end{proof}

We have dealt with the global nature of $\Sing(\mSs)$ in Proposition \ref{global singularities}, which generalizes results of \cite{Kim98} for stable vector bundles on Enriques surfaces.  The next result, which deals with the local nature of these singularities, can be proven exactly as in \cite{Yamada} for the corresponding result on sheaves:

\begin{Prop}\label{local singularities} Suppose that $S$ is a bielliptic surface, $\v\in\Hal(S,\Z)$, and $\sigma\in\Stab^{\dagger}(S)$.  If $\v^2\geq 3\o$, then $\mSs$ is normal and Gorenstein with only terminal l.c.i. singularities.  Furthermore, if $\codim\Sing(\mSs)\geq 2$ (for example whenever $\v^2\geq 6$), then $\omega_{\mSs}$ is torsion in $\Pic(\mSs)$.  If $S$ is generic then both conclusions hold for all $\v^2>0$ except when $\o=2,3$, $\v^2=2\o$ and $\o\mid\v$, in which case the second conclusion still holds if $\o=3$.
\end{Prop}

\begin{proof} Although our claim is stronger, the proof follows along the same lines as in \cite{Yamada}, so we will use the notation found therein and indicate the needed modifications.  

 We first observe that since $\mSs$ has only hypersurface singularities, it is normal and Gorenstein as long as it is regular in codimension one.  To prove the terminality of the singularities, it suffices to look at the the completion of the local ring $\OO_{\mSs,E}$ at the point $E$ since the notion of terminal singularities are analytically local \cite[Proposition 4-4-4]{Mat02}.  By \cite[Facts 2.4, 2.5]{Yamada}, this completion looks like $\C[[t_1,\dots,t_{d+1}]]/(t_1^2+...+t_N^2+g(t_{N+1},\dots,t_{d+1}))$, where $d=\v^2+1$, $N$ is an integer between $1$ and $d+1$, and $g\in(t_{N+1},\dots,t_{d+1})^3$.  Moreover, as terminal singularities are stable under small deformations \cite[Proposition 9.1.4]{Ish14},\cite[Corollary 5.3]{Nak04}, it suffices to consider only the quadratic term.  That is, it suffices to prove that $\C[[t_1,\dots,t_{d+1}]]/(t_1^2+...+t_N^2)$ has at worst terminal singularities.  By \cite[Theorem 2]{Lin02}, a hypersurface singularity $t_1^{a_1}+\dots+t_n^{a_n}$ has terminal singularities if and only if $\sum_{i=1}^n\frac{1}{a_i}>1+\frac{1}{\lcm(a_i)}$.  So it suffices to show that $N$ above is at least 4.  

From the proof of \cite[Theorem 1.1]{Yamada}, $N\geq 4$ and $\codim(\Sing(\mSs))\geq 2$ if $\v^2\geq 4\o$.  We nevertheless show that the original dimension condition in \cite{Yamada}, namely $\v^2\geq 3\o$, suffices to guarantee that the singularities are not just canonical but in fact terminal.  To this end, first note  \begin{equation}\label{codim}\dim\mSs-\dim\Sing(\mSs)\geq(\v^2+1)-\frac{1}{\o}(\v^2+2\o)=\frac{\o-1}{\o}\v^2-1\geq 2,\end{equation} is satisfed as long as $\v^2\geq 3(\frac{\o}{\o-1})$.  So the codimension condition on $\Sing(\mSs)$ is satisfied for all $\v^2\geq 6$ as $\frac{\o}{\o-1}\leq 2$.  Furthermore, observe that if equality is not achieved in \eqref{singular}, then \cite[(12)]{Yamada} is a strict inequality so that there exists $e\neq h\in G'$ such that $$N\geq\ext^1(F_e,F_h)>\frac{\v^2}{\o}\geq 3,$$ guaranteeing that $N\geq 4$.  If, however, equality is achieved in \eqref{singular}, then $\o\mid\v$ so that $2\o\mid \v^2$.  As $\frac{\v^2}{\o}\geq 3$ is then an even integer, we must again have $N\geq 4$.

If $S$ is generic, then $\Num(S)_{\Q}\cong\Num(X)$, so if $\Sing(\mSs)\neq\varnothing$, then we have equality in \eqref{singular} as $c_1(F_h)=h^*c_1(F_e)=c_1(F_e)$ for all $h\in G'$.  Thus, $\o\mid\v$.  The only case not covered above is if $(\frac{\v}{\o})^2=2$ so that $\v^2=2\o$.  In this case, $\ext^1(F_e,F_g)=\langle \v(F_e),\v(F_g)\rangle=\v(F_e)^2=(\frac{\v}{\o})^2=2$, so from the previous argument we cannot guarantee that $N\geq 4$.  Nevertheless, see Example \ref{order 4 singular} and Example \ref{order 6 singular} for a proof in these cases.  We also note that $\codim\Sing(\mSs)=(\v^2+1)-(\v(F_e)^2+2)=2\o-3\geq 2$ if $\o>2$.

The statement about the canonical divisor is shown precisely as in \cite[Proposition 8.3.1]{HL10} as long as $\codim\Sing(\mSs)\geq 2$.
\end{proof}
We examine the cases not covered by \cref{local singularities} in the appendix since it would take us too far afield from our main goals to discuss it here.
\section{Moduli of Gieseker stable sheaves and non-emptiness of Bridgeland moduli spaces}\label{sec:ModuliOfGieseker}
While Theorem \ref{thm:projective coarse moduli spaces} guarantees that a projective coarse moduli space exists for Bridgeland semistable objects, and \cref{global singularities,local singularities} describes the singularities of the open stable locus, these results say nothing if the moduli space is empty.  However, just as the question of coarse moduli spaces can be reduced to the same question for Gieseker semistability, we will see in this section that the same is true for the question of non-emptiness of Bridgeland moduli spaces.  After beginning with a crucial result that allows us to reduce the study of topological invariants of $M_{\sigma,S}(\v)$ to those of Gieseker moduli spaces, we return to the existence problem on $S$ for Gieseker semistable sheaves themselves.  
\subsection{Motivic invariants}\label{subsec:MotivicInvariants}
We begin by introducing a specific case of the notion of motivic invariants.  A \emph{motivic invariant} is an invariant $J$ of varieties (over $\C$) such that $J(Y)=J(Z)+J(Y\backslash Z)$ for every variety $Y$ and closed subscheme $Z\subset Y$.  We will primarily be interested in the case that $J$ is the virtual Hodge polynomial.  Let us recall that for a variety $Y$ over $\C$, the cohomology with compact support $H^*_c(Y,\Q)$ has a natural mixed Hodge structure. Let $e^{p,q}(Y):=\sum_k (-1)^k h^{p,q}(H^k_c(Y))$ and $e(Y):=\sum_{p,q}e^{p,q}(Y)x^p y^q$ be the virtual Hodge number and virtual Hodge polynomial, respectively.  For $\sigma\in\Stab(Y)$ and $\v\in\Hal(Y,\Z)$, we define $J_\sigma(\v):=e(M_{\sigma,Y}(\v))$, whenever the coarse moduli scheme $M_{\sigma,Y}(\v)$ exists.

\begin{Lem}\label{Motivic invariant} The motivic invariant $J_\sigma(\v)$ is invariant under wall-crossing and autoequivalences.
\end{Lem}
\begin{proof}
We use the construction of the Joyce invariant $J_\sigma(\v)$ of \cite[Section 5]{Tod08}, which is quite general.  In particular, Lemma 5.12 there applies.  Likewise the analogous algebra $A(\mathcal A_{\phi},\Lambda,\chi)$ is still commutative since $K_S$ is numerically trivial and thus the Mukai pairing is commutative.  This and the results above show that \cite[Theorem 5.24 and Corollary 5.26]{Tod08} still apply.  In particular, $J_\sigma(\v)$ is the motivic invariant of the projective coarse moduli space $M_{\sigma,S}(\v)$, and is invariant under autoequivalences and changes in $\sigma$.  
\end{proof}

\subsection{Moduli spaces of Gieseker semistable sheaves on $S$}
With the exception of some early results of Takemoto \cite{Tak73} and Umemura \cite{Ume75} on stable sheaves with isotropic Mukai vector, the main result about stable sheaves on bielliptic surfaces involves relative Fourier-Mukai transforms.  Before discussing this result, we introduce some alternate notation for divisors.  Recall that $S$ admits two elliptic fibrations, $p_A:S\to A/G$, whose fibres are smooth and isomorphic to $B$, and $p_B:S\to B/G\cong\P^1$, which has some multiple fibres but whose smooth fibres are isomorphic to $A$.  Then $A_0$ provides a multisection of $p_A$ with minimal degree $\lambda_S$, and $B_0$ provides a multisection of $p_B$ with degree $\o$.  Define the fibre degree $d_{p_A}(\v)$ (resp. $d_{p_B}(\v)$) of $\v$ with respect to $p_A$ (resp. $p_B$) to be $d_{p_A}(\v):=c_1(\v).B$ (resp. $d_{p_B}(\v):=c_1(\v).A$) (we drop $p_A$ (resp. $p_B$) in the notation when it is clear from the context).  With this notation, any divisor class  $c_1\in\Num(S)$ can be written $c_1=\left(\frac{d_{p_A}(c_1)}{\lambda_S}\right)A_0+\left(\frac{d_{p_B}(c_1)}{\o}\right)B_0$.  The main theorem to-date, proved by Bridgeland in \cite{Bri98}, is the following:
\begin{Thm}[Bridgeland]\label{elliptic sheaves} Suppose that $\omega$ is a generic polarization that is $\v$-suitable with respect to $p_A$, that $r>1$, and that $\gcd(r,d_{p_A}(\v))=1$.  Then $M_\omega(\v)$ is a smooth projective variety birational to $$\Pic^{0}(S)\times\Hilb^{\v^2/2}(S).$$ Furthermore, if $r>a\left(\frac{\v^2}{2}\right)$, then the birational equivalence extends to an isomorphism of varieties, where $a$ is the unique integer $0<a<r$ such that $br-ad_{p_A}(\v)=1$ has an integral solution $b$.  
\end{Thm}

Here a polarization $\omega$ on a smooth projective surface $Y$ is called generic and $\v$-suitable with respect to a fibration $p:Y\to C$ if it lies in a chamber $\CC$ of $\Amp(Y)_\R$ for $\v$ such that the fiber $p^{-1}(\pt)\in\overline{\CC}$.  A consequence of the definition of $\v$-suitability is that a coherent sheaf $E$ of Mukai vector $v$ that is Gieseker semistable with respect to $\omega$ restricts to a semistable sheaf on the generic fibre of $p$, and conversely, a sheaf of Mukai vector $\v$ restricting to a stable sheaf on the generic fibre of $p$ is Gieseker stable with respect to $\omega$ \cite[Theorem 5.3.2]{HL10}.  In Theorem \ref{elliptic sheaves}, we assume that $\gcd(r,d_{p_A}(\v))=1$, so genericity implies that $\mu$-stablility is equivalent to $\mu$-semistability, which are also equivalent to Gieseker stability.  

The proof of Theorem \ref{elliptic sheaves}, whose more general form for any elliptic surface is \cite[Theorem 1.1]{Bri98}, involves an application of the following result:

\begin{Thm}\cite[Theorem 1.2]{Bri98}\label{relative FM} Suppose that $\omega$ is a generic $\v$-suitable polarization on an elliptic surface $p:Y\to C$ with minimal multisection degree $\lambda_Y$, and that integers $a>0$ and $b$ satisfy $\gcd(a\lambda_Y,b)=1$.  Let $q:\tilde{Y}:=J_Y(a,b)\to C$ be the union of those components of $M_\omega(Y/C)$, the relative moduli space of $\omega$-Gieseker stable pure dimension 1 sheaves on the fibers of $p$,  that contain a rank $a$, degree $b$ vector bundle supported on a smooth fibre of $p$.  Then for an element $$\left(\begin{matrix} c & a\\d & b\end{matrix}\right)\in\SL_2(\Z),$$ such that $\lambda_Y$ divides $d$, there exist tautological sheaves on $Y\times\tilde{Y}$, supported on $Y\times_C\tilde{Y}$, and for each such sheaf $\PP$, the Fourier-Mukai functor $\Phi_{\PP}:\Db(\tilde{Y})\to\Db(Y)$ is an equivalence and satisfies $$\left(\begin{matrix} r(\Phi_{\PP}(E))\\d_{p}(\Phi_{\PP}(E))\end{matrix}\right)=\left(\begin{matrix} c & a\\d & b\end{matrix}\right)\left(\begin{matrix}r(E)\\d_q(E)\end{matrix}\right)$$ for all objects $E\in\Db(\tilde{Y})$.
\end{Thm}

By \cite[Proposition 6.2]{BM01}, bielliptic surfaces have no nontrivial FM partners, so $J_S(a,b)\cong S$ for any integers $a>0$ and $b$ with $\gcd(b,a\lambda_S)=1$.  Theorem \ref{elliptic sheaves} results upon considering how stable sheaves transform upon applying the Fourier-Mukai (auto)equivalence of Theorem \ref{relative FM}.  By using Bridgeland stability, we can significantly streamline this approach.

Here we will use Theorem \ref{relative FM} and the results of  \cref{sec: coarse moduli,sec:Singularities} to obtain a more general version of Theorem \ref{elliptic sheaves}:
\begin{Thm}\label{Thm:Non-emptinessOfGieseker}
Let $S$ be a bielliptic surface and $\omega$ a generic polarization with respect to a primitive positive Mukai vector $\v$ of positive rank.  Let $n\in\N$.
\begin{enumerate}
    \item If $\v^2=0$, then $M^{\mu s}_\omega(n\v)\ne\varnothing$ iff $nl(\v)\mid\o$.
    \item If $\v^2>0$, then $M^{\mu ss}_\omega(\v)\neq\varnothing$ and there is a $\mu$-stable sheaf in each irreducible component of $M^{\mu ss}_\omega(\v)$.  Moreover, except when $\o=2$ and $\v=(2,0,-1)e^D$ or $S$ is of Type 1 and $\v=(2,B_0,-1)e^D$, this $\mu$-stable sheaf may be taken to be locally free as long as $\rk(\v)>1$.  Even in these exceptional cases, there are components which contain $\mu$-stable locally free sheaves.
\end{enumerate}
\end{Thm}
Here, positivity of $\v\in\Hal(S,\Z)$ is meant in the sense of \cite[Definition 0.1]{Yos01}.  While the existence of rank zero stable sheaves follows from the existence of stable vector bundles on smooth curves, we will focus on the positive rank case.  We will prove \cref{Thm:Non-emptinessOfGieseker} by reducing to a finite number of explicit, easily-handled cases using the following result:

\begin{Prop}\label{FM reduction} Suppose that $\omega$ is a generic polarization for a primitive Mukai vector $\v=(R,aA_0+bB_0,s)$.  Then $$e(M_\omega(\v))=e(M_\omega(\v_0)),$$ where $\v_0$ is one of the reduced forms appearing in Table \ref{table:Exceptions} and $1\leq\rk(\v_0)\leq R$.
\end{Prop}

\begin{table}[ht]
\caption{The reduced forms of primitive Mukai vectors}
\begin{tabular}{|c| c| c| c| c| c|}
\hline
Type & $\lambda_S$ & $\o$& $\v_0$\\
\hline
1 & $1$ & $2$ & $(q,0,s),(2q,qB_0,s)$\\
\hline
2 & $2$ & $2$ & $(q,0,s),(2q,qA_0,s),(2q,qB_0,s),(2q,qA_0+qB_0,s)$\\
\hline
3 & $1$ & $4$ & $(q,0,s),(4q,qB_0,s),(2q,qB_0,s)$\\
\hline
4 & $2$ & $4$ & $(q,0,s),(2q,qA_0,s),(4q,qB_0,s),(2q,qB_0,s),$\\
   &        &        &$(4q,2qA_0+qB_0,s),(2q,qA_0+qB_0,s)$\\
\hline
5 & $1$ & $3$ & $(q,0,s),(3q,qB_0,s)$\\
\hline
6 & $3$ & $3$ & $(q,0,s),(3q,qA_0,s),(3q,qB_0,s),$\\
   &        &        &$(3q,qA_0+qB_0,s)$\\
   \hline
7 & $1$ & $6$ & $(q,0,s),(6q,qB_0,s),(3q,qB_0,s),(2q,qB_0,s),(q,bB_0,s), \frac{1}{3}<\frac{b}{q}<\frac{1}{2}$\\
\hline
\end{tabular}
\label{table:Exceptions}
\end{table}
\begin{proof}
From \cref{Motivic invariant}, the equality of virtual Hodge polynomials will follow if we can apply an autoequivalence $F$ of $S$ whose cohomological action takes $\v$ to one of the reduced Mukai vectors $\v_0$ of Table \ref{table:Exceptions}.  Indeed, we may take $\sigma$ such that $M_{\sigma,S}(\v)\cong M_\omega(\v)$, so $e(M_\omega(\v))=e(\mS)$.  Then applying the autoequivalence $F$, we get 
$$e(\mS)=e(M_{F(\sigma),S}(F(\v))=e(M_{F(\sigma),S}(\v_0))=e(M_\omega(\v_0)),$$ 
where the final equality comes from moving $F(\sigma)$ to the Gieseker chamber for $\v_0$.  In addition to tensoring by line bundles, we will apply the derived dual, as well as two relative Fourier-Mukai transforms $\Phi$ and $\Psi$ corresponding to the matrix $\left(\begin{matrix}1&1\\0&1\end{matrix}\right)$ relative to $p_A$ and $p_B$, respectively, as in Theorem \ref{relative FM}.  According to the description in Theorem \ref{relative FM}, \begin{align*}&\Phi(0,0,1)=(0,\lambda_SB_0,1),\Phi(1,0,0)=(1,yB_0,w),\\ &\Phi(0,A_0,0)=(\lambda_S,A_0+xB_0,t),\Phi(0,B_0,0)=(0,zB_0,u),\end{align*} for $y,w,x,t,z,u\in\Z$.  Using the fact that $\Phi$ acts on $\Hal(S,\Z)$ by an isometry, we get that $w=0,x=\lambda_St,y=t,u=0$, and $z=1$.  Replacing $\Phi$ by $(-\otimes\OO_S(-tB_0))\circ\Phi$ if necessary, we may in-fact assume that $t=0$ so that \begin{align*}&\Phi(0,0,1)=(0,\lambda_SB_0,1),\Phi(1,0,0)=(1,0,0),\\ &\Phi(0,A_0,0)=(\lambda_S,A_0,0),\Phi(0,B_0,0)=(0,B_0,0).\end{align*}
Similarly, we may assume that $\Psi$ satisfies \begin{align*}&\Psi(0,0,1)=(0,\o A_0,1),\Psi(1,0,0)=(1,0,0),\\ &\Psi(0,A_0,0)=(0,A_0,0),\Psi(0,B_0,0)=(\o,B_0,0).\end{align*}  As we will actually apply $\Phi^{-1}$ and $\Psi^{-1}$, it is worth explicitly stating that they satisfy \begin{align*}&\Phi^{-1}(R,aA_0+bB_0,s)=(R-\lambda_S a,aA_0+(b-\lambda_S s)B_0,s),\\&\Psi^{-1}(R,aA_0+bB_0,s)=(R-\o b,(a-\o s)A_0+bB_0,s).\end{align*}

Our proof will proceed by induction on the rank of $\v=(R,aA_0+bB_0,s)$.
We can tensor $\v'=(qr,qD,s')$ by $\OO_S(-\left\lfloor\frac{a}{r}\right\rfloor A_0-\left\lfloor\frac{b}{r}\right\rfloor B_0)$ to assume that $0\leq a,b<R$.  
Moreover, by \cref{slope stability}, if $qr>1$ there is always a locally free $\mu$-stable sheaf in the moduli space, so applying $(-\otimes\OO_S(A_0+B_0))\circ(-)^\vee$ allows us to assume further that $0\leq a,b\leq\frac{r}{2}$.  
We will show that the reduced forms $\v_0$ are the end results of our algorithm.   

If $a=b=0$, then we are in the first reduced Mukai form appearing in Table \ref{table:Exceptions}.  Otherwise, if, say, $a>0$, then we may assume that $\frac{R}{\lambda_S}\leq a\leq\frac{R}{2}$ (in which case $\lambda_S>1$). Indeed, if $R>\lambda_Sa$, we can apply $\Phi^{-1}$ $\lfloor\frac{R}{\lambda_S a}\rfloor$-times to reduce to a Mukai vector $\v'=\left(R',aA_0+\left(b-\lambda_S\left(\lfloor\frac{R}{\lambda_S a}\rfloor\right)s\right)B_0,s\right)$ with $R'\leq\lambda_S a<R$.  By induction, $$e(M_\omega(\v))=e(M_\omega(\v'))=e(M_\omega(\v_0))$$ for some $\v_0$ in Table \ref{table:Exceptions} with $1\leq\rk(\v_0)\leq R'$, as required.  Thus, we may assume from the outset that either $a=0$, $a=\frac{R}{\lambda_S}$, $a=\frac{R}{2}$, or $\frac{R}{\lambda_S}<a<\frac{R}{2}$.  

If $\frac{R}{\lambda_S}<a\leq\frac{R}{2}$, then $\lambda_S=3$, in which case $S$ is of Type 6.  We apply the (anti-)autoequivalence $(\blank)^\vee\circ[1]\circ\Phi^{-1}\circ\left(\blank\otimes \OO_S(A_0)\right)\circ (\blank)^\vee$, which sends  $$(R,n(aA_0+bB_0),s)\mapsto\v'=(2R-3a,(R-a)A_0+(-b-3T)B_0,-T).$$ where $T=s-b$.  As we assume that $\frac{R}{3}<a\leq\frac{R}{2}<\frac{2R}{3}$,  it follows that $0<2R-3a<R$.
By induction we again get $e(M_\omega(\v))=e(M_\omega(\v_0))$ for $\v_0$ in Table \ref{table:Exceptions} with $\rk(\v_0)<R$.  To summarize, define the set $\NN_S$ by 
$$\NN_S:=\begin{cases}
\{0\} &\lambda_S=1\\
\{0,\frac{1}{\lambda_S}\} & \lambda_S>1
\end{cases}.
$$
Then we may assume $a/R\in\NN_S$.

Similarly, by performing the same procedure with $\Psi^{-1}$ replacing $\Phi^{-1}$ and $\blank\otimes\OO_S(B_0)$ replacing $\blank\otimes\OO_S(A_0)$, we may assume that either $b\in\{0,\frac{R}{\o},\frac{R}{2}\}$, or $\frac{R}{\o}<b<\frac{R}{2}$.  We break up the study of the region $\frac{R}{\o}<b\leq\frac{R}{2}$ into cases according to $\o$.  We note that $\o>2$.  

Suppose first that $\o=3$ and  $\frac{R}{3}<b\leq\frac{R}{2}<\frac{2R}{3}$. Then $0<2R-3b<R$, and applying $(\blank)^\vee\circ[1]\circ\Psi^{-1}\circ(\blank\otimes\OO_S(B_0))\otimes(\blank)^\vee$ sends
$$(R,aA_0+bB_0,s)\mapsto(2R-3b,(-a-3 s)A_0+(R-b)B_0,-s).$$  This again gives $e(M_\omega(\v))=e(M_\omega(\v_0))$ with $\v_0$ as in Table \ref{table:Exceptions} and $\rk(\v_0)<R$ by induction.  

Suppose next that $\o=4,6$ and   $\left(\frac{1}{\o}\right)R<b<\left(\frac{2}{\o}\right) R$.  Then $0<\o b-R<R$, and applying $(\blank)^\vee\circ[1]\circ\Psi^{-1}$ sends 
$$(R,aA_0+bB_0,s)\mapsto(\o b-R,(a-\o s)A_0+b B_0,-s),$$ which gives $e(M_\omega(\v))=e(M_\omega(\v_0))$ with $\v_0$ as in \cref{table:Exceptions} and $\rk(\v_0)<R$ by induction.

To summarize, let $\BB_S$ be the set defined by 
$$\BB_S:=\begin{cases}
[\frac{1}{3},\frac{1}{2}] &\o=6\\
\frac{1}{2} & \o=4\\
\varnothing & \o=2,3
\end{cases}.
$$
Then we get the result by induction unless $\frac{a}{R}\in\NN_S$ and $\frac{b}{R}\in\{0,\frac{1}{\o}\}\bigcup \BB_S$.  But then $\v$ itself is as in \cref{table:Exceptions}.
\end{proof}

\begin{Rem}\label{rem: Bridgeland case} The case $\gcd(R,d_{p_A})=1$ (or $\gcd(R,d_{p_B})=1$) corresponds to the situation considered in Theorem \ref{elliptic sheaves}.  Indeed, the proof of \cref{FM reduction} shows that then $$e(M_{\omega}(\v))=e(M_{\omega}(1,0,-\v^2/2)).$$  Of course, $M_{\omega}(1,0,-\v^2/2)$ parametrizes sheaves of the form $I_Z\otimes L$, where $Z$ is a 0-dimensional subscheme of length $\v^2/2$ and $L\in\Pic^0(S)=\Ker(c_1:\Pic(S)\to\NS(S))$, so in fact $M_\omega(1,0,-\v^2/2)\cong \Pic^{0}(S)\times \Hilb^{\v^2/2}(S)$.  Theorem \ref{elliptic sheaves} states the stronger fact that these two moduli spaces are in fact birational for suitable polarizations, which certainly implies the equality of Hodge polynomials.  We will see later that a similar stronger version of Proposition \ref{FM reduction} holds, using wall-crossing in the space of Bridgeland stability conditions.
\end{Rem}

Now we have everything we need to classify the Mukai vectors of $\mu$-stable sheaves on $S$.  

\begin{proof}[Proof of \cref{Thm:Non-emptinessOfGieseker}]
Write $\v=(qr,qD,s)$ where $\gcd(r,D)=1$.  Using classical methods, we have shown that to prove the theorem, it suffices to prove that for any fixed $(r,D)=(r,aA_0+bB_0)$ we have $M_\omega^{\mu ss}\left(qr,qD,\left\lfloor\frac{qD^2}{2r}\right\rfloor\right)\ne\varnothing$ for $0<q\leq n_0$.

We can rephrase \cref{FM reduction} as saying that $e(M_\omega(\v))=e(M_\omega(qr,qD,s))$ where $(r,D)$ with $\gcd(r,D)=1$ is one of the finitely many pairs in \cref{table:Exceptions}.

\begin{table}[ht]
\caption{The finitely many pairs $(r,D)$}
\begin{tabular}{|c| c| c| c| c| c|}
\hline
Type & $\lambda_S$ & $\o$& $(r,D)$\\
\hline
1 & $1$ & $2$ & $(1,0),(2,B_0)$\\
\hline
2 & $2$ & $2$ & $(1,0),(2,A_0),(2,B_0),(2,A_0+B_0)$\\
\hline
3 & $1$ & $4$ & $(1,0),(4,B_0),(2,B_0)$\\
\hline
4 & $2$ & $4$ & $(1,0),(2,A_0),(4,B_0),(2,B_0),$\\
   &        &        &$(4,2A_0+B_0),(2,A_0+B_0)$\\
\hline
5 & $1$ & $3$ & $(1,0),(3,B_0)$\\
\hline
6 & $3$ & $3$ & $(1,0),(3,A_0),(3,B_0),$\\
   &        &        &$(3,A_0+B_0)$\\
   \hline
7 & $1$ & $6$ & $(1,0),(6,B_0),(3,B_0),(2,B_0),(r,bB_0), \frac{1}{3}<\frac{b}{r}<\frac{1}{2}$\\
\hline
\end{tabular}
\label{table:Exceptions1}
\end{table}

For any pair $(r,D)$ in \cref{table:Exceptions1} with $D^2=0$ we have $n_0=1$, while the remaining pairs have $n_0=\lambda_S$.  Thus it suffices to consider the finite number of primitive isotropic Mukai vectors $\u=(n_0r,n_0D,s_0)$ appearing in \cref{isotropic exceptions} as well as the exceptional Mukai vectors $\bar{\v}=(qr,qD,\left\lfloor\frac{q}{n_0}\right\rfloor)$ for $1\leq q<n_0$ which appear in \cref{base case}. 
\begin{table}[ht]
\caption{The reduced forms of primitive isotropic vectors}
\begin{tabular}{|c| c| c| c| c| c|}
\hline
Type & $\lambda_S$ & $\o$& $\u$\\
\hline
1 & $1$ & $2$ & $(1,0,0),(2,B_0,0)$\\
\hline
2 & $2$ & $2$ & $(1,0,0),(2,A_0,0),(2,B_0,0),(4,2A_0+2B_0,1)$\\
\hline
3 & $1$ & $4$ & $(1,0,0),(4,B_0,0),(2,B_0,0)$\\
\hline
4 & $2$ & $4$ & $(1,0,0),(2,A_0,0),(4,B_0,0),(2,B_0,0)$\\
   &        &        &$(8,4A_0+2B_0,1),(4,2A_0+2B_0,1)$\\
\hline
5 & $1$ & $3$ & $(1,0,0),(3,B_0,0)$\\
\hline
6 & $3$ & $3$ & $(1,0,0),(3,A_0,0),(3,B_0,0),$\\
   &        &        &$(3,2B_0,0),(9,3A_0+3B_0,1),$\\
\hline
7 & $1$ & $6$ & $(1,0,0),(6,B_0,0),(3,B_0,0),(2,B_0,0),(r,bB_0,0),\frac{1}{3}<\frac{b}{r}<\frac{1}{2}$\\
\hline
\end{tabular}
\label{isotropic exceptions}
\end{table}
\begin{table}[ht]
\caption{Fundamental Mukai vectors for $\mu$-semistability}
\begin{tabular}{|c| c| c| c| c| c|}
\hline
Type & $\lambda_S$ & $\o$& $\bar{\v}$\\
\hline
2 & $2$ & $2$ & $(2,A_0+B_0,0)$\\
\hline
4 & $2$ & $4$ & $(4,2A_0+B_0,0),(2,A_0+B_0,0)$\\
\hline
6 & $3$ & $3$ & $(3,A_0+B_0,0),(6,2(A_0+B_0),0)$\\
  \hline
\end{tabular}
\label{base case}
\end{table}

As the isotropic case is more involved, we first prove that $M_\omega^{\mu ss}(\bar{\v})\ne\varnothing$ for all $\bar{\v}$ in \cref{base case}.

For $\bar{\v}=(2,A_0+B_0,0)$ in Types 2 and 4, since $\langle (1,A_0,0),(1,B_0,0)\rangle=1$ and $$\Hom(\OO(A_0),\OO(B_0))=\Ext^2(\OO(A_0),\OO(B_0))=0,$$ we must have $\Ext^1(\OO(A_0),\OO(B_0))=\C$ so that there is a unique non-split extension $$0\to \OO(B_0)\to E\to\OO(A_0)\to 0.$$  For $\omega=A_0+2B_0$, the usual argument shows that such an $E$ is $\mu_\omega$-stable of Mukai vector $(2,A_0+B_0,0)$, so this must be true for generic $\omega$ as well.  

For the remaining Mukai vector in Type 4, namely $(4,2A_0+B_0,0)$, we note that we may pull back by an intermediate bielliptic cover as in \eqref{bielliptic cover 1} to produce $\mu$-semistable sheaves via push-forward of $\mu$-stable sheaves of Mukai vector $(2,\tilde{A}_0+\tilde{B}_0,0)$.

It remains to consider Type 6.  For $\bar{\v}=(3,A_0+B_0,0)$, we consider the generic extension $E$ in $\Ext^1(F,\OO(B_0))$, where $F\in M^s_\omega(2,A_0,0)$, which is non-empty, smooth, and one-dimensional by \cref{isotropic}.  Then for generic $\omega$ such that the slope of the first entry of $\Ext^1(-,-)$ is larger than the slope of the second (note that for $\omega=A_0+B_0$ they are equal), we get that $E$ is $\mu_\omega$-stable by \cite[Lemma 6.1 and Lemma 6.2]{CH15}.  Note that $E$ is locally free of Mukai vector $(3,A_0+B_0,0)$ and $E^{\oplus 2}$ is a $\mu_\omega$-semistable locally free sheaf of Mukai vector $\bar{\v}=(6,2(A_0+B_0),0)$, as required.

It remains to prove the isotropic case of the theorem.  The necessity of the condition $nl(\v)\mid\o$ was shown in \cref{isotropic}.  It remains to prove the existence direction.  We begin the proof of the converse direction by showing that if $nl(\v)=\o$, then $M_\omega^s(n\v)$ is smooth of dimension two and non-empty.  Indeed, if $\pi^*\v=l(\v)\w$, then $M^s_{\pi*\omega,X}(\w)$ must be non-empty since $\w$ is primitive and isotropic.  As $g^*\w=\w$, $G':=\Z/\o\Z$ acts on $M^s_{\pi^*\omega,X}(\w)$, and we will show that the open subset $$M^s_{\pi^*\omega,X}(\w)^\circ:=\Set{F\in M^s_{\pi^*\omega,X}(\w)\ | \ (g^i)^*F\ncong F,\mbox{ for all }0<i<\o}$$ is non-empty.  If $(g^d)^*F\cong F$ for some $F\in M^s_{\pi^*\omega,X}(\w)$ and a proper divisor $1<d$ of $\o$, then we can take an intermediate bielliptic surface $\tilde{S}$ with the same canonical cover $X$  such that $\ord(K_{\tilde{S}})=\frac{\o}{d}$ with notation as in \eqref{bielliptic cover 1}.  From $(g^d)^*F\cong F$, there exists $\tilde{E}\in M^s_{\pi'^*\omega,\tilde{S}}(\tilde{\v})$ such that $\tilde{\pi}^*\tilde{E}\cong F$ and thus $l(\tilde{\v})=1$.  From $g^*F\ncong F$ we must also have that $\tilde{E}\ncong\pi'^*E$ for any $E$ on $S$.  Setting $$\Fix(g^d):=\Set{F\in M^s_{\pi^*\omega,X}(\w)\ |\ (g^d)^*F\cong F},$$ we see that $\Fix(g^d)\subset M^s_{\pi^*\omega,X}(\w)$ is a proper closed one-dimensional subscheme since it is the image under the finite map $\tilde{\pi}^*$ of $M^s_{\pi'^*\omega,\tilde{S}}(\tilde{\v})$.  Indeed, if  $\tilde{E}\in M^s_{\pi'^*\omega,\tilde{S}}(\tilde{\v})$ lied on a two-dimensional component, then $\tilde{E}\cong\tilde{E}\otimes\OO_{\tilde{S}}(K_{\tilde{S}})$ and thus $\tilde{\pi}^*\tilde{E}$ could not be stable by \cref{lem:WhenIsPullbackStable}.  Ranging over all divisors $d>1$ of $\o$, we see that $M^s_{\pi^*\omega,X}(\w)^\circ$ is indeed non-empty.  But then $\pi_*(F)$ is stable for all $F\in M^s_{\pi^*\omega,X}(\w)^\circ$ by \cref{lem:WhenIsPushForwardStable}.  From the push-pull formula, we have $$\o\v =\pi_*\pi^*\v=l(\v)\pi_*(\w),$$ so $\pi_*(\w)=\frac{\o}{l(\v)}\v=n\v$ and thus $M^s_\omega(n\v)$ is smooth, two-dimensional, and non-empty, as claimed.
 
We now complete the proof of the existence direction by showing non-emptiness of $M^s_\omega(n\v)$ when $n$ is a proper

divisor of $\frac{\o}{l(\v)}$, and it suffices to do this for $\v=\u$ appearing in \cref{isotropic exceptions}.  

We begin with the case $l(\u)>1$.  Notice that $\o$ thus must either be 4 or 6, so $n=1$ and $l(\u)=2$ or 3 (the latter only in Type 7).  By inspection of Table \ref{isotropic exceptions} we see that there is a degree $l(\u)$ intermediate \'{e}tale bielliptic cover $\tilde{S}\mor[\pi']S$ such that $\pi'^*\u=l(\u)\tilde{\u}$ as in \eqref{bielliptic cover 1}.  By induction on $\o$ we may assume we have shown that $M_{\pi'^*H,\tilde{S}}(\tilde{\u})$ is non-empty, smooth, and 1-dimensional.  Since $\u$ is primitive, no bundle in $M_{\pi'^*\omega,\tilde{S}}(\tilde{\u})$ is $\overline{g}$-invariant, where $\overline{g}$ is the image of $g\in G'$ in the quotient group $G'/\langle g^{l(\u)}\rangle$ that defines $S$ as a quotient of $\tilde{S}$.  Thus $(\pi')_*(M_{\pi'^*\omega,\tilde{S}}(\tilde{\u}))\subset M_\omega(\u)$, so the latter is non-empty.

Now we treat the case $l(\u)=1$.  When $\u=(1,0,0)$, Umemura describes $M^s_\omega(n\u)$ in \cite[Theorem 2.15]{Ume75} and shows that in particular $M^s_\omega(n\u)\neq\varnothing$ if and only if $n\mid\o$.  It remains to consider the other cases where $l(\u)=1$ and $n$ is a proper divisor of $\o$.  By inspection of Table \ref{isotropic exceptions}, the other primitive isotropic Mukai vectors with $l(\u)=1$ occur in Type 2 (for $\u=(2,A_0,0)$, $(2,B_0,0)$), in Type 4 (for $\u=(2,A_0,0)$), and in Type 6 (for $\u=(3,A_0,0)$,  $(3,B_0,0)$).\footnote{In Type 7, if $\u=(r,bB_0,0)$ with $l(\u)=1$, then $\gcd(r,d_{p_B}(\u))=\gcd(r,6b)=1$, so by \cref{elliptic sheaves} we may reduce to the case $\u=(1,0,0)$.}  

Using the second type of intermediate bielliptic cover as in \eqref{bielliptic cover 2}, we can pull the Mukai vectors $(2,A_0,0)$ and $(3,A_0,0)$ back to a non-primitive multiple of $\tilde{\u}$, where $\tilde{\u}=(1,\tilde{A}_0,0)$.  Again we have $\pi'_*(M^s_{\pi'^*\omega,\tilde{S}}(\tilde{'u}))\subset M^s_\omega(\u)$, so the latter is non-empty.  In Type 4, we must also show that $M_\omega(2\u)\ne\varnothing$ for $v_0=(2,A_0,0)$.  But in this case $\pi'^*\u=2(1,\tilde{A}_0,0)=2\tilde{\u}$, and after twisting by $\OO_{\tilde{S}}(-\tilde{A}_0)$ we may identify $M_{\pi'^*\omega,\tilde{S}}(2\tilde{\u})$ with $M_{\pi'^*\omega,\tilde{S}}(2,0,0)$ which is isomorphic to two disjoint copies of $\Pic^0(S')$ by \cite[Theorem 2.15]{Ume75}, where $S'$ is the Type 1 intermediate bielliptic surface between $A\times B$ and $\tilde{S}$.  These two copies of $\Pic^0(S')$ parametrize bundles whose determinants have linear equivalence classes differing by the torsion element in $H^2(\tilde{S},\Z)$.  Moreover, the subgroup $H<G$ acts on $\tilde{S}$ and switches these two components.  Letting $h$ be the generator of $H$ and $\tilde{E}$ be on one component of $M^s_{\pi'^*\omega,\tilde{S}}(2\tilde{\u})$, we must have $h^*\tilde{E}$ on the other component, so $\tilde{E}\ncong h^*\tilde{E}$ and thus $\pi'_*(\tilde{E})\in M^s_\omega(2\u)$, which is therefore nonempty as claimed.

It only remains to consider the Mukai vectors $(2,B_0,0)$ in Type 2 and $(3,B_0,0)$ in Type 6.  But this is the content of Lemma \ref{random case} below.  

\end{proof}
\begin{Lem}\label{random case} Suppose that $S$ is of Type 2 and $\u=(2,B_0,0)$ or $S$ is of Type 6 and $\u=(3,B_0,0)$.  Then $M^s_\omega(\u)\neq\varnothing$.
\end{Lem}
\begin{proof}  We first observe that in either case $\pi^*B_0=B_X$, so $$\C=H^0(X,\OO(B_X))=H^0(S,\pi_*\pi^*\OO(B_0))=\bigoplus_{i=0}^{\o-1}H^0(S,\OO(B_0+iK_S)),$$ so one and only one of $B_0+iK_S$, $0\leq i<\o$ is effective, and it is isolated in its linear system.  For simplicity, say its $B_0$.  Then as $\chi(B_0)=0$ and $h^2(B_0)=h^0(-B_0+K_S)=0$, we must have $h^1(B_0)=1$.  Thus $\Ext^1(\OO(B_0),\OO(K_S))\cong\Ext^1(\OO,\OO(B_0))=\C$.  Let $E$ be the unique non-split extension in $\Ext^1(\OO(B_0),\OO(K_S))$, \begin{equation}\label{ses}0\to \OO(K_S)\to E\to \OO(B_0)\to 0.\end{equation} 

As there are no walls for isotropic Mukai vectors (since $\Delta(v)=v^2$ for our surfaces), we may suppose that $\omega=A_0+B_0$.  If $E$ is not stable, then, as $\u$ is primitive and isotropic and thus cannot be properly semistable, there exists a subbundle $\OO(L)\subset E$ of rank one with $\mu_\omega(L)>\mu_\omega(E)=\frac{1}{2}>0$.  It follows that the composition $\OO(L)\to E\to \OO(B_0)$ is non-zero, so $\mu_\omega(L)\leq\mu_\omega(\OO(B_0))$.  As $E$ is non-split, we must in fact have $\mu_\omega(L)<\mu_\omega(\OO(B_0))=1$.  But it is impossible for the integer $L.\omega$ to be between $1/2$ and 1.  So $E\in M^s_\omega(2,B_0,0)$.  Furthermore, notice that $h^0(E(K_S))=h^1(E(K_S))=1$ from taking cohomology in \eqref{ses}.  

To construct $F\in M^s_\omega(3,B_0,0)$, consider the unique extension in $\Ext^1(E,\OO)=H^1(E(K_S))=\C$, $$0\to \OO\to F\to E\to 0.$$  Again, we may suppose that there is a stable subbundle $G\subset F$ with $\mu_H(G)>\mu_H(F)=\frac{1}{3}$ and $\rk G=1,2$.  Thus the composition $G\to F\to E$ is non-zero so that $\mu_H(G)<\mu_H(E)=\frac{1}{2}$, with strict inequality because the extension is non-split.  This is impossible for a rational number of the form $\frac{c_1(G).H}{\rk G}$ with $\rk G=1,2$.  Thus $F\in M^s_\omega(3,B_0,0)$.
\end{proof}
Having resolved the classification of the Mukai vectors of semistable sheaves, we now make further use of Fourier-Mukai transforms to do the same for the Mukai vectors of $\sigma$-semistable objects.
\begin{Thm}\label{Bridgeland non-empty}
For any $\sigma\in\Stabd(S)$, $\mS\ne\varnothing$ if and only if $\v^2\geq 0$.
\end{Thm}
\begin{proof}
Since we are interested at the moment in semistable objects, it suffices to prove the claim for $\v$ primitive.  Indeed, suppose $\v=m\v_0$ with $\v_0$ primitive.  If $E_0\in M_{\sigma,S}(\v_0)$, then $E=E_0^{\oplus m}\in M_{\sigma,S}(\v)$.  By the remarks preceeding \cite[Lemma 8.2]{Bri07}, semistability is a closed condition, so it also suffices to suppose that $\sigma$ is generic with respect to $\v$ so that every $\sigma$-semistable object of class $\v$ is stable since these remain at least semistable at the boundary.

The necessity of $\v^2\geq 0$ is then just \cref{lem:MukaiVectorsOfStableObjects}, and it remains to prove that this is sufficient.

From the invariance of $J_\sigma(\v)$ under autoequivalences, as in \cref{Motivic invariant}, we can assume that $\v$ is positive in the sense of \cite[Definition 0.1]{Yos01}.  Indeed, if $r\neq 0$, then we can apply $[1]$ to make $r>0$ if necessary.  If $r=0$, then $0\leq \v^2=c_1^2$, so $c_1$ or $-c_1$ is the numerical class of an effective divisor so that we may assume $c_1$ is effective, after applying $[1]$ if necessary.  Thus $\v$ is positive if $s\neq 0$.  Finally, we are reduced to the case $\v=(0,c_1,0)$, where again we may assume that $c_1$ is the numerical class of an effective divisor.  We can tensor with $\OO(D)$ for any $D\in\NS(S)$ such that $D.c_1\neq 0$ (we can even choose $D$ such that $D.c_1=1$ since $\NS(S)$ is unimodular) to obtain a Mukai vector of the form $(0,c_1,D.c_1)$, which is positive.  

Therefore $\v$ is the Mukai vector of a semistable sheaf, and by \cref{Thm:GiesekerChamber}, there is a chamber $\GG$ in $\Stabd(S)$ corresponding to Gieseker stability. Using the invariance of $J_\sigma(\v)$ under changes in $\sigma$, we may take $\sigma\in\GG$ so $J_\sigma(\v)=e(M_\omega(\v))$, the virtual Hodge polynomial of the moduli space $M_\omega(\v)$ of $\omega$-Gieseker semistable sheaves of Mukai vector $\v$.  But this is non-trivial by \cref{Thm:Non-emptinessOfGieseker}.
\end{proof}

\part{Wall-crossing for Bridgeland moduli spaces on bielliptic surfaces}\label{Part:WallCrossing}
Now that we have established the existence of coarse moduli spaces for Bridgeland semistable objects, we move on to studying how these varieties behave as we cross a wall $W\subset\Stabd(S)$.  We will see that crossing $W$ always induces a birational contraction.
Thus the reduction to a finite list of Mukai vectors in \cref{FM reduction} has a geometric explanation. 
Indeed, the motivic invariants of the moduli spaces on opposite sides of a wall are equal because the moduli spaces themselves are birationally equivalent.
This is the content of the main theorem we prove in this part:
\begin{Thm}\label{Thm:MainTheorem1}
Let $\v\in\Hal(S,\Z)$ satisfy $\v^2\geq 0$, and let $\sigma,\tau\in\Stabd(S)$ be generic stability conditions with respect to $\v$ (that is, they are not contained in any wall for $\v$).
\begin{enumerate}
\item\label{enum:MT1-two moduli are birational}  The two moduli spaces $M_{\sigma,S}(\v)$ and $M_{\tau,S}(\v)$ of Bridgeland semistable objects are birational to each other.
\item\label{enum:MT1-birational map given by FM transform} More precisely, there is a birational map induced by a derived (anti-)autoequivalence $\Phi$ of $\Db(S)$ in the following sense: there exists a common open subset $U\subset M_{\sigma,S}(\v),U\subset M_{\tau,S}(\v)$ such that for any $u\in U$, the corresponding objects $E_u\in M_{\sigma,S}(\v)$ and $F_u\in M_{\tau,S}(\v)$ satisfy $F_u=\Phi(E_u)$.  
\end{enumerate}

\end{Thm}
\section{Line bundles on moduli spaces, wall-crossing, and birational transformations}\label{sec:LineBundlesWallCrossBirat}
\subsection{The Bayer-Macr\`{i} map} We introduce the first main tool we will use for relating wall-crossing in the Bridgeland stability manifold to birational transformations on the moduli spaces.  We denote by $Y$ either a bielliptic surface $S$ or its abelian canonical cover $X$, and we recall the definition of the Donaldson-Mukai morphism.  Fix a Mukai vector $\v\in\Hal(Y,\Z)$, a line bundle $L\in\Pic(Y)$ with $c_1(L)=c_1(\v)$, a stability condition $\sigma\in\Stabd(Y)$, and a universal family $\EE\in\Db(M_\sigma(\v,L)\times Y)$.\footnote{In case only a quasi-universal family exists, it is easy to adjust the definitions here by dividing by the similitude.  See \cite[Section 4]{BM14a} for more details.}  Then we have the following definition.
\begin{Def} Set $K(Y)_{\v}:=\Set{x\in K(Y)\ |\ \langle\v(x),\v\rangle=0}$.  The Donaldson-Mukai morphism from $K(Y)_\v$ to $\Pic(M_\sigma(\v,L))$ is defined by:  
\begin{equation}\label{eqn:DonaldsonMukai}
\begin{matrix}
\theta_{\v,\sigma}:& K(Y)_\v & \to & \Pic(M_{\sigma}(\v,L))\\
 & x & \mapsto & \det (p_{M_\sigma(\v,L)!}(\EE \otimes p_Y^*(x^{\vee}))).
\end{matrix} 
\end{equation}
More generally, for a scheme $Z$ and a family $\EE\in\Db(Z\times Y)$ over $Z$ of objects in $M_\sigma(\v,L)$, there is a Donaldson-Mukai morphism $\theta_\EE:K(Y)_\v\to\Pic(Z)$ associated to $\EE$, defined as in \eqref{eqn:DonaldsonMukai}, which satisfies $\theta_\EE=\lambda_\EE^*\theta_{\v,\sigma}$, where $\lambda_\EE:Z\to M_{\sigma}(\v,L)$ is the associated classifying map.  See \cite[Section 8.1]{HL10} for more details.

Recal that $Z_\sigma$ can be written as $Z_\sigma(E)=\langle\mho_\sigma,\v(E)\rangle$ for all $E\in\Db(X)$.  Setting 
\begin{equation}\label{eqn:def of xi}
\xi_\sigma:=\Im\frac{ \mho_\sigma}{\langle \mho_\sigma, \v \rangle}
\in \v^\perp,
\end{equation} we define the numerical divisor class
\begin{equation}\label{eqn:def of ell(sigma)}
\ell_\sigma:=\theta_{\v,\sigma}(\xi_\sigma)\in\Num(M_\sigma(\v,L)),
\end{equation}
where we abuse notation by also using $\theta_{\v,\sigma}$ for the extension of \eqref{eqn:DonaldsonMukai} to the Mukai lattice.
\end{Def}

From \cref{prop:wall-and-chamber} it follows that the moduli space $M_\sigma(\v,L)$ and $\EE$ remain constant when varying $\sigma$ in a chamber for $\v$, so for each chamber $\CC$, we get the Bayer-Macr\`{i} map $$l_{\CC}:\CC\to\Num(M_\CC(\v,L)),\;\;\sigma\mapsto \ell_\sigma,$$ where the notation $M_\CC(\v,L)$ denotes the fixed moduli space.  From our proof of the projectivity of $M_{\sigma,S}(\v,L)$ in \cref{thm:projective coarse moduli spaces} along with the argument
in \cite{MYY14} or \cite{BM14a}, we have the following result:
\begin{Thm}[{cf. \cite[Theorem 5.15]{MYY14} in the abelian case}]\label{Thm:NefAmpleDivisor}
For a generic $\sigma$,
$\ell_\sigma=\theta_{\v,\sigma}(\xi_\sigma)$ is an ample divisor on $M_\sigma(\v,L)$.
\end{Thm}

\subsection{Wall-crossing and Birational transformations} Considering $\ell_\sigma$ as $\sigma$ approaches a wall makes explicit the connection between crossing walls in $\Stabd(Y)$ and birational transformations of the moduli space $M_\sigma(\v,L)$.  Let $\v\in\Hal(Y,\Z)$ with $\v^2>0$, and let $W$ be a wall for $\v$.  We say $\sigma_0\in W$ is \emph{generic} if it does not belong to any other wall, and we denote by $\sigma_+$ and $\sigma_-$ two generic stability conditions nearby $W$ in two opposite adjacent chambers.  Then all $\sigma_\pm$-semistable objects are still $\sigma_0$-semistable, and thus the universal familes $\EE^\pm$ on $M_{\sigma_\pm}(\v)\times Y$ induce nef divisors $l_{\sigma_0,\pm}$ on $M_{\sigma_\pm}(\v)$ by $$l_{\sigma_0,\pm}:=\theta_{\v,\sigma_\pm}(\xi_{\sigma_0}).$$  The main result about $l_{\sigma_0,\pm}$ is the following. It can be proved exactly as in \cite[Thm. 1.4(a)]{BM14a},\cite[Thm. 11.3]{Nue14b}:
\begin{Thm}\label{Thm:WallContraction} Let $\v\in\Hal(Y,\Z)$ satisfy $\v^2>0$, and let $\sigma_\pm$ be two stability conditions in opposite chambers nearby a generic $\sigma_0\in W$.  Then:
\begin{enumerate}
\item The divisors $l_{\sigma_0,\pm}$ are semiample on $M_{\sigma_\pm}(\v)$.  In particular, they induce contractions $$\Sigma^\pm:M_{\sigma_\pm}(\v)\to\overline{M}_\pm,$$ where $\overline{M}_\pm$ are normal projective varieties.  When $Y=X$, an abelian surface, $\overline{M}_\pm$ are irreducible.
\item For any curve $C\subset M_{\sigma_\pm}(\v)$, $l_{\sigma_0,\pm}.C=0$ if and only if the two objects $\EE_c^\pm$ and $\EE_{c'}^\pm$ corresponding to two general points $c,c'\in C$ are S-equivalent.  In particular, the curves contracted by $\Sigma^\pm$ are precisely the curves of objects that are S-equivalent with respect to $\sigma_0$.
\end{enumerate}
\end{Thm}

This theorem leads us to the following definition describing a wall $W$ in terms of the geometry of the induced morphisms $\Sigma^\pm$.
\begin{Def}
We call a wall $W$:
\begin{enumerate}
\item a \emph{fake wall}, if there are no curves contracted by $\Sigma^\pm$;
\item a \emph{totally semistable wall}, if $M_{\sigma_0}^s(\v)=\varnothing$;
\item a \emph{flopping wall}, if we can identify $\overline{M}_+=\overline{M}_-$ and the induced map $M_{\sigma_+}(\v)\dashrightarrow M_{\sigma_-}(\v)$ induces a flopping contraction;
\item a \emph{divisorial wall}, if the morphisms $\Sigma^\pm$ are both divisorial contractions;
\item a \emph{$\P^1$-wall}, if the morphisms $\Sigma^\pm$ are both $\P^1$-fibrations.
\end{enumerate}
\end{Def}

A non-fake wall $W$ such that $\codim(M_{\sigma_\pm}(\v)\backslash M^s_{\sigma_0}(\v))\geq 2$ is necessarily a flopping wall by the arguments in \cite[Thm. 1.4(b)]{BM14a} and \cite[Thm. 11.3]{Nue14b}.  

\section{The hyperbolic sublattice associated to a wall}
Our approach to classifying the birational behavior induced by crossing a given wall $W$ will involve a certain lattice associated to $W$.  From their definition, these walls are associated to the existence of another Mukai vector with the same phase as $\v$, so it is natural to consider the set of these ``extra'' classes, as in the following definition.  As it turns out, this set will contain all of the Mukai vectors we are interested in.
\begin{PropDef}\label{hyperbolic} To a wall $W$, let $\HH_{W}\subset\Hal(S,\Z)$ be the set of Mukai vectors $$\HH_W:=\Set{\w\in\Hal(X,\Z)\ |\ \Im\frac{Z(\w)}{Z(\v)}=0\mbox{ for all }\sigma\in W}.$$
Then $\HH_{W}$ has the following properties:
\begin{enumerate}
\item It is a primitive sublattice of rank two and of signature $(1,-1)$ (with respect to the restriction of the Mukai form).
\item Let $\sigma_+,\sigma_-$ be two sufficiently close and generic stability conditions on opposite sides of the wall $W$, and consider any $\sigma_+$-stable object $E\in M_{\sigma_+,S}(\v)$.  Then any HN-filtration factor $A_i$ of $E$ with respect to $\sigma_-$ satisfies $\v(A_i)\in\HH_{W}$.
\item If $\sigma_0$ is a generic stability condition on the wall $W$, the conclusion of the previous claim also holds for any $\sigma_0$-semistable object $E$ of class $\v$.
\item Similarly, let $E$ be any object with $\v(E)\in\HH_{W}$, and assume that it is $\sigma_0$-semistable for a generic stability condition $\sigma_0\in W$.  Then every Jordan-H\"{o}lder factor of $E$ with respect to $\sigma_0$ will have Mukai vector contained in $\HH_{W}$.
\end{enumerate}
\end{PropDef}
\begin{proof} The proof of \cite[Proposition 5.1]{BM14a} carries over word for word.
\end{proof}

We would like to characterize the type of the wall $W$, i.e. the type of birational transformation induced by crossing it, in terms of the lattice $\HH_{W}$.  We will find it helpful to also go in the opposite direction as in \cite[Definition 5.2]{BM14a}:
\begin{Def} Let $\HH\subset\Hal(S,\Z)$ be a primitive rank two hyperbolic sublattice containing $\v$.  A \emph{potential wall} $W$ associated to $\HH$ is a connected component of the real codimension one submanifold consisting of those stability conditions $\sigma$ such that $Z_\sigma(\HH)$ is contained in a line.
\end{Def}

We will also have cause to consider two special convex cones in $\HH_{\R}$.  The first is defined as follows (see \cite[Definition 5.4]{BM14a}):

\begin{Def} Given any primitive hyperbolic lattice $\HH\subset \Hal(S,\Z)$ of rank two containing $\v$, denote by $P_{\HH}\subset \HH_{\R}$ the cone generated by classes $\u\in\HH$ with $\u^2\geq 0$ and $\langle \v,\u\rangle>0$.  We call $P_{\HH}$ the \emph{positive cone} of $\HH$, and a class in $P_{\HH}\cap\HH$ a \emph{positive class}.
\end{Def}

The next cone, called the \emph{effective cone} and whose integral classes are \emph{effective classes}, is classified by the following proposition (see \cite[Proposition 5.5]{BM14a} for the analogue in the K3 case):

\begin{Prop} Let $W$ be a potential wall associated to a hyperbolic rank two sublattice $\HH\subset\Hal(S,\Z)$.  For any $\sigma\in W$, let $C_{\sigma}\subset\HH_{\R}$ be the cone generated by classes $\u\in\HH$ satisfying the two conditions $$\Re\frac{Z_{\sigma}(\u)}{Z_{\sigma}(\v)}>0\mbox{ and }\u^2\geq 0.$$  Then this cone does not depend on the choice of $\sigma\in W$, so we may and will denote it by $C_{W}$.  Moreover, it contains $P_{\HH}$.

If $\u\in C_{W}$, then there exists a $\sigma$-semistable object of class $\u$ for every $\sigma\in W$, and if $\u\notin C_{W}$, then for a generic $\sigma\in W$ there does not exist a $\sigma$-semistable object of class $\u$.
\end{Prop}
\begin{proof} The proof is identical to that of \cite[Proposition 5.5]{BM14a} except that for the statements about the existence of semistable objects we must use \cref{Bridgeland non-empty}.  
\end{proof}
In fact, it is not difficult to see that $P_\HH=C_W$, but we introduce  $C_W$ because of the following remark (see also \cite[Remark 5.6]{BM14a}):
\begin{Rem} From the positivity condition on $\Re\frac{Z_{\sigma}(\u)}{Z_{\sigma}(\v)}$, it is clear that $C_{W}$ contains no line through the origin, i.e. if $\u\in C_{W}$ then $-\u\notin C_{W}$.  Thus there are only finitely many classes in $C_{W}\cap(\v-C_{W})\cap\HH$.

We use this fact to make the following assumption: when we refer to a generic $\sigma_0\in W$, we mean that $\sigma_0$ is not in any of the other walls associated to the finitely many classes in $C_{W}\cap(\v-C_{W})\cap\HH$.  Likewise, $\sigma_{\pm}$ will refer to stability conditions in adjacent chambers to $W$ in this more refined wall-and-chamber decomposition.
\end{Rem}

Finally, we single-out two types of primitive hyperbolic lattices as the nature of our arguments differ greatly between them:

\begin{Def} We say that $W$ is \emph{isotropic} if $\HH_{W}$ contains an isotropic class and \emph{non-isotropic} otherwise.
\end{Def}

\begin{Thm}\label{classification of walls}
Let $\HH\subset \Hal(S,\Z)$ be a primitive hyperbolic rank two sublattice containing $\v$, and let $W\subset\Stabd(X)$ be a potential wall associated to $\HH$.  

The set $W$ is a totally semistable wall if and only if one of the following conditions hold:
\begin{description}
\item[(TSS1)] there exists an isotropic class $\u\in\HH$ with $l(\u)=\o$ and $\langle \v,\u\rangle=1$; or
\item[(TSS2)] there exists a primitive isotropic class $\u\in\HH$ such that  $\langle\v,\u\rangle=2=l(\u)=\o=l(\v)$ and $\v^2=4$.
\end{description}
In addition,
\begin{enumerate}
\item\label{thm:Classification,Divisorial} The set $W$ is a wall inducing a divisorial contraction if one of the following conditions hold:
\begin{description*}
\item[(Hilbert-Chow)] $\v^2\geq 4$ and there exists an isotropic class $\u$ with $\langle \v,\u\rangle=1$ and $l(\u)=\o$; or
\item[(Li-Gieseker-Uhlenbeck, general)] $\v^2\geq 4$ (and $\o\ne 2,3$ if $\v^2=4$) and there exists a primitive isotropic class $\u\in\HH$ with $\langle \v,\u\rangle=2$ and $l(\u)=\o$; or
\item[(Li-Gieseker-Uhlenbeck, $\o=2$)] $\v^2=4$, $\o=2$, $l(\v)=1$, and there exists a primitive isotropic class $\u\in\HH$ with $\langle \v,\u\rangle=2$ and $l(\u)=\o$; or
\item[($\o=2$ exceptional case)] $\v^2=6$ and there exists a primitive isotropic $\u\in\HH$ with $\langle\v,\u\rangle=3$, $l(\u)=\o$ and $3\mid(\v-\u)$.
\item[($\o=3$ exceptional case)] $\v^2=6$ and there exists a primitive isotropic $\u\in\HH$ with $\langle\v,\u\rangle=3=\o=l(\u)$. 
\end{description*} 
\item\label{thm:Classification,Fibration} The set $W$ is a wall inducing a $\P^1$-fibration on $M_{\sigma_+}(\v)$ if there exists a primitive isotropic class $\u\in\HH$ such that  $\langle\v,\u\rangle=2=l(\u)=\o=l(\v)$ and $\v^2=4$.
\item\label{thm:Classification,Flops}
Otherwise, if $\v$ is primitive with $\v^2\geq 4$ and $\v$ can be written as the sum 
$\v = \a_1 + \a_2$ with $\a_i\in P_\HH$ then $W$ induces a small contraction.
\item In all other cases, $W$ is either a fake wall or not a wall at all.
\end{enumerate}

\end{Thm}

\section{Dimension estimates for substacks of Harder-Narasimhan filtrations}
In this section we recall our last tool for studying the connection between the lattice $\HH_W$ and the birational tranformation induced by crossing $W$, namely substacks of Harder-Narasimhan filtrations.  We work over the bielliptic surface $S$ and continue to use the notation $\sigma_\pm$ and $\sigma_0$.

For given Mukai vectors $\v_1,\v_2,\dots,\v_s\in\HH_W$ and $\sigma_0\in W$,
let $\FF(\v_1,\dots,\v_s)$ be the stack of filtrations of $\sigma_0$-semistable objects of Mukai vector $\v=\v_1+\cdots+\v_s$:
for a scheme $T$, 
\begin{equation}\FF(\v_1,\dots,\v_s)(T):=\Set{0 \subset \FF_1 \subset \cdots \subset \FF_s\ | \ \FF_i/\FF_{i-1} \in \MM_{\sigma_0}(\v_i)(T), 1\leq i\leq s,\FF_s\in\MM_{\sigma_0}(\v)(T)}.
\end{equation}
We showed in \cite[Theorem 4.2]{NY19} that $\FF(\v_1,\dots,\v_s)$ is an Artin stack of finite type and there is a natural morphism
\begin{equation*}
\begin{matrix}
\FF(\v_1,\dots,\v_s)&\longrightarrow&\FF(\v_1,\dots,\v_{s-1})\times\MM_{\sigma_0}(\v_s)\\
(0\subset\FF_1\subset\cdots\subset\FF_s)&\longmapsto&((0\subset\FF_1\subset\cdots\subset\FF_{s-1}),\FF_s/\FF_{s-1})
\end{matrix},
\end{equation*} 
and hence a morphism
$$
\Pi:\FF(\v_1,\dots,\v_s) \to \prod_{i=1}^s \MM_{\sigma_0}(\v_i).
$$
We consider the open substack $$\FF(\v_1,\dots,\v_s)^*:=\Pi^{-1}\left(\prod_{i=1}^s(\MM_{\sigma_-}(\v_i)\cap\MM_{\sigma_0}(\v_i))\right)\subset \FF(\v_1,\dots,\v_s)$$ where each $\FF_i/\FF_{i-1}$ is $\sigma_-$-semistable as well, where $\sigma_-$ is sufficiently close to
$\sigma_0$.  The intersections are taken within Lieblich's ``mother of all moduli spaces'' \cite{Lie}.  Assuming further that $\v_1,\dots,\v_s$ are the Mukai vectors of the semistable factors of the Harder-Narasimhan
filtration with respect to $\sigma_-$ of an object $E\in\MM_{\sigma_0}(\v)$, with $\v=\sum_{i=1}^s\v_i$, then the natural map 
\begin{equation}
\begin{matrix}
 \FF(\v_1,\dots,\v_s)^*& \to &\MM_{\sigma_0}(\v)\\
 (0\subset\FF_1\subset\cdots\subset\FF_s)&\mapsto&\FF_s
\end{matrix}   
\end{equation}
 is injective with image the substack of $\MM_{\sigma_0}(\v)$ parameterizing objects with Harder-Narasimhan filtration factors having Mukai vectors
$\v_1,\dots,\v_s$.  In this case, we get the following result:
\begin{Prop}\label{Prop:HN codim} Let $S$ be a bielliptic surface, and suppose that $\v_1,\dots,\v_s$ are the Mukai vectors of the semistable factors of the Harder-Narasimhan filtration with respect to $\sigma_-$ of an object $E\in\MM_{\sigma_+}(\v)$, where $\v=\sum_{i=1}^s\v_i$ satisfies $\v^2>0$.  Then letting $\FF(\v_1,\dots,\v_s)^o:=\FF(\v_1,\dots,\v_s)^*\cap\MM_{\sigma_+}(\v)$, where the intersection is taken in $\MM_{\sigma_0}(\v)$, we have
\begin{equation}\label{eqn:HNFiltrationCodim}\codim\FF(\v_1,\dots,\v_s)^o
\geq\sum_{i=1}^s \left(\v_i^2-\dim\MM_{\sigma_-}(\v_i)\right)+\sum_{i<j}\langle \v_i,\v_j\rangle,
\end{equation}
where the codimension is taken with respect to $\MM_{\sigma_+}(\v)$.
\end{Prop}
\begin{proof}
We apply \cite[Theorem 4.2]{NY19}, where the hypothesis of the theorem is met since for $E,E'\in\AA_{\sigma_0}$ with $\phi_{\min}(E')>\phi_{\max}(E)$, Serre duality gives $$\Hom(E,E'[2])=\Hom(E',E(K_S))=0,$$ where the last equality follows since $S$ is numerically $K$-trivial so that $\phi_{\max}(E(K_S))=\phi_{\max}(E)$.

Noting that $\FF(\v_1,\dots,\v_s)^o$ is an open substack of $\FF(\v_1,\dots,\v_s)^*$ by openness of stability, we get that $\dim\FF(\v_1,\dots,\v_s)^o\leq\dim\FF(\v_1,\dots,\v_s)^*$, with equality if and only if the component of $\FF(\v_1,\dots,\v_s)^*$ of largest dimension contains a $\sigma_+$-semistable object.  As $\v^2>0$, $\dim\MM_{\sigma_+}(\v)=\v^2$ by the proofs of \cref{Bridgeland non-empty,Thm:Non-emptinessOfGieseker}, so 
\begin{equation}
    \begin{split}
        \codim\FF(\v_1,\dots,\v_s)^o&=\v^2-\dim\FF(\v_1,\dots,\v_s)^o\geq\v^2-\dim\FF(\v_1,\dots,\v_s)^*\\
        &=\v^2-\left(\sum_{i=1}^s\dim\MM_{\sigma_-}(\v_i)+\sum_{i<j}\langle\v_i,\v_j\rangle\right)\\
        &=\sum_{i=1}^s\left(\v_i^2-\dim\MM_{\sigma_-}(\v_i)\right)+\sum_{i<j}\langle\v_i,\v_j\rangle,
    \end{split}
\end{equation}
where the equality in the second line follows from \cite[Theorem 4.2]{NY19}.  
\end{proof}

We will use \cref{Prop:HN codim} to determine precisely when the locus of strictly $\sigma_0$-semistable objects has small codimension to determine the conditions for $W$ to be totally semistable or divisorial.  We begin, however, with the simplest case of a non-isotropic wall.  We will show that if a potential wall $W$ is non-isotropic, then it is either a flopping wall or not a wall at all.
\begin{Prop}\label{Prop:HN filtration all positive classes}
As above, let $\FF(\v_1,\dots,\v_n)^o$ be the substack of $\MM_{\sigma_+}(\v)$ parametrizing objects with $\sigma_-$ Harder-Narasimhan filtration factors of classes $\v_1,\dots,\v_n$ (in order of descending phase with respect to $\sigma_-$), and suppose that $\v_i^2>0$ for all $i$.  Then $\codim\FF(\v_1,\dots,\v_n)^o> 2$. In particular, if $W$ is an actual wall, then it is a flopping wall.\end{Prop}
\begin{proof}
By the proofs of \cref{Bridgeland non-empty,Thm:Non-emptinessOfGieseker}, $\v_i^2>0$ implies that $\dim\MM_{\sigma_-}(\v_i)=\v_i^2$.  Thus by Proposition \ref{Prop:HN codim}, $$\codim\FF(\v_1,\dots,\v_n)^o\geq\sum_{i<j}\langle \v_i,\v_j\rangle.$$  But as $\v_i^2\geq 2$ and $\HH$ has signature $(1,-1)$, we must have $$\langle \v_i,\v_j\rangle>\sqrt{\v_i^2 \v_j^2}\geq 2,$$ for $i<j$.  It follows that $$\codim\FF(\v_1,\dots,\v_n)^o > n(n-1)\geq 2,$$ as $n\geq 2$.

If $W$ is actual wall, then $\Sigma^+$ contracts some curves of $S$-equivalent objects with respect to $\sigma_0$, so it must be a flopping contraction.  
\end{proof}
It follows from the proposition that in order for there to be more interesting wall-crossing behavior, $\HH$ must contain some class $\u$ with $\u^2= 0$, which we investigate next.

\section{Isotropic walls}\label{Sec:IsotropicWalls}
We start with a simple structural result about the cone $C_W$ and the lattice $\HH$.

\begin{Lem}\label{isotropic lattice} Assume that there exists an isotropic class $\u\in\HH$.  Then there are two effective, primitive, isotropic classes $\u_1$ and $\u_2$ in $\HH$, which satisfy $C_W=P_{\HH}=\R_{\geq 0}\u_1+\R_{\geq 0}\u_2$, $M_{\sigma_0}(\u_i)=M^s_{\sigma_0}(\u_i)$, and $\langle\v',\u_i\rangle\geq 0$ for $i=1,2$ and any $\v'\in P_{\HH}$.
\end{Lem}
\begin{proof}Let $\u\in\HH$ be a primitive isotropic class, and up to replacing $\u$ by $-\u$, we can and will assume that $\u_1=\u\in C_{W}$.  As $\v$ and $\u$ must be linearly independent, they form a basis of $\HH_{\Q}$.  Solving $(a\v+b\u)^2=0$ for $a,b\in\Q$ gives the existence of a second integral primitive isotropic class $\u_2$, which can be chosen to be effective.  Then clearly $C_{W}=P_{\HH}=\R_{\geq 0}\u_1+\R_{\geq 0}\u_2$.  The claim $\langle \v',\u_i\rangle\geq 0$ is then clear.  

As $\u_i$ spans an extremal ray of $C_W$, $W$ cannot be a wall for $\u_i$.  Indeed, if for $\u=\u_i$ we have $\u=\a+\b$ with $\a,\b\in P_{W}=C_{W}$ then $\langle \a,\b\rangle>\sqrt{\a^2 \b^2}\geq 0$, since $\u$ primitive forces $\a$ and $\b$ to be linearly independent.  Thus $$0=\u^2=\a^2+2\langle \a,\b\rangle+\b^2\geq 2\langle \a,\b\rangle>0,$$ a contradiction.  Thus $M_{\sigma_\pm}(\u_i)=M_{\sigma_0}(\u_i)=M^s_{\sigma_0}(\u_i)$, as claimed.
\end{proof}

We continue by using \cref{Prop:HN codim} to determine the conditions under which $W$ is a totally semistable wall or the strictly $\sigma_0$-semistable locus has codimension one.
\begin{Prop}\label{Prop:TotallySemistableCodimOne}
For a potential wall $W$ for the Mukai vector $\v\in\Hal(S,\Z)$ with $\v^2>0$, assume that the associated lattice $\HH_W$ is isotropic.  
\begin{enumerate}
    \item\label{enum:TSS} If $W$ is totally semistable, then there is an effective primitive isotropic $\u\in\HH_W$ such that either 
    \begin{enumerate}
        \item\label{enum:HilbertChow} $\langle\v,\u\rangle=1$ and $l(\u)=\o$, or 
        \item\label{enum:TSSRankTwoOrderTwoException} $\langle\v,\u\rangle=2=l(\u)=\o=l(\v)$ and $\v^2=4$.
    \end{enumerate}
    \item\label{enum:CodimOne} If $\codim(M_{\sigma_+}(\v)\backslash M_{\sigma_0}^s(\v))=1$, then there is an effective primitive isotropic $\u\in\HH_W$ such that either 
    \begin{enumerate}
        \item\label{enum:CodimOneRankOne} $\langle\v,\u\rangle=1$ and $l(\u)<\o$, or
        \item\label{enum:CodimOneRankTwo} $\langle\v,\u\rangle=2$ and $l(\u)=\o$, or
        \item\label{enum:CodimOneRankThreeOrderTwoThreeException} $\langle\v,\u\rangle=3=l(\u)=\o$,  and $\v^2=6$.
    \end{enumerate}
    \item In all other cases, $\codim(M_{\sigma_+}(\v)\backslash M_{\sigma_0}^s(\v))\geq 2$.
\end{enumerate}
\end{Prop}
\begin{proof}
For a given $E\in M_{\sigma_+}(\v)$, let the Harder-Narasimhan filtration of $E$ with respect to $\sigma_-$ correspond to a decomposition $\v=\sum_i \a_i$.  Using Proposition \ref{Prop:HN codim}, we shall estimate the codimension of the sublocus $\FF(\a_1,\ldots,\a_n)^o$ of destabilized objects,  which is at least
 \begin{equation}
\sum_i (\a_i^2-\dim \MM_{\sigma_-}(\a_i))+\sum_{i<j}\langle \a_i,\a_j \rangle.
\end{equation}

Suppose first that $\a_1=b_1 \u_1$ and $\a_2=b_2 \u_2$, and for the sake of brevity, assume that $l(\u_1)\leq l(\u_2)$ without loss of generality.  Then our estimate becomes
 \begin{equation}\label{eq:l=2 case II}
\begin{split}
\codim\FF(\a_1,\dots,\a_n)^o\geq& \sum_i (\a_i^2-\dim \MM_{\sigma_-}(\a_i))+\sum_{i<j}\langle \a_i,\a_j \rangle\\
\geq &
-\left\lfloor\frac{b_1l(\u_1)}{\o}\right\rfloor-\left\lfloor\frac{b_2l(\u_2)}{\o}\right\rfloor+b_1 b_2 \langle \u_1,\u_2 \rangle\geq 0,\\
\end{split}
\end{equation}
with equality only if 
\begin{enumerate}
    \item $\o=2$, and either 
    \begin{enumerate}
        \item $\v=\u_1+\u_2$, $l(\u_1)=l(\u_2)=2$, and $\langle\u_1,\u_2\rangle=2$, or
        
        \item $\v=2\u_1+\u_2$, $l(\u_1)=1$, $l(\u_2)=2$, and $\langle\u_1,\u_2\rangle=1$,
    \end{enumerate}
    as in \ref{enum:TSSRankTwoOrderTwoException}; or
    \item $\o$ is arbitrary, and  $\v=\u_1+b_2\u_2$, $l(\u_1)=1$, $l(\u_2)=\o$, and $\langle\u_1,\u_2\rangle=1$, as in \ref{enum:HilbertChow}.
\end{enumerate}
Furthermore, we have \begin{equation}\label{eqn:codim1}
-\left\lfloor\frac{b_1l(\u_1)}{\o}\right\rfloor-\left\lfloor\frac{b_2l(\u_2)}{\o}\right\rfloor+b_1 b_2 \langle \u_1,\u_2 \rangle=1
\end{equation}
only if \begin{enumerate}
    \item $\o=2$, and either 
    \begin{enumerate}
        \item $l(\u_1)=l(\u_2)=1=\langle\u_1,\u_2\rangle$ and $\v=\u_1+\u_2$, $\u_1+2\u_2$, or $2\u_1+\u_2$, as in \ref{enum:CodimOneRankOne}, or 
        \item $l(\u_1)=1=\langle\u_1,\u_2\rangle$, $l(\u_2)=2$, and $\v=2\u_1+2\u_2$, as in \ref{enum:CodimOneRankTwo}, or   $\v=3\u_1+\u_2$, $4\u_1+\u_2$, as in \ref{enum:CodimOneRankOne}, or
        \item $l(\u_1)=1$, $l(\u_2)=2=\langle\u_1,\u_2\rangle$, and $\v=\u_1+\u_2$, as in \ref{enum:CodimOneRankTwo}, or 
        \item $l(\u_1)=l(\u_2)=2=\langle\u_1,\u_2\rangle$ and $\v=\u_1+2\u_2$ or $2\u_1+\u_2$, as in \ref{enum:CodimOneRankTwo}, or 
        \item $l(\u_1)=l(\u_2)=2$, $\langle\u_1,\u_2\rangle=3$, and $\v=\u_1+\u_2$.  But this is impossible as then $6=3\o=\langle l(\u_1),l(\u_2)\rangle=4\langle\bar{\u}_1,\bar{\u}_2\rangle$; or 
    \end{enumerate}
    \item $\o=3$, and either
    \begin{enumerate}
        \item $l(\u_1)=l(\u_2)=1=\langle\u_1,\u_2\rangle$ and $\v=\u_1+\u_2$, as in \ref{enum:CodimOneRankOne}, or
        \item $l(\u_1)=l(\u_2)=3=\langle\u_1,\u_2\rangle$ and $\v=\u_1+\u_2$, as in \ref{enum:CodimOneRankThreeOrderTwoThreeException} or 
        \item $l(\u_1)=1=\langle\u_1,\u_2\rangle$, $l(\u_2)=3$ and $\v=2\u_1+\u_2$ or $3\u_1+\u_2$, as in \ref{enum:CodimOneRankOne}, or
        \item $l(\u_1)=1$, $l(\u_2)=3$, $\langle\u_1,\u_2\rangle=2$, and $\v=\u_1+\u_2$, as in \ref{enum:CodimOneRankTwo}; or
    \end{enumerate}
    \item $\o=4$, and either
    \begin{enumerate}
        \item $l(\u_1)=l(\u_2)=1=\langle\u_1,\u_2\rangle$ and $\v=\u_1+\u_2$, as in \ref{enum:CodimOneRankOne}, or 
        \item $l(\u_1)=1=\langle\u_1,\u_2\rangle$, $l(\u_2)=2$, and $\v=\u_1+\u_2$ or $\u_1+2\u_2$, as in \ref{enum:CodimOneRankOne}, or
        \item $l(\u_1)=l(\u_2)=2$, $\langle\u_1,\u_2\rangle=1$, and $\v=\u_1+\u_2$, $2\u_1+\u_2$, or $\u_1+2\u_2$, as in \ref{enum:CodimOneRankOne}, or 
        \item $l(\u_1)=1=\langle\u_1,\u_2\rangle$, $l(\u_2)=4$, and $\v=2\u_1+\u_2$, as in either \ref{enum:CodimOneRankOne} or \ref{enum:CodimOneRankTwo},  or
        \item $l(\u_1)=1$, $l(\u_2)=4$, $\langle\u_1,\u_2\rangle=2$, and $\v=\u_1+\u_2$, as in \ref{enum:CodimOneRankTwo}, or
        \item $l(\u_1)=2=\langle\u_1,\u_2\rangle$, $l(\u_2)=4$ and $\v=\u_1+\u_2$, as in \ref{enum:CodimOneRankTwo}; or
    \end{enumerate}
    \item $\o=6$, and either
    \begin{enumerate}
        \item $l(\u_1)=l(\u_2)=1=\langle\u_1,\u_2\rangle$ and $\v=\u_1+\u_2$, as in \ref{enum:CodimOneRankOne}, or 
        \item $l(\u_1)=1=\langle\u_1,\u_2\rangle$, $l(\u_2)=2$, and $\v=\u_1+\u_2$, as in \ref{enum:CodimOneRankOne}, or
        \item $l(\u_1)=l(\u_2)=2$, $\langle\u_1,\u_2\rangle=1$, and $\v=\u_1+\u_2$, as in \ref{enum:CodimOneRankOne}, or 
        \item $l(\u_1)=1=\langle\u_1,\u_2\rangle$, $l(\u_2)=3$, and $\v=\u_1+\u_2$ or $\u_1+2\u_2$, as in \ref{enum:CodimOneRankOne}, or 
        \item $l(\u_1)=2$, $l(\u_2)=3$, $\langle\u_1,\u_2\rangle=1$, and $\v=\u_1+\u_2$ or $\u_1+2\u_2$, as in \ref{enum:CodimOneRankOne}, or
        \item $l(\u_1)=1=\langle\u_1,\u_2\rangle$, $l(\u_2)=6$, and $\v=2\u_1+\u_2$, as in \ref{enum:CodimOneRankOne}, or 
        \item $l(\u_1)=1$, $l(\u_2)=6$, $\langle\u_1,\u_2\rangle=2$, and $\v=\u_1+\u_2$, as in \ref{enum:CodimOneRankTwo}, or
        \item $l(\u_1)=2=\langle\u_1,\u_2\rangle$, $l(\u_2)=6$, and $\v=\u_1+\u_2$, as in \ref{enum:CodimOneRankTwo}.
    \end{enumerate}
\end{enumerate}
In all other cases, $\codim\FF(\a_1,\dots,\a_n)^o\geq 2$.\footnote{We have spared the reader with the purely numerical verification of equality in \eqref{eq:l=2 case II} and \eqref{eqn:codim1}.  We used the ``Reduce'' operation in Mathematica to solve these equalities.}

Now we assume that $\a_1=b_1 \u_j$ and $\a_i^2>0$ for $i \geq 2$.
In this case, we also see that
\begin{equation}
\begin{split}
\codim\FF(\a_1,\dots,\a_n)^o\geq&\sum_i (\a_i^2-\dim \MM_{\sigma_-}(\a_i))+\sum_{i<j}\langle \a_i,\a_j \rangle\\
=&-\left\lfloor\frac{b_1l(\u_j)}{\o}\right\rfloor+\sum_{i>1}b_1\langle \u_j,\a_i\rangle+\sum_{1<i<k}\langle \a_i,\a_k\rangle\\
\geq&b_1(\langle \v,\u_j\rangle-1)+\sum_{1<i<k}\langle \a_i,\a_k\rangle\\
\geq&b_1(\langle \v,\u_j\rangle-1)\geq 0.
\end{split}
\end{equation}
Thus $\codim\FF(\a_1,\dots,\a_n)^o=0$ only if $\langle\v,\u_j\rangle=1$, $l(\u_j)=\o$, and $\v=b_1\u_j+\a_2$, as in \ref{enum:HilbertChow}.  Similarly, $\codim\FF(\a_1,\dots,\a_n)^o=1$ only if either $\v=\u_j+\a_2$ with $\langle\v,\u_j\rangle=2$ and $l(\u_j)=\o$, as in \ref{enum:CodimOneRankTwo}, $\v=\u_j+\a_2$ with $\langle\v,\u_j\rangle=1$ and $l(\u_j)<\o$, as in \ref{enum:CodimOneRankOne}, or $\v=2\u_j+\a_2$ with $\langle\v,\u_j\rangle=1$ and $\frac{\o}{l(\u_j)}=2$, as in \ref{enum:CodimOneRankOne}.  

Finally, if $\a_i^2>0$ for all $i$, then $\codim\FF(\a_1,\dots,\a_n)^o\geq 2$ by Proposition \ref{Prop:HN filtration all positive classes}.
\end{proof}

The cases in \cref{Prop:TotallySemistableCodimOne} involving an isotropic $\u\in\HH_W$ with $l(\u)=\o$ can be proven by using the Fourier-Mukai transform $\Xi=\Phi_\EE^{-1}$ associated to the universal family $\EE$ over $M_{\sigma_0}(\u)$, which satisfies $$\Xi:\Db(S)\cong\Db(S)$$ and $\Xi(\u)=(0,0,1)$ by \cref{lem:moduli spaces of fully induced isotropic vectors}.  By construction of $\Stabd(S)$, skyscraper sheaves of points on $S$ are $\Xi(\sigma_0)$-stable. By Bridgeland's Theorem \cite[Proposition 10.3]{Bri08}, there exist divisor classes $\omega,\beta\in\NS(S)_{\Q}$, with $\omega$ ample, such that up to the $\wGL2$-action, $\Xi(\sigma_0)=\sigma_{\omega,\beta}$. In particular, the category $\PP_{\omega,\beta}(1)$ is the extension-closure of skyscraper sheaves of points and the shifts, $F[1]$, of $\mu_{\omega}$-stable torsion-free sheaves $F$ with slope $\mu_{\omega}(F) =\omega\cdot\beta$. Since $\sigma_0$ by assumption does not lie on any other wall with respect to $\v$, the divisor $\omega$ is generic with respect to  $\Xi(\v)$. Under these identifications, we have the following result whose proof is identical to that of \cite[Theorem 3.2.7]{MYY14b}, \cite[Proposition 8.2]{BM14b}, and \cite[Section 5]{LQ14}.

\begin{Prop}\label{Prop:Uhlenbeck morphism}
Suppose that $\HH_W$ contains a primitive isotropic $\u$ such that $l(\u)=\o$.  Then, an object $E$ of class $\v$ is $\sigma_+$-stable if and only if $\Xi(E)$ is the shift $F[1]$ of a $\beta$-twisted $\omega$-stable sheaf $F$ on $S$; therefore $[-1]\circ\Xi$ induces the following identification of moduli spaces: $$M_{\sigma_+}(\v) = M_{\omega}^{\beta}(-\Xi(\v)).$$
Moreover, the contraction morphism $\Sigma^+$ induced by the wall $W$ is the Li-Gieseker-Uhlenbeck (LGU) morphism to the Uhlenbeck compactification.

Similarly, an object $F$ of class $\v$ is $\sigma_-$-stable if and only if it is the shift $F^\vee[1]$ of the derived dual of a $-\beta$-twsited $\omega$-stable sheaf on $S$.
\end{Prop}
It follows from the above description that a $\sigma_+$-stable object $E$ becomes $\sigma_0$-semistable if and only if $F=\Phi(E)[-1]$ is not locally free or if $F$ is strictly $\mu$-semistable, as these are the sheaves contracted by the Uhlenbeck contraction.  This translates our problem to the equivalent problem of determining the codimensions of the strictly $\mu$-semistable locus and the non-locally free locus.  One can then use the estimates in the proof of \cref{slope stability} to determine the codimensions of these loci.

\begin{Rem}
The reader may notice that in \cref{Prop:TotallySemistableCodimOne}, when $\langle\v,\u\rangle=2=l(\u)=\o$, we simultaneously claim that $W$ is totally semistable and that $\codim(M_{\sigma_+}(\v)\backslash M_{\sigma_0}^s(\v))=1$.  This occurs because in certain cases, $M_{\sigma_+}(\v)$ is reducible with one component $M_0$ consisting entirely of strictly $\sigma_0$-semistable objects and at least one more component $M_1$ which intersects $M_0$ along a divisor.  We can see this most easily by applying  $\Xi$ as in \cref{Prop:Uhlenbeck morphism} to reduce the problem to that of rank two sheaves on $S$ with $\o=2$ as in \cref{half-half}.
\end{Rem}

To initiate the converse of \cref{Prop:TotallySemistableCodimOne}, we take each case on its own, beginning with walls of \emph{Hilbert-Chow type}.
\begin{Lem}\label{Lem:HilbertChow}
Suppose that $\HH_W$ contains a primitive isotropic $\u$ such that $\langle\v,\u\rangle=1$ and $l(\u)=\o$.  Then $W$ is totally semistable, and the contraction $\Sigma^+$ is a divisorial contraction if $\v^2>2$.
\end{Lem}
\begin{proof}
By \cref{Prop:Uhlenbeck morphism}, it follows that $M_{\sigma_+}(\v)$ is isomorphic to the product of $\Pic^0(S)$ and the Hilbert scheme of points, $$M_{\sigma_+}(\v)\cong M_{\omega}^{\beta}(-\Xi(\v))\cong\Pic^0(S)\times\Hilb^{\v^2/2}(S),$$ where we may assume that $\beta=0$ up to tensoring by a line bundle.  From the description of $W$ and $\Sigma^+$ in \cref{Prop:Uhlenbeck morphism}, we furthermore see that $W$ is totally semistable as no ideal sheaf is locally free.  More explicitly, any ideal sheaf $I_Z$ fits into a short exact sequence in $\PP_{\omega,0}(1)$, $$0\to\OO_Z\to I_Z[1]\to\OO_S[1]\to 0,$$ which expresses $I_Z[1]$ as an extension of $\sigma_0$-semistable objects of the same phase.  Moreover, by taking filtrations of $\OO_Z$ and $\OO_{Z'}$, we see that $I_Z[1]$ and $I_{Z'}[1]$ are $S$-equivalent if and only if $Z$ and $Z'$ have the same support.  Thus, up to the factor $\Pic^0(S)$, the morphism $\Sigma$ is precisely the Hilbert-Chow morphism $\Hilb^{\v^2/2}(S)\to\Sym^{\v^2/2}(S)$ which contracts the divisor of non-reduced $0$-dimensional subschemes of length $\frac{\v^2}{2}>1$.  
\end{proof}
\begin{Rem}\label{Rem:ArithmeticOfHC}
It is worth noting that under the hypothesis of \cref{Lem:HilbertChow}, we may assume that $\u=\u_2$ and $\v=\u_1+\frac{\v^2}{2}\u_2$ with $l(\u_1)=1=\langle\u_1,\u_2\rangle$.  
Indeed, it follows from $\langle\v,\u_2\rangle=1$ that $\v$ and $\u_2$ generate $\HH_W$, and writing $\u_1=x\v+y\u_2$ with $x,y\in\Z$ gives $\u_1=\v-\frac{\v^2}{2}\u_2$ as $\u_1$ is primitive.  
Pairing with $\u_2$ gives that $\langle\u_1,\u_2\rangle=1$.  
Writing $\pi^*\u_i=l(\u_i)\bar{\u_i}$, we see that $$\o=\langle\pi^*\u_1,\pi^*\u_2\rangle=\o l(\u_1)\langle\bar{\u}_1,\bar{\u}_2\rangle$$ forces $l(\u_1)=1$, as claimed.
\end{Rem}

Now we study the converse of the other case of a possible totally semistable wall.
\begin{Lem}\label{Lem:TSSRankTwoOrderTwoException}
Suppose that $\v^2=4$, $l(\v)=2$, and $\HH_W$ contains a primitive isotropic $\u$ such that $\langle\v,\u\rangle=2=l(\u)=\o$.  Then $W$ is totally semistable for at least one component of $M_{\sigma_+}(\v)$ and induces a $\P^1$-fibration.  For a different component of $M_{\sigma_+}(\v)$, $W$ is not totally semistable. 
\end{Lem}
\begin{proof}
Observe that $\v^2=4$ and $\langle\v,\u\rangle=2$ imply that $\v-\u$ is isotropic.  Moreover, $\langle\u,\v-\u\rangle=2$, so we may assume that $\u=\u_2$ and either $\v-\u_2=\u_1$ with $\langle\u_1,\u_2\rangle=2$ or $\v-\u_2=2\u_1$ with $\langle\u_1,\u_2\rangle=1$.  

Suppose first that $\v=\u_1+\u_2$ with $\langle \u_1,\u_2\rangle=2$.  Then the condition $l(\v)=2$ implies $l(\u_1)=2$.  
Taking $E_i\in M_{\sigma_0}^s(\u_i)$, we have $$\Hom(E_2,E_1)=\Hom(E_1,E_2(K_S))=0$$ by stability, so from $\langle\u_1,\u_2\rangle=2$, we see that $\ext^1(E_2,E_1)=2$.  
Thus for each $(E_1,E_2)\in M_{\sigma_0}^s(\u_1)\times M_{\sigma_0}^s(\u_2)$, there is a $\P^1$ worth of $S$-equivalent extensions $$0\to E_1\to E\to E_2\to 0,$$ which are $\sigma_+$-stable by \cite[Lemma 6.9]{BM14a}, assuming without loss of generality that $\mu_{\sigma_+}(\u_1)<\mu_{\sigma_+}(\u_2)$.  Varying the $E_i\in M_{\sigma_0}^s(\u_i)$ sweeps out a component of $M_{\sigma_+}(\v)$ which admits a $\P^1$-fibration over $M_{\sigma_0}^s(\u_1)\times M_{\sigma_0}^s(\u_2)$.

To see the existence of a second component of $M_{\sigma_+}(\v)$ for which $W$ is not totally semistable, we apply $[-1]\circ\Xi$ from \cref{Prop:Uhlenbeck morphism}.  Then we can assume that $\v=(r,D,s)$, $\u_2=-(0,0,1)$, $\u_1=(r_1,D_1,s_1)$, and that $\sigma_+$-stability agrees with $\beta$-twisted $\omega$-stability while $W$ induces the LGU morphism.  The equality $\langle\v,\u_2\rangle=2$ implies that $r=2$.

As $\v=\u_1+\u_2$ with $\langle \u_1,\u_2\rangle=2$, we get $r_1=2$, and $D=D_1$.  
As $0=\u_1^2=2ab-4s_1$, we must have $a$ or $b$ even, so tensoring by a line bundle, we may assume that $D=0$, $A_0$ or $B_0$.
From $l(\v)=2$ and $\v-\u_2=\u_1$ is primitive, we must have $\v=(2,B_0,-1)$ and $S$ is of Type 1.
The existence of a second component that is not totally semistable follows from \cref{half-half}.  

If instead $\v=2\u_1+\u_2$ with $\langle\u_1,\u_2\rangle=1$, then we note that $l(\u_1)=1$.  Indeed, writing $\pi^*\u_i=l(\u_i)\bar{\u_i}$, we have  $$2=\langle\pi^*\u_1,\pi^*\u_2\rangle=2l(\u_1)\langle\bar{\u_1},\bar{\u_2}\rangle,$$ which forces $l(\u_1)=1$.  Taking $E_1\in M_{\sigma_0}^s(2\u_1)$, which is nonempty and two-dimensional by \cref{isotropic}, and $E_2\in M_{\sigma_0}^s(\u_2)$, we have $$\Hom(E_2,E_1)=\Hom(E_1,E_2(K_S))=0$$ by stability, so from $\langle2\u_1,\u_2\rangle=2$, we see that $\ext^1(E_2,E_1)=2$.  Thus for each $(E_1,E_2)\in M_{\sigma_0}^s(2\u_1)\times M_{\sigma_0}^s(\u_2)$, there is a $\P^1$ worth of $S$-equivalent extensions $$0\to E_1\to E\to E_2\to 0,$$ which are $\sigma_+$-stable by \cite[Lemma 6.9]{BM14a}, assuming without loss of generality that $\mu_{\sigma_+}(\u_1)<\mu_{\sigma_+}(\u_2)$.\footnote{Otherwise, just switch the direction of the extension.}  Varying the $E_i$ in their respective moduli spaces sweeps out a component of $M_{\sigma_+}(\v)$ which admits a $\P^1$-fibration over $M_{\sigma_0}^s(2\u_1)\times M_{\sigma_0}^s(\u_2)$, as claimed.  

To see the existence of a second component that is not totally semistable, we argue similarly as in the first case.  Then after applying $[-1]\circ\Xi$ and possibly tensoring by a line bundle, we may assume that $\v=(2,0,-1)$, $\u_2=-(0,0,1)$, and $\u_1=(1,0,0)$.  Then the claim follows from \cref{half-half}.  
\end{proof}
We move on to the first case of a divisorial contraction.
\begin{Lem}\label{Lem:CodimOneRankOne}
Suppose that $\HH_W$ contains a primitive isotropic $\u$ with $l(\u)<\o$ such that $\langle\v,\u\rangle=1$.  The wall $W$ is not totally semistable and $D=M_{\sigma_+}(\v)\backslash M_{\sigma_0}^s(\v)$ is a divisor, except when $\HH_W$ contains another primitive isotropic $\u'$ with $l(\u')=\o$ and $\langle\v,\u'\rangle=1$, in which case $W$ is a totally semistable fake wall.  Moreover, the morphism $\Sigma^+$ only contracts $D$ if 
\begin{enumerate}
    \item $\o=2$ and $\HH_W$ also contains another primitive isotropic $\u'$ with $l(\u')=\o$ and $\langle\v,\u'\rangle=3$, or
    \item $\o>2$ is even and $\HH_W$ also contains another primitive isotropic $\u'$ with $l(\u')=\o$ and $\langle\v,\u'\rangle=2$.
\end{enumerate} 
\end{Lem}
\begin{proof}
Observe that $(\v-\u)^2=\v^2-2\langle\v,\u\rangle=\v^2-2\geq 0$, and suppose first that $(\v-\u)^2>0$.  Then $M_{\sigma_0}^s(\v-\u)\neq\varnothing$ unless $\u=\u_1$, $l(\u_2)=\o$, $\langle\u_1,\u_2\rangle=1$, and either $\v=2\u_1+\u_2$ or $\o=2$ and $\v=3\u_1+\u_2$.  Setting these two cases aside for the moment, we see that $W$ is not totally semistable by part \ref{enum:TSS} of \cref{Prop:TotallySemistableCodimOne}, and furthermore, for $G\in M_{\sigma_0}^s(\u)$ and $F\in M_{\sigma_0}^s(\v-\u)$, we must have $\ext^1(F,G)=1$ by stability.  Thus there exists a unique extension $$0\to G\to E\to F\to 0,$$ which is $\sigma_+$-stable by \cite[Lemma 6.9]{BM14a}, and a dimension count shows that these sweep out a non-contracted divisor as we vary $F,G$ in their respective moduli spaces.

Returning to the two excluded cases, the usual pull-back argument shows that $\langle\u_1,\u_2\rangle=1$ and $l(\u_2)=\o$ force $l(\u_1)=1$.  We begin with the more typical of these two cases, namely $\o$ arbitrary and $\v=2\u_1+\u_2$.  If $2\mid\o$, then taking $E_1\in M_{\sigma_0}^s(2\u_1)$ and $F\in M_{\sigma_0}^s(\u_2)$ gives a $\P^1$ worth of $\sigma_+$-stable extensions $$0\to E_1\to E\to F\to 0$$ by \cite[Lemma 6.1-6.3]{CH15}.  These $\P^1$'s are all contracted by $\Sigma^+$, and varying $(E_1,F)\in M_{\sigma_0}^s(2\u_1)\times M_{\sigma_0}^s(\u_2)$ sweeps out a divisor if $2<\o$ since then $\dim M_{\sigma_0}^s(2\u_1)=1$.  If $\o=3$, then $M_{\sigma_0}^s(2\u_1)=\varnothing$, so taking $G_1\ncong G_2\in M_{\sigma_0}^s(\u_1)$ and $F\in M_{\sigma_0}^s(\u_2)$, there is a unique $\sigma_+$-stable extension
$$0\to G_1\oplus G_2\to E\to F\to 0$$
 by \cite[Lemma 6.1-6.3]{CH15}.  Varying $(G_1,G_2,F)\in M_{\sigma_0}^s(\u_1)^2\times M_{\sigma_0}^s(\u_2)$ sweeps out a non-contracted divisor.
 
 When $\o=2$ and $\v=2\u_1+\u_2$, we are back in the case discussed in \cref{Lem:TSSRankTwoOrderTwoException}, when the strictly semistable locus was a non-contracted divisor in at least one component.  

When $\o=2$ and $\v=3\u_1+\u_2$ (i.e. $\u'=\u_2$), we take $G_1\in M_{\sigma_0}^s(\u_1)$, $G_2\in M_{\sigma_0}^s(2\u_1)$ and $F\in M_{\sigma_0}(\u_2)$, which are nonempty by \cref{isotropic}, and we consider extensions of the form $$0\to G_1\oplus G_2\to E\to F\to 0.$$ By \cite[Lemma 6.1-6.3]{CH15}, for fixed $G_1$,$G_2$, and $F$ there is a $\P^1$ worth of $S$-equivalent distict $\sigma_+$-stable such extensions, so by a dimension count we see that varying $(G_1,G_2,F)\in M_{\sigma_0}^s(\u_1)\times M_{\sigma_0}^s(2\u_1)\times M_{\sigma_0}(\u_2) $ gives a contracted divisor.

Now suppose that $(\v-\u)^2=0$.  Then $\langle\u,\v-\u\rangle=1$ so that we can assume that $\u=\u_1$ and $\v=\u_1+\u_2$ with $\langle\u_1,\u_2\rangle=1$.  This last condition forces $\HH_W=\Z\u_1+\Z\u_2$ so that $\v$ admits no other decomposition into effective classes.  Thus the strictly $\sigma_0$-semistable locus consists of extensions of $G_i\in M_{\sigma_0}^s(\u_i)$.  By stability, $\ext^1(G_2,G_1)=1$, so $\Sigma^+$ does not contract any curves of $S$-equivalent objects in this case.  Varying $(G_1,G_2)\in M_{\sigma_0}^s(\u_1)\times M_{\sigma_0}^s(\u_2)$, a dimension count shows that these extensions sweep out a divisor or cover a component of $M_{\sigma_+}(\v)$ if $l(\u_2)<\o$ or $l(\u_2)=\o$, respectively, as required.
\end{proof}
\begin{Lem}\label{Lem:LGUDivisorial}
Suppose that $\v^2\geq 4$ and $\HH_W$ contains a primitive isotropic $\u$ with $l(\u)=\o$ such that $\langle\v,\u\rangle=2$.  Then (on at least one component) $W$ is not totally semistable, and moreover $\Sigma^+$ is a divisorial contraction unless $\v^2=4$ and $\o=2,3$.   
\end{Lem}
\begin{proof}
Without loss of generality, we may assume that $\u=\u_2$.  We will show that $M_{\sigma_0}^s(\v-\u_2)\neq\varnothing$.  Then given $F\in M_{\sigma_0}^s(\v-\u_2)$ and $E_2\in M_{\sigma_0}^s(\u_2)$, we get a $\P^1$ worth of extensions $$0\to F\to E\to E_2\to 0,$$ which are $S$-equivalent with respect to $\sigma_0$ but $\sigma_+$-stable by \cite[Lemma 6.9]{BM14a}.  A dimension count shows that varying $F$ and $E_2$ in their moduli spaces sweeps out a divisor in $M_{\sigma_+}(\v)$ that gets contracted by $\Sigma^+$, as long as $\dim M_{\sigma_0}^s(\v-\u_2)=(\v-\u_2)^2+1$.  

Let us suppose first that $\v^2\ne 4,8$.  Then $(\v-\u_2)^2>0$ and different from $4$, $\langle\v,\u_2\rangle=2$, and $l(\u_2)=\o$, so it follows from part \ref{enum:TSS} of \cref{Prop:TotallySemistableCodimOne}  that $M_{\sigma_0}^s(\v-\u_2)\neq\varnothing$ and has dimension $(\v-\u_2)^2+1$, as required.  

If instead now $(\v-\u_2)^2=0$ (i.e. $\v^2=4$), then $\langle\v-\u_2,\u_2\rangle=2$ forces  $\v=\u_1+\u_2$ and $\langle\u_1,\u_2\rangle=2$ or $\v=2\u_1+\u_2$ and $\langle\u_1,\u_2\rangle=1$.  
Writing $\pi^*\u_i=l(\u_i)\bar{\u}_i$, we see that $$2\o=\o\langle\u_1,\u_2\rangle=\langle\pi^*\u_1,\pi^*\u_2\rangle=\o l(\u_1)\langle\bar{\u}_1,\bar{\u}_2\rangle,$$ so $l(\u_1)=1$ or 2 in the first case or $l(\u_1)=1$ in the second case.  
In the first case, if $l(\u_1)<\o$ (which is guaranteed if $\o>2$), then $\dim M_{\sigma_0}^s(\u_1)=1=\u_1^2+1$.  Thus the argument in the previous paragraph goes through.  Otherwise, $l(\u_1)=\o=2$. 
But then $l(\v)=2$, so we are in the case described in \cref{Lem:TSSRankTwoOrderTwoException}.  But then we already know the existence of a component whose strictly semistable locus is a non-contracted divisor.  

In the second case, we have $\v=2\u_1+\u_2$, with $\langle\u_1,\u_2\rangle=1$, and $l(\u_1)=1$.  If $\o=2$, then $l(\v)=2$, and we are again in the case described in \cref{Lem:TSSRankTwoOrderTwoException}.  But then we already know the existence of a component whose strictly semistable locus is a non-contracted divisor.  So we may suppose that $\o>2$.  Then $W$ is not totally semistable by \cref{Prop:TotallySemistableCodimOne} and $M_{\sigma_0}^s(2\u_1)\ne\varnothing$ of dimension one unless $\o=3$.  But then again the argument above produces a contracted divisor.  If $\o=3$, then $M_{\sigma_0}^s(2\u_1)=\varnothing$, so taking $G_1\ncong G_2\in M_{\sigma_0}^s(\u_1)$ and $F\in M_{\sigma_0}^s(\u_2)$, there is a unique $\sigma_+$-stable extension $$0\to G_1\oplus G_2\to E\to F\to 0$$ by \cite[Lemma 6.1-6.3]{CH15}.  Varying $(G_1,G_2,F)\in M_{\sigma_0}^s(\u_1)^2\times M_{\sigma_0}^s(\u_2)$ sweeps out a non-contracted divisor.

Finally, suppose that $(\v-\u_2)^2=4$.  We know that $W$ is not totally semistable by \cref{Prop:TotallySemistableCodimOne}. Then by the previous case of the lemma, we know that for at least one component $M$ of $M_{\sigma_+}(\v-\u_2)$, $W$ is not totally semistable, so $M_{\sigma_0}^s(\v-\u_2)\ne\varnothing$ and has the right dimension, so the argument above produces a contracted divisor.  
\end{proof}
\begin{Lem}\label{Lem:rank3counterexample}
Suppose that $\v^2=6$ and $\HH_W$ contains a primitive isotropic $\u$ such that $\langle\v,\u\rangle=3=\o=l(\u)$.  Then $\Sigma^+$ induces a divisorial contraction.
\end{Lem}
\begin{proof}
We know that $W$ is not totally semistable by \cref{Prop:TotallySemistableCodimOne}.  We may assume that $\u=\u_2$.  Then $(\v-\u_2)^2=\v^2-6=0$.  Thus $\v=b_1\u_1+\u_2$.  Since $3=\langle\v,\u_2\rangle=b_1\langle\u_1,\u_2\rangle$, we have either $\v=3\u_1+\u_2$ and $\langle\u_1,\u_2\rangle=1$ or $\v=\u_1+\u_2$ and $\langle\u_1,\u_2\rangle=3$.  In particular, it follows that $l(\u_1)=1$ or $l(\u_1)=3$, respectively. 
As $\o=3$, we get that $M_{\sigma_0}^s(\v-\u_2)\ne\varnothing$ and has dimension two, in either case, and the same is true for $M_{\sigma_0}^s(\u_2)$.  There is a $\P^2$-worth of extensions 
$$0\to E_1\to E\to E_2\to 0,$$
for $E_1\in M_{\sigma_0}^s(\v-\u_2)$ and $E_2\in M_{\sigma_0}^s(\u_2)$, and these are all $S$-equivalent with respect to $\sigma_0$ but $\sigma_+$-stable by \cite[Lemma 6.9]{BM14a}.  A dimension count shows that varying $E_1$ and $E_2$ in their moduli spaces sweeps out a divisor in $M_{\sigma_+}(\v)$ that gets contracted by $\Sigma^+$.  
\end{proof}
\section{Flopping walls}\label{Sec:FloppingWalls}
In \cref{Sec:IsotropicWalls}, we gave necessary and sufficient criteria for the wall $W$ to be totally semistable, to induce a $\P^1$ fibration, and to induce a divisorial contraction.  In this section, we discuss the remaining possibility for the contraction morphism $\Sigma^+$.  That is, if $W$ does not induce a fibration or a divisorial contraction, then it must either induce a small contraction, that is, the exceptional locus of $\Sigma^+$ must have codimension at least two, or $W$ is fake wall so that $\Sigma^+$ does not contract any curves.  In the next result, we give precise criteria for when $W$ is a genuine wall inducing a small contraction, at least for $\v$ primitive.  It is the only result in our work so far that has assumed that $\v$ is primitive.
\begin{Prop} \label{prop:flops}
Assume that $\v$ is primitive and that $W$ induces neither a divisorial contraction nor a $\P^1$-fibration.  If $\v^2\geq 4$ and $\v$ can be written as a sum $\v = \a_1 +\a_2$ with $\a_i\in P_\HH$, 
then $W$ induces a small contraction on $M_{\sigma_+}(\v)$.
%, except when 
%\begin{equation}\label{eqn:o=3 exception}\v^2=4,\o=3,\a_i^2=0,l(\a_i)=1,\langle\v,\a_2\rangle=2,\mbox{ and }2\mid\a_1.
%\end{equation}
\end{Prop}

\begin{proof}
Note that it suffices to show that some positive dimensional subvariety of $\sigma_+$-stable objects becomes $S$-equivalent with respect to $\sigma_0$ and thus gets contracted by $\Sigma^+$. 

Write $\v=\a_1+\a_2$ with $\a_i\in P_{\HH}$.  Using \cite[Lemma 9.2]{BM14b}, we may assume that the parallelogram with vertices $0,\a_1,\v,\a_2$ does not contain any lattice point other than its vertices.  In particular, the $\a_i$ are primitive, and without loss of generality, we may assume that $\phi^+(\a_1)<\phi^+(\a_2)$.  By \cref{Bridgeland non-empty}, there exist $\sigma_+$-stable objects $A_i$ with $\v(A_i)=\a_i$.  If $\a_i^2>0$ for each $i$, then the signature of $\HH$ forces $\langle \a_1,\a_2\rangle> 2$ so that $\ext^1(A_2,A_1)\geq 3$.  

Now suppose that, say, $\a_1^2=0$.  We consider first the case that $\langle\v,\a_1\rangle=1$.  If $l(\a_1)=\o$, then by \cref{Lem:HilbertChow} and the assumption $\v^2\geq 4$, we would have a divisorial contraction, contrary to assumption.  Thus we may assume that $l(\a_1)<\o$ and let $n=\frac{\o}{l(\a_1)}\geq 2$.  If $n\leq\frac{\v^2}{2}$, then $\Sigma_+$ contracts a $\P^{n-1}$ worth of nontrivial $\sigma_+$-stable extensions
$$0\to E_1\to E\to E_2\to 0,$$
with $E_1\in M_{\sigma_0}^s(n\a_1)$ and $E_2\in M_{\sigma_0}^s(\v-n\a_1)$.  Varying $E_1$ and $E_2$ in their moduli sweeps out a locus of dimension 
\begin{align*}
&2+n-1+\begin{cases}
2 & \mbox{ if }n=\frac{\v^2}{2},l(\v-n\a_1)=\o\\
(\v-n\a_1)^2+1 & o/w
\end{cases}\\&=\dim M_{\sigma_+}(\v)-\begin{cases}
n-2 & \mbox{ if }n=\frac{\v^2}{2},l(\v-n\a_1)=\o\\
n-1 & o/w
\end{cases}.
\end{align*}
In the first case, $\langle\v,\a_1\rangle=1$ and $l(\v-n\a_1)=\o$ force $l(\a_1)=1$ from the usual pull-back argument.
Then $n=\o$.  If $n=2$ then we get a $\P^1$-fibration, while if $n=3$, then we get a divisorial contraction, both of which are contrary to assumption.  
So we may assume that $n\geq 4$, which gives a contracted locus of codimension at least two.  
In the second case, if $n=2$, then we get a divisor covered by contracted $\P^1$'s, i.e. a divisorial contraction, contrary to assumption.  
So we may assume that $n\geq 3$, again giving a contracted locus of codimension at least two.  

If, instead, $\frac{\o}{l(\a_1)}=n>\frac{\v^2}{2}\geq 2$, then $l(\a_1)=1$.  Suppose first that $\o>3$.  Then since $\o$ is divisible by $2$, we have $M_{\sigma_0}^s(2\a_1)\ne\varnothing$ and has dimension one.  As $\v^2\geq 4$, we get $(\v-2\a_1)^2\geq 0$, so $\Sigma_+$ contracts a $\P^1$ worth of nontrivial $\sigma_+$-stable extensions
$$0\to E_1\to E\to E_2\to 0,$$
with $E_1\in M_{\sigma_0}^s(2\a_1)$ and $E_2\in M_{\sigma_0}^s(\v-2\a_1)$.  Varying $E_1$ and $E_2$ in their moduli sweeps out a locus of dimension 
$$1+(\v-2\a_1)^2+1+1=\dim M_{\sigma_+}(\v)-2,$$
unless $l(\v-2\a_1)=\o$, in which case we get a divisorial contraction, contrary to assumption.  So we get a contracted locus of codimension at least two, as required.
If now $\o=3$, then we must have $\v^2=4$.  Then $(\v-2\a_1)^2=0$, so $\v=2\u_1+\u_2$ with $\langle\u_1\u_2\rangle=1=l(\u_1)$.
As $M_{\sigma_0}^s(2\u_1)=\varnothing$ since $\o=3$, we get the JH-filtration of a strictly $\sigma_0$-semistable $E$ object corresponds to the decomposition 
$$\v=\u_1+\u_1+\u_2.$$
Taking $G_1\ncong G_2\in M_{\sigma_0}^s(\u_1)$ and $F\in M_{\sigma_0}^s(\u_2)$, there is a unique $\sigma_+$-stable extension $$0\to G_1\oplus G_2\to E\to F\to 0$$ by \cite[Lemma 6.1-6.3]{CH15}.  Varying $(G_1,G_2,F)\in M_{\sigma_0}^s(\u_1)^2\times M_{\sigma_0}^s(\u_2)$ sweeps out a non-contracted locus.  Taking $G$ to be the unique non-trivial self-extension of $G_1\in M_{\sigma_0}^s(\u_1)$, we get a $\P^1$-worth of $\sigma_+$-stable extensions
$$0\to G\to E\to F\to 0$$
that gets contracted by $\Sigma^+$.

Otherwise, we may assume that $\langle\v,\a_1\rangle\geq 2$.  
By \cite[Lemma 9.3]{BM14b}, any nontrivial extension 
$$0\to A_1\to E\to A_2\to 0$$
with $A_1\in M_{\sigma_0}^s(\a_1)$ and $A_2\in M_{\sigma_0}^s(\a_2)$ is $\sigma_+$-stable of class $\v$.  All such extensions are non-isomorphic and parametrized by a $\P^{\langle\v,\a_1\rangle-1}$ which is contracted by $\Sigma^+$. 
If $l(\a_1)<\o$, then, varying $A_1$ and $A_2$ in their respective moduli spans a locus of dimension
\begin{align*}
&1+\langle\v,\a_1\rangle-1+\begin{cases}
(\v-\a_1)^2+2 & \mbox{ if }(\v-\a_1)^2=0,l(\v-\a_1)=\o\\
(\v-\a_1)^2+1 & o/w
\end{cases}\\&=\dim M_{\sigma_+}(\v)-\begin{cases}
\langle\v,\a_1\rangle-1 & \mbox{ if }(\v-\a_1)^2=0,l(\v-\a_1)=\o\\
\langle\v,\a_1\rangle & o/w
\end{cases}.
\end{align*}
The result then follows except if $\langle\v,\a_1\rangle=2$, $(\v-\a_1)^2=0$ and $l(\v-\a_1)=\o$.  
In that case, $\v^2\geq 4$ implies that $\v^2=4$, in which case $W$ induces a divisorial contraction by \cref{Lem:LGUDivisorial}, contrary to assumption.  Similarly, if $l(\a_1)=\o$, then we would get a divisorial contraction by \cref{Lem:LGUDivisorial}, contrary to assumption.  This proves the proposition.
\end{proof}
%\todo{The only thing we would need to do in the above proof to allow $v$ to be non-primitive is to show that the extensions constructed are not $S$-equivalent with respect to $\sigma_+$.  A small modification to Lemmas 9.2 and 9.3 in \cite{BM14b} (which I've included at the end of the source file) gives $\sigma_+$-semistable extensions.  Now, we can assume that $W$ is a fake wall for all $m'v_0$ where $m'<m$ and $v=mv_0$.  Indeed, any curve contracted by $\Sigma^+$ for $m'v_0$ can be easily embedded into $M_{\sigma_+}(v)$ to give another curve contracted by $\Sigma^+$.  I believe that under this assumption, we can show that the above extensions are distinct with respect to $S$-equivalence.  But let's move on for the moment.}

Now we prove the converse to Proposition \ref{prop:flops}, namely that if $\HH_W$ does not fall into any of the above mentioned cases, then $W$ is not a genuine wall.
\begin{Prop}\label{prop: fake or non-walls}
Assume that $\v$ is primitive and that $W$ induces neither a divisorial contraction nor a $\P^1$-fibration.  Assume further that the hypothesis of \cref{prop:flops} is not satisfied.  Then $W$ is either a fake wall, or not a wall at all.
\end{Prop}
\begin{proof}
We assume for now that $\v^2\geq4$.  We prove that in this case every $\sigma_+$-stable object $E$ of class $\v$ is $\sigma_0$-stable.  If not, then some such $E$ is strictly $\sigma_0$-semistable, and thus $\sigma_-$-unstable.  Let $\a_1,\ldots,\a_n$ be the Mukai vectors of the HN-filtration factors of $E$ with respect to $\sigma_-$.  But all of these $\a_i\in P_{\HH}$, contradicting the assumption that $\v$ cannot be written as a sum of two positive classes.  So $W$ is not a wall at all in this case.

Now let us suppose that $\v^2=2$ .  We will show that though there may be some strictly $\sigma_0$-semistable object $E$, there are no curves of such objects that are $\sigma_+$-stable.  Take a strictly $\sigma_0$-semistable object $E$ and consider its Jordan-H\"{o}lder filtration with respect to $\sigma_0$ with $\sigma_0$-stable factors $E_i$ with $\v(E_i)=\a_i$.  Then we may write $\v=\sum_{i=1}^n\a_i$ with $\langle\v,\a_i\rangle\geq 1$ for all $i$, so we may order the $\a_i$ such that $$1\leq\langle\v,\a_1\rangle\leq\langle\v,\a_2\rangle\leq\dots\leq\langle\v,\a_n\rangle.$$

Observe that for $E$ to be strictly $\sigma_0$-semistable we must have $n\geq 2$, so  
$$2=\v^2=\sum_{i=1}^n\langle\v,\a_i\rangle\geq n\langle\v,\a_1\rangle\geq n\geq 2,$$  
so $n=2$ and $1=\langle\v,\a_1\rangle=\langle\v,\a_2\rangle$.  
It follows that 
$$1=\langle\v,\a_1\rangle=\a_1^2+\langle\a_1,\a_2\rangle=\a_1^2+1,$$ so $\a_1^2=\a_2^2=0$.
Since this is the only decomposition of $\v$ into positive classes, we see that the only strictly $\sigma_0$-semistable objects are the unique nontrivial extensions 
$$0\to A_1\to E\to A_2\to 0$$
with $A_1\in M_{\sigma_0}^s(\a_1)$ and $A_2\in M_{\sigma_0}^s(\a_2)$.  Thus $W$ is a fake wall in this case.
\end{proof}
\section{Proofs of the main theorems}
We have established all of the pieces needed to prove \cref{classification of walls}:
\begin{proof}[Proof of \cref{classification of walls}]
The theorem follows from \cref{prop: fake or non-walls,prop:flops,Prop:TotallySemistableCodimOne,Prop:TotallySemistableCodimOne,Lem:HilbertChow,Lem:TSSRankTwoOrderTwoException,Lem:CodimOneRankOne,Lem:LGUDivisorial,Lem:rank3counterexample}.
\end{proof}
With this wall-crossing classification in hand, we can prove \cref{Thm:MainTheorem1}.
\begin{proof}[Proof of \cref{Thm:MainTheorem1}]  
Connect $\sigma$ and $\tau$ by a path, $\sigma(t)$, with $0\leq t\leq 1$, $\sigma(0)=\sigma$, and $\sigma(1)=\tau$.  
Observe that as the set of walls is locally finite, $\sigma(t)$ only crosses finitely many of them, and by perturbing $\sigma(t)$ slightly, we may assume that $\sigma(t)$ only crosses one wall at a time (that is, if $\sigma(t_0)\in W$ then $\sigma(t_0)$ is a generic point of $W$).  
If $\sigma(t)$ does not cross any totally semistable walls, then $M_\sigma(\v)$ and $M_\tau(\v)$ are clearly birational, and in the second part of the claim, we may take any common open subset $U$ and $\Phi=\Id$ for the (anti-)autoequivalence.  
Otherwise, it suffices consider that $\sigma=\sigma_+$ and $\tau=\sigma_-$ are two sufficiently close stability conditions separated by a single totally semistable wall $W$.  
By \cref{classification of walls}, we may assume that there exists an isotropic $\u\in\HH_W$ such that $l(\u)=\o$ and either $\langle\v,\u\rangle=1$ or  $\langle\v,\u\rangle=2=\o=l(\v)$ and $\v^2=4$.  
We will show that there is an autoequivalence $\Phi$ of $\Db(S)$ inducing an isomorphism $\Phi\colon M_{\sigma_+}(\v)\to M_{\sigma_-}(\v)$.

In the first case we assume that $\HH_W$ contains an isotropic vector $\u$ such that $\langle \v,\u\rangle=1$ and $l(\u)=\o$.  Then by \cref{Prop:Uhlenbeck morphism} and \cref{Lem:HilbertChow} we may, up to the shift by 1, identify $M_{\sigma_+}(\v)$ with the $(\beta,\omega)$-Gieseker moduli space $M^\beta_{\omega}(-\v)$ of stable sheaves of rank one.  Thus $M_{\sigma_+}(\v)$ parametrizes the shifts $I_Z(L)[1]$ of the twists of ideal sheaves of 0-dimensional subschemes $Z\in\Hilb^n(X)$, where $n=\frac{\v^2+1}{2}$, by $L\in\Pic^0(S)$.  Moreover, $\sigma_-$-stable objects are precisely $I_Z(L)^{\vee}[1]$.  But then $\Phi(\blank):=(\blank)^{\vee}[2]$ is the required an autoequivalence.

In the second case, we again use \cref{Prop:Uhlenbeck morphism}.  Then shifting by one identifies $M_{\sigma_+}(\v)$ with the moduli space $M_{\omega}(2,c,s)$ of $\omega$-Gieseker stable sheaves $F$ with $\v(F)=(2,c,s)=-\v$.  Choosing $L\in\Pic(X)$ with $c_1(L)=c$, we get that $\Phi(\blank):=(\blank)^{\vee}\otimes \OO(L)[2]$ is the required autoequivalence, as $\Phi(F[1])=F^{\vee}\otimes \OO(L)[1]$ is an object of $M_{\sigma_-}(\v)$, for any $F\in M_{\omega}(2,c,s)$.
\end{proof}
Putting things together we can finally prove \cref{MainThm4:FM reduction}:
\begin{proof}[Proof of \cref{MainThm4:FM reduction}]
The theorem follows directly from \cref{FM reduction} and \cref{Thm:MainTheorem1}.
\end{proof}
\part*{Appendix: Singularities of moduli spaces of small dimension}

In this appendix we analyze the singularities of small dimensional $\mSs$.  To do so, we need to consider each case separately according to $\o$.  But first recall that by Theorem \ref{global singularities}, $$\v(F_e)^2+2=\dim\Sing(\mSs)\leq\frac{1}{\o}(\v^2+2\o)=\frac{1}{\o}\pi^*\v^2+2,$$ so $\v(F_e)^2\leq\frac{\v^2}{\o}$.  Furthermore, we observe that $$c_1(F_e).c_1(F_{g^{j-i}})=c_1(F_{g^i}).c_1(F_{g^j}),\mbox{ and } c_1(F_e).c_1(F_{g^i})=c_1(F_e).c_1(F_{g^{\o-i}}).$$  Indeed, for the first equality, $$c_1(F_{g^i}).c_1(F_{g^j})=(g^i)^*c_1(F_e).(g^i)^*c_1(F_{g^{j-i}})=c_1(F_e).c_1(F_{g^{j-i}}),$$ and for the second $$c_1(F_e).c_1(F_{g^i})=(g^{\o-i})^*c_1(F_e).(g^{\o-i})^*c_1(F_{g^i})=c_1(F_e).c_1(F_{g^{\o-i}}).$$  Thus if $\v(E)=(r,c_1(E),s)$, then $\v(F_{g^i})=(\frac{r}{\o},c_1(F_{g^i}),s)$ and, \begin{equation}\label{term 1}\o\v^2=(\pi^*\v)^2=\left(\sum_{i=0}^{\o-1}\v(F_{g^i})\right)^2=\o \v(F_e)^2+2\sum_{i<j}\langle \v(F_{g^i}),\v(F_{g^j})\rangle.\end{equation}  If $\o$ is even, then \begin{equation}\label{term 2}\sum_{i<j}\langle \v(F_{g^i}),\v(F_{g^j})\rangle=\o\sum_{i=1}^{\o/2-1}\langle \v(F_e),\v(F_{g^i})\rangle+\frac{\o}{2}\langle \v(F_e),\v(F_{g^{\o/2}})\rangle,\end{equation} while if $\o=3$, then \begin{equation}\label{term 3}\sum_{i<j}\langle \v(F_{g^i}),\v(F_{g^j})\rangle=\o\langle \v(F_e),\v(F_g)\rangle.\end{equation}

Now we break the proof into cases according to $\o$.  
\section*{Order Two}
It remains to consider $\v^2=4,2$, and we note that by \eqref{term 1} and \eqref{term 2} $$\ext^1(F_e,F_g)=\langle \v(F_e),\v(F_g)\rangle=\v^2-\v(F_e)^2.$$  First consider $\v^2=4$.  As $\v(F_e)^2\leq 2$, $\v(F_e)^2=0$ or $2$.  In the latter case, we must have $2\mid\pi^*\v$, $\ext^1(F_e,F_g)=2$, and $\codim\Sing(\mSs)=1$, so $\mSs$ is not even normal.  Otherwise, we have $\v(F_e)^2=0$, so we see that $N\geq\ext^1(F_e,F_g)=4$ and $\codim\Sing(\mSs)=3>2$.  If $\v^2=2$, then $\v(F_e)^2\leq 1$.  So $\v(F_e)^2=0$.  In this case we must have $\ext^1(F_e,F_g)=2$ and $\codim\Sing(\mSs)=1$, so $\mSs$ is again not normal.

\section*{Order Three} We must consider $\v^2=8,6,4,2$, and in this case \eqref{term 1} and \eqref{term 3} give $$\ext^1(F_e,F_g)=\frac{\v^2-\v(F_e)^2}{2}.$$  For $\v^2=8$, $\v(F_e)^2\leq 2$.  As noted above, $\codim\Sing(\mSs)\geq 2$ in any case.  If $\v(F_e)^2=2$, then $\ext^1(F_e,F_g)=3$, so we can only guarantee that $\mSs$ has canonical singularities as in \cite{Yamada}.  Furthermore, this case occurs when $c_1(F_e)^2-c_1(F_e).c_1(F_g)=-1$.  Otherwise, $\v(F_e)^2=0$ and $\ext^1(F_e,F_g)=4$, so $N\geq 4$ and $\mSs$ has terminal singularities.  This occurs when $c_1(F_e)^2-c_1(F_e).c_1(F_g)=-4$.  

For $\v^2=6$, $\v(F_e)^2\leq 2$.  If equality is achieved then $3\mid\pi^*\v$, so as in Proposition \ref{local singularities} $\ext^1(F_e,F_g)=2$, and we can only guarantee that $\mSs$ has torsion $\omega_{\mSs}$.  Otherwise, $\v(F_e)^2=0$ and $\ext^1(F_e,F_g)=3$ so that $\mSs$ has at worst canonical singularities.  This occurs if $c_1(F_e)^2-c_1(F_e).c_1(F_g)=-3$.  When $\v^2=4$, $\v(F_e)^2=0$, so $\ext^1(F_e,F_g)=2$, but $\codim\Sing(\mSs)=3$, so $\mSs$ is at least normal with torsion canonical divisor.  When $\v^2=2$, $\v(F_e)^2=0$, so both $\ext^1(F_e,F_g)$ and $\codim\Sing(\mSs)$ are equal to 1, and $\mSs$ is not even normal.

In the next two examples, $\o$ is large enough to be able to push Yamada's techniques even further.  Taking two elements $g_1,g_2\in G'\backslash\{e\}$ with $g_1 g_2\neq e$ and $\alpha_i\in\im(q_e^*\circ {i_{g_i}}_*)\subset\Ext^1(\pi^*E,\pi^*E)$, we get $\bar{\alpha_i}\in\Ext^1(E,E)$ such that $\pi^*\bar{\alpha_i}=\sum_{g\in G'} g(\alpha_i)$, and we must have $0=F_f(\alpha_1,\alpha_2)$.  Indeed, the proof of \cite[Lemma 2.12]{Yamada} shows that $$F_f(\alpha_1,\alpha_2)=(1-\frac{1}{\lambda_1})\Tr(f\circ\bar{\alpha_1}\circ\bar{\alpha_2})=(1-\lambda_2)\Tr(f\circ\bar{\alpha_1}\circ\bar{\alpha_2}),$$ where $\lambda_i=\frac{g_i(a_1)}{a_1}\in\C=H^0(K_X)=H^0(\OO_X)^{G'}$.  If $\Tr(f\circ\bar{\alpha_1}\circ\bar{\alpha_2})\neq 0$, then we must have $\lambda_1\lambda_2=1$, i.e. $g_1g_2(a_1)=a_1$, so $g_1g_2=e$ by \cite[Lemma 2.8]{Yamada}, which is contrary to our assumption.  Thus even if no individual $\ext^1(F_e,F_g)\geq 4$, as long as the sum of $\ext^1(F_e,F_{g_1})>0,\ext^1(F_e,F_{g_2})>0$ with $g_1g_2\neq e$ is at least 4, Yamada's argument gives $N\geq 4$.

\section*{Order Four}\label{order 4 singular} 
Now \eqref{term 1} and \eqref{term 2} give $$2\ext^1(F_e,F_g)+\ext^1(F_e,F_{g^2})=\v^2-\v(F_e)^2,$$ and the remaining cases to consider are $\v^2=10,8,6,4,2$.  For $\v^2=10$, we have $\v(F_e)^2\leq 2$.  Moreover, since $c_1(F_e)\neq c_1(F_g)$ we must have strict inequality in \cite[(12)]{Yamada}, so $\ext^1(F_e,F_g)>\v(F_e)^2$, and $\ext^1(F_e,F_{g^2})\geq \v(F_e)^2$.  If $\v(F_e)^2=2$, then we must have $\ext^1(F_e,F_g)=3,\ext^1(F_e,F_{g^2})=2$.  As the sum is larger than 4, the argument preceeding this example shows that $\mSs$ has terminal singularities.  Otherwise $\v(F_e)^2=0$, and it is easy to see that at least one of $\ext^1(F_e,F_g)$ or $\ext^1(F_e,F_{g^2})$ is at least 4.  Indeed, if both were at most $3$, then $2\ext^1(F_e,F_g)+\ext^1(F_e,F_{g^2})\leq 9\neq 10$.  So $\mSs$ has terminal singularities.

For $\v^2=8$, again $\v(F_e)^2\leq 2$.  If $\v(F_e)^2=2$, then $4\mid\pi^*\v$, and we must have $\ext^1(F_e,F_g)=\ext^1(F_e,F_{g^2})=2$, which sum to 4, giving $N\geq 4$ and thus terminal singularities.  Otherwise $\v(F_e)^2=0$, so $\ext^1(F_e,F_g)>0,\ext^1(F_e,F_{g^2}\geq 0$.  The only case in which neither of these is at least 4 is when $\ext^1(F_e,F_g)=3,\ext^1(F_e,F_{g^2})=2$, in which case their sum is again at least 4.  Thus $\mSs$ has at worst terminal singularities in either case.

If $2\leq \v^2\leq 6$, then $\v(F_e)^2=0$.  The codimension condition on $\Sing(\mSs)$ is always satisfied except when $\v^2=2$, in which case $\ext^1(F_e,F_g)=1,\ext^1(F_e,F_{g^2})=0=\v(F_e)^2$, $\mSs$ is not normal.  Furthermore, we only have $\Sing(\mSs)\neq\varnothing$ if $c_1(F_{g^2})=c_1(F_e)$ but $c_1(F_e)^2-c_1(F_e).g^*c_1(F_e)=-1$.  When $\v^2=6$, the only solutions are $$(\ext^1(F_e,F_g),\ext^1(F_e,F_{g^2}))=(1,4),(2,2),(3,0).$$
The first two solutions guarantee that $N\geq 4$ and $\mSs$ has at worst terminal singularities.  The final solution guarantees that $\mSs$  has at worst canonical singularities and requires that $c_1(F_e)^2-c_1(F_e).g^*c_1(F_e)=-3$ while $(g^2)^*c_1(F_e)=c_1(F_e)$.  When $\v^2=4$, the only two solutions are $(\ext^1(F_e,F_g),\ext^1(F_e,F_{g^2}))=(1,2),(2,0)$.  The first solution guaranties that $\mSs$ has at worst canonical singularities, while the latter, which requires that $(g^2)^*c_1(F_e)=c_1(F_e)$ and $c_1(F_e)^2-c_1(F_e).g^*c_1(F_e)=-2$, only guarantees that $\mSs$ is normal and Gorenstein with torsion canonical divisor.

\section*{Order Six}\label{order 6 singular}  
We must consider $2\leq \v^2\leq 16$, and \eqref{term 1} and \eqref{term 2} give $$2\ext^1(F_e,F_g)+2\ext^1(F_e,F_{g^2})+\ext^1(F_e,F_{g^3})=\v^2-\v(F_e)^2.$$  Furthermore, $\codim\Sing(\mSs)\geq 2$ as long as $\v^2>2$.  For $\v^2=16$, $\v(F_e)^2\leq 2$, and if $\v(F_e)^2=2$, then as $\ext^1(F_e,F_{g^i})\geq \v(F_e)^2$ (with strict inequality for $i=1$), the only solution which does not give some $i$ for which $\ext^1(F_e,F_{g^i})\geq 4$ is $\ext^1(F_e,F_g)=\ext^1(F_e,F_{g^2})=3,\ext^1(F_e,F_{g^3})=2$, which sums to at least 4, so $\mSs$ still has at worst terminal singularities.  If $\v(F_e)^2=0$, then it is easily seen that $\ext^1(F_e,F_{g^i})\geq 4$ for some $i\neq 0$.  If $\v^2=14$ and $\v(F_e)^2=2$, the only numerical solutions are $\ext^1(F_e,F_g)=3,\ext^1(F_e,F_{g^2})=\ext^1(F_e,F_{g^3})=2=\v(F_e)^2$, which requires that $(g^2)^*c_1(F_e)=c_1(F_e),(g^3)^*c_1(F_e)=c_1(F_e)$.  But as $\gcd(2,3)=1$, we would then have $g^c_1(F_e)=c_1(F_e)$, a contradiction to $\ext^1(F_e,F_g)=3$, so we must have $\v(F_e)^2=0$.  The only solution which does not give some $i$ such that $\ext^1(F_e,F_{g^i})\geq 4$ is when $\ext^1(F_e,F_g)=\ext^1(F_e,F_{g^2})=3,\ext^1(F_e,F_{g^3})=2$, in which case the sum is still at least 4 so the singularities are still terminal.

When $\v^2=12$, equality can be achieved in \eqref{singular}.  But then $6\mid\pi^*\v$ and $2=\ext^1(F_e,F_{g^i})$ for all $i$, giving a sum of $6>4$ and thus terminal singularities.  So $\v(F_e)^2=0$,  and it can be seen that the sum of $\ext^1(F_e,F_{g^i})$ for $i=1,2,3$ is always greater than 4, so $\mSs$ has at worst terminal singularities.

If $2\leq \v^2\leq 10$, then $\v(F_e)^2=0$.  Using the same analysis as above, it is easily checked that the singularities are always at worst terminal if $\v^2=10,8$.  When $\v^2=6$,  the singularities are at worst terminal unless $(g^3)^*c_1(F_e)=c_1(F_e)$ but $(g^i)^*c_1(F_e)\neq c_1(F_e)$ for $i=1,2$, in which case $(\ext^1(F_e,F_g),\ext^1(F_e,F_{g^2}))=(1,2)$ or $(2,1)$, so the singularities are at worst canonical.  

If $\v^2=4$, then $\mSs$ has at worst canonical singularities if $(g^2)^*c_1(F_e)=c_1(F_e)$, but we can only guarantee that it is normal and Gorenstein with torsion canonical divisor if $(g^3)^*c_1(F_e)=c_1(F_e)$.

When $\v^2=2$, the only numerical solution is $(\ext^1(F_e,F_g),\ext^1(F_e,F_{g^2}),\ext^1(F_e,F_{g^3}))=(1,0,0)$, in which case $(g^2)^*c_1(F_e)=c_1(F_e)=(g^3)^*c_1(F_e)$.  But this would force $g^*c_1(F_e)=c_1(F_e)$, contradicting $\ext^1(F_e,F_g)=1$, so $\mSs$ is smooth when $\v^2=2$.

\bibliographystyle{plain}
\bibliography{all}

\end{document}